\documentclass[11pt]{article}
\usepackage{amsfonts}
\usepackage{CJK}
\usepackage{graphics}
\usepackage{indentfirst}
\usepackage{cite}
\usepackage{latexsym}
\usepackage{amsmath}
\usepackage{amsthm}
\usepackage{amssymb}
\usepackage[dvips]{epsfig}
\usepackage{amscd}

\usepackage{color}

\newtheorem{theorem}{Theorem}[section]
\newtheorem{remark}{Remark}[section]
\newtheorem{definition}{Definition}[section]
\newtheorem{lemma}[theorem]{Lemma}
\newtheorem{pro}{Proposition}[section]
\newtheorem{cor}[theorem]{Corollary}

\renewcommand{\div}{ {\rm div }  }

\newcommand{\na}{\nabla }

\newcommand{\pa}{\partial}
\renewcommand{\r}{\mathbb{R}}

\newcommand{\dis}{\displaystyle}
\newcommand{\ia}{\int_0^T}

\newcommand{\bt}{{\hat\theta}}
\newcommand{\bl}{\begin{lemma}}
\newcommand{\el}{\end{lemma}}
\newcommand{\et}{\end{theorem}}
\newcommand{\ga}{\gamma}

\newcommand{\curl}{{\rm curl} }
\newcommand{\te}{\theta}

\newcommand{\al}{\alpha}
\newcommand{\de}{\delta}
\newcommand{\ve}{\varepsilon}
\newcommand{\la}{\label}

\newcommand{\p}{p(\rho)  }

\newcommand{\ka}{\kappa}

\newcommand{\bn}{\begin{eqnarray}}
\newcommand{\en}{\end{eqnarray}}
\newcommand{\bnn}{\begin{eqnarray*}}
\newcommand{\enn}{\end{eqnarray*}}

\newcommand{\bnnn}{\begin{eqnarray*}}
\newcommand{\ennn}{\end{eqnarray*}}

\newcommand{\ba}{\begin{aligned}}
\newcommand{\ea}{\end{aligned}}
\newcommand{\be}{\begin{equation}}
\newcommand{\ee}{\end{equation}}
\def\O{\Omega}
\def\p{\partial}
\def\norm[#1]#2{\|#2\|_{#1}}

\newcommand{\ep}{\varepsilon}
\newcommand{\n}{\rho}
\newcommand{\si}{\sigma}

\def\la{\label}

\def\na{\nabla}
\def\on{\hat\rho}

\def\tn{1}

\def\xl{\left}
\def\xr{\right}
\def\bp{\overline{P}}

\makeatletter   
\@addtoreset{equation}{section}
\makeatother       

\title{Global   Existence of Classical Solutions  to      Full Compressible Navier-Stokes System with Large Oscillations and Vacuum in 3D Bounded Domains}

\author{Jing L{\small I}$^{a,b,c} $, Boqiang L{\small \"U}$^{a} $, Xue W{\small ANG}$^{b}$  \thanks{email:  ajingli@gmail.com (J. Li), lvbq86@163.com (B.   L\"u), xuewa@amss.ac.cn (X.  Wang) } \\
{\normalsize a. Department of Mathematics }\\ {\normalsize \& Institute of Mathematics and Interdisciplinary Sciences,}\\ {\normalsize  Nanchang University, Nanchang 330031, P. R. China;}
 \\ {\normalsize b.  School of Mathematical Sciences,}\\
{\normalsize  University of Chinese Academy of Sciences, Beijing 100049, P. R. China;}\\
{\normalsize c. Institute of Applied Mathematics, AMSS,} \\ {\normalsize \&   Hua Loo-Keng Key Laboratory of Mathematics,}\\
{\normalsize  Chinese Academy of Sciences,    Beijing 100190, P. R. China}}

\date{}
\begin{document}
\maketitle

\begin{abstract}
The full compressible Navier-Stokes system describing the motion of a viscous, compressible, heat-conductive, and Newtonian polytropic fluid is studied in a three-dimensional   simply connected bounded domain with smooth boundary having a finite number of two-dimensional connected components.
For the initial-boundary-value problem with slip boundary conditions on the velocity and Neumann boundary  one on the temperature,  the global existence of classical and weak solutions  which are of small energy but possibly large oscillations is established. In particular, both the density and temperature are allowed to vanish initially. Finally, the exponential stability of the density, velocity, and temperature is also obtained.   Moreover,  it is   shown that for the classical solutions, the oscillation of the density will grow unboundedly in the long run with an exponential  rate provided  vacuum appears  (even at a point) initially.
This is the first result concerning the global existence of classical   solutions to the full compressible Navier-Stokes equations with   vacuum in general three-dimensional bounded smooth domains.
\end{abstract}


\section{Introduction}
The motion of a compressible viscous, heat-conductive, and Newtonian polytropic fluid occupying a spatial domain $\Omega \subset \mathbb{R}^{3}$ is governed by the following full compressible Navier-Stokes system:
\be \la{a0}
\begin{cases}
 	\n_t+{\rm div} (\n u)=0,\\
 	(\n u)_t+{\rm div}(\n u\otimes u)+\na P=\div \mathbb{S},\\
 	(\n E)_t+\div (\n Eu+Pu)=\div (\kappa \na \te)+\div (\mathbb{S}u),
\end{cases}\ee
where   $\mathbb{S}$ and $E$ are respectively the viscous stress tensor and the total energy given by
$$\mathbb{S}=2\mu\mathfrak{D}(u) +\lambda \div  u \mathbb{I}_3,~~E=e+\frac{1}{2}|u|^2,$$
with $\mathfrak{D}(u)  = (\nabla u + (\nabla u)^{\rm tr})/2$ and $\mathbb{I}_3$  denoting the deformation tensor and the $3\times3$ identity matrix respectively.
Here, $t\ge 0$ is time, $x\in \Omega$ is the spatial coordinate, and $\n$, $u=\left(u^1,u^2,u^3\right)^{\rm tr},$ $e$, $P,$ and $\te $ represent respectively the fluid density, velocity, specific internal energy, pressure, and  absolute temperature.
The viscosity coefficients $\mu$ and $\lambda$  are constants satisfying the physical restrictions:
\be\la{h3} \mu>0,\quad 2 \mu + 3\lambda\ge 0.\ee
The heat-conductivity coefficient $\ka$ is a positive constant.
We consider the ideal  polytropic fluids so that $P$ and $e$ are given by the state equations:
\be \la{pg}
P(\n,e)=(\ga-1)\n e=R\n \te, \quad e=\frac{R\theta}{\ga-1},
\ee
where $\ga>1$  is the adiabatic constant  and $ R$ is a positive constant.

Let $\Omega \subset \r^3 $ be a simply connected bounded domain.
Note that for the classical solutions, the system \eqref{a0} can be rewritten as
\be \la{a1}
\begin{cases}
	\n_t+{\rm div} (\n u)=0,\\
	\n (u_t+  u\cdot \na u )=\mu\Delta u+(\mu+\lambda)\na({\rm div}u)-\na P,\\
	\frac{R}{\ga-1}\n ( \te_t+u\cdot\na \te)=\ka\Delta\te-P\div u+\lambda (\div u)^2+2\mu |\mathfrak{D}(u)|^2.
\end{cases}
\ee
We consider the system \eqref{a1}  subjected  to the given initial data
\be \la{h2}
(\rho,\n u,\n \te)(x,{t=0})=(\rho_0,\n_0 u_0, \n_0\te_0)(x), \quad x\in \Omega,
\ee
and boundary conditions
\be \la{h1}
u\cdot n=0,~~\mbox{curl} u\times n=0,~~\na \theta\cdot n=0~~~~\text{on}~\p\Omega \times(0,T),
\ee
where $n =(n^1, n^2, n^3)^{\rm{tr}}$ is the unit outward normal vector on $\partial\Omega$.

There is a lot  of literature  on the global existence and large time
behavior of solutions to (\ref{a0}). The one-dimensional problem
with strictly positive initial density and temperature has been
studied extensively by many people (see \cite{kazh01,Kaz,akm} and
the references therein). For the multi-dimensional case,  the local existence and uniqueness of classical solutions are known in \cite{Na,se1}  in the absence of vacuum.  
The global classical solutions were
first obtained by Matsumura-Nishida \cite{M1} for initial data close
to a non-vacuum equilibrium in some Sobolev space $H^s.$
Later, Hoff \cite{Hof1} studied the
global weak solutions with strictly positive initial density and
temperature for discontinuous initial data. On the other hand, in
the presence of vacuum, this issue becomes much more complicated.
Concerning viscous compressible fluids in a barotropic regime, where
the state of these fluids at each instant $t>0 $ is completely
determined by the density $\n=\n(x,t)$ and the velocity $u=u(x,t),$
the pressure $P$ being an explicit function of the density,    the
major breakthrough is due to Lions \cite{L1} (see also  Feireisl
\cite{F1, feireisl1}), where he obtained global existence of weak
solutions, defined as solutions with finite energy,  when the
pressure $P$ satisfies $P(\n)=a\n^\ga (a>0,\ga>1)$  with suitably large  $\ga.$
The main restriction on initial data is that the initial energy  is
finite, so that the density vanishes at far fields, or even has
compact support. Recently,  Huang-Li-Xin \cite{hulx} and Li-Xin \cite{2dlx} established the global well-posedness
of classical solutions to the Cauchy problem for the 3D and 2D barotropic compressible Navier-Stokes equations in whole space with smooth initial data that are of small energy but
possibly large oscillations, in particular, the initial density is allowed to vanish.
More recently, for slip boundary condition  in bounded domains,
Cai-Li \cite{C-L}   obtained the global classical solutions with initial vacuum, provided that the initial energy is suitably small.

Compared with the barotropic flows, it seems much more difficult and complicated to study the global well-posedness of solutions to full compressible Navier-Stokes system (\ref{a0}) with   vacuum, where some additional difficulties arise, such as the degeneracy of both momentum and energy equations, the strong coupling between the velocity and temperature,  et al. For specific pressure laws   excluding the perfect gas
equation of state, the question of existence of so-called
``variational" solutions in dimension $d\ge 2$ has been recently
 addressed
in \cite{feireisl,feireisl1}, where the temperature equation is
satisfied only as an inequality  which justifies the notion of
variational solutions. 
 Moreover, for   a very particular form of the viscosity
coefficients depending on the density, Bresch-Desjardins
\cite{bd} obtained global stability of weak solutions.
For   the global well-posedness of classical solutions to the full compressible Navier-Stokes system \eqref{a0}, it  is shown in  {Xin} \cite{X1}  that  there is no solution in $C^{1}\left([0, \infty), H^{s}\left(\mathbb{R}^{d}\right)\right)$ for large $s$ to the Cauchy problem for the full compressible Navier-Stokes
system without heat conduction  provided  that the initial density has compact support.  See also the recent generalizations to the
case   for non-compact but rapidly decreasing at far field initial densities (\cite{R}).
Recently, Huang-Li \cite{H-L} established the global existence and uniqueness for the classical solutions to the 3D Cauchy problem with interior vacuum provided the initial energy is small enough.  Later,  Wen-Zhu  \cite{W-C} obtained the global existence and uniqueness of the classical solutions  for vanishing far-field density under the assumption that the initial mass is sufficiently small or both viscosity and heat-conductivity coefficients are large enough.  It should be mentioned here that   the results of \cite{H-L,W-C} hold  only for the Cauchy problem. However, 
 the global existence   of classical solutions or even weak ones with vacuum to multi-dimensional   full compressible Navier-Stokes system \eqref{a0} in general bounded domains remains completely  open except for spherically or cylindrically symmetric
initial data (see  \cite{wenzhu1,wenzhu2}). In fact, one of the aims of  this paper is to study the global well-posedness of classical solutions to full compressible Navier-Stokes system \eqref{a0} in general bounded domains.

Before stating the main results, we explain the notations and conventions used throughout this paper.  We denote
\bnn\int fdx\triangleq\int_{\Omega}fdx,\enn
and \bnn \overline{f}\triangleq\frac{1}{|\O|}\int_{\Omega}fdx,\enn
which is the average of a function $f$ over $\Omega$.
For $1\le p\le \infty $ and integer $k\ge 0,$  we adopt the simplified notations for Sobolev spaces as
follows:
\be\ba\notag \begin {cases}
L^p=L^p(\Omega),\quad W^{k,p}=W^{k,p}(\Omega),\quad H^k=W^{k,2}(\O),\\
H_\omega^i= \left.\left\{f\in H^i \right | f\cdot n=0,\,\curl f\times n=0 \rm \,\,{on}\,\, \p\O\right\}\,(i=1,2).
\end{cases}\ea\ee

Without loss of generality, we assume that
\be\la{m}\overline{\rho_0}=\frac{1}{|\O|}\int \rho_0 dx=1.\ee
We then define  the initial energy  $C_0$  as follows:
\be\la{e}\ba  C_0\triangleq &\frac{1}{2}\int\n_0
|u_0|^2dx+R \int  \left( 1+\n_0\log {\n_0}-\n_0 \right)dx\\
&+\frac{R}{\ga-1}\int  \n_0\left(\te_0- \log {\te_0} -1
\right) dx .\ea\ee

The first main result in this paper can be stated  as follows:

\begin{theorem}\la{th1}
Let $\O\subset\r^3$ be a simply connected bounded smooth domain, whose boundary $\p\O$ has a finite number of 2-dimensional connected components. For  given numbers $M>0$ (not necessarily
small), $q\in (3,6),$
  $\on> 2,$ and $\bt>1,$
suppose that the initial data $(\n_0,u_0,\te_0)$ satisfies \be\ba
\la{co3} \n_0\in  W^{2,q}, \quad  u_0 \in H^2_\omega,\quad
\te_0\in\left.\left\{f\in H^1 \right| \na f\cdot n=0 \,\,\rm {on}\,\, \p\O\right\}, \ea\ee
\be \la{co4} 0\le\inf\rho_0\le\sup\rho_0< \hat{\rho},\quad 0  \le\inf\te_0\le\sup\te_0\le \bt, \quad \|\na u_0\|_{L^2} \le M,
   \ee
   and the compatibility condition
\be
\la{co2}-\mu \Delta u_0-(\mu+\lambda)\na\div u_0+R\na (\n_0\te_0)
=\sqrt{\n_0} g,\ee
with  $g\in L^2. $ Then there exists a positive constant $\ve$
depending only
 on $\mu,$ $\lambda,$ $\ka,$ $ R,$ $ \ga,$  $\on,$ $\bt$, $\O$, and $M$ such that if
 \be
 \la{co14} C_0\le\ve,
   \ee  the   problem  (\ref{a1})--(\ref{h1})
admits a unique global classical solution $(\rho,u,\te)$ in
   $\Omega\times(0,\infty)$ satisfying
  \be\la{h8}
  0\le\rho(x,t)\le 2\hat{\rho},\quad \te(x,t)\ge 0,\quad x\in \Omega,\,~~ t\ge 0,
  \ee
 and \be
   \la{h9}\begin{cases}
   \rho\in C([0,T];W^{2,q}),\\ 
   u \in C([0,T];W^{1,\tilde p})\cap L^\infty(0,T;H^2)\cap   L^\infty(\tau,T; W^{3,q}),  \\
   \te\in   L^\infty(\tau,T;H^4)\cap  C([\tau,T];  W^{3,\tilde p}), \\
  u_t \in L^2(0,T;H^1)\cap
L^{\infty}(\tau,T;H^2)\cap H^1(\tau,T;H^1), \\
  \te_t \in
L^{\infty}(\tau,T;H^2)\cap H^1(\tau,T;H^1),\end{cases} \ee for any $0<\tau<T<\infty $ and $\tilde p\in [1,6).$
Moreover, for any $p\in  [1 ,\infty)$ and $r\in [1,6]$, there exist positive constants $C$, $\al_0$, and $\te_\infty$ depending only  on $\mu,$ $\lambda,$ $\ka,$ $ R,$ $ \ga,$  $\on,$ $\bt$, $\O$, $p$, $r$, and $M$  such that for any $t\geq 1,$
\be\la{h11}
\|\n-1\|_{L^p}+\|u\|_{W^{1,r}}^2+\|\te-\te_\infty\|_{H^2}^2\le C e^{-\al_0 t}.
\ee

\end{theorem}

The next result of this paper concerns  weak solutions whose   definition is as follows.
\begin{definition}\la{def} We say that $(\n,u,E=\frac{1}{2}|u|^2+\frac{R}{\ga-1}\te)$ is a weak solution to Cauchy problem (\ref{a0}) (\ref{h2}) (\ref{h1}) provided that
$$\n\in L^\infty_{\rm loc}([0,\infty);L^\infty(\O)),\quad u ,\te\in L^2_{\rm loc} ([0,\infty); H^1(\O)),$$
and that for all test functions $\psi\in\mathcal{D}(\O\times(-\infty,\infty)),$
\be \la{def1}
\int_{\O}\n_0\psi(\cdot,0)dx+\int_0^\infty\int_{\O}\left(\n\psi_t+\n u\cdot\na\psi\right) dxdt=0,\ee
\be\la{def2}\ba&
\int_{\O}\n_0u^j_0\psi(\cdot,0)dx+\int_0^\infty\int_{\O}\left(\n u^j\psi_t
   +\n u^ju\cdot\na\psi+P(\n,\te)\psi_{x_j}\right)dxdt\\
&-\int_0^\infty\int_{\O}\left(  \mu\na u^j\cdot\na\psi+(\mu+\lambda)(\div u)\psi_{x_j}\right) dxdt=0,\quad
j=1,2,3, \ea\ee

\be \la{def3}\ba
&\int_{\O}\left(\frac{1}{2}\n_0|u_0|^2 +\frac{R}{\ga-1}\n_0\te_0 \right)\psi(\cdot,0)dx\\
&+\int_0^\infty\int_{\O}\left(\n E\psi_t+ (\n E  +P)u\cdot\na\psi\right) dxdt\\
&-\int_0^\infty\int_{\O}\left(\ka\na\te+\frac{1}{2}\mu \na(|u|^2)
 +\mu  u\cdot \na u+\lambda\div u u\right)\cdot\na \psi dxdt=0.
\ea\ee
\end{definition}

Then we state our second main result as follows:
\begin{theorem}\la{th2}  Under the conditions of Theorem \ref{th1} except (\ref{co2}), where the condition (\ref{co3}) is replaced by
\be\la{dt7}u_0\in H^1_\omega,\ee
assume  further that $C_0$  as in (\ref{e}) satisfies (\ref{co14}) with $\varepsilon$ as in Theorem \ref{th1}. Then there exists a global weak solution $(\rho,u,E=\frac{1}{2}|u|^2+\frac{R}{\ga-1}\te)$ to the problem  (\ref{a0}) (\ref{h2}) (\ref{h1})
    satisfying
\be\la{hq1}
\n\in C([0,\infty);L^p), \quad (\n u,\,\n |u|^2,\,\n\te)\in C([0,\infty);H^{-1}) ,
\ee
\be\la{hq2}u\in L^\infty(0,\infty;H^1)\cap C((0,\infty);L^2 )  ,\quad \te\in C((0,\infty); W^{1,\tilde p}),\ee
\be \la{hq3}
u(\cdot,t),\,\,\curl u(\cdot,t),\,\,((2\mu+\lambda)\div u-(P-\overline P))(\cdot,t),  \,\,\na\te(\cdot,t)\in H^1,\quad t>0,\ee
\be\la{hq4}\rho\in [0,2\hat{\rho}] \quad \mbox{ \rm a.e.}, \quad \te\ge 0 \quad\mbox{ \rm a.e.},\ee
 and the exponential decay property (\ref{h11})  with $p\in  [1 ,\infty)$, $\tilde p\in [1,6)$, and $r\in[1,6]$. In addition, there exists some  positive constant $C$ depending  only on $ \mu,$ $\lambda,$ $ \ka,$ $ R,$ $ \ga,$  $\on$, $\bt,$ $\O$, and $M $ such that, for  $\si(t)\triangleq\min\{1,t\},$  the following estimates hold
  \be\la{hq5}\ba  \sup_{t\in (0,\infty)} \| u \|_{H^1}^2  +\int_0^\infty\int\left| (\n u)_t+{\rm div}(\n u\otimes u)\right|^2dxdt\le C,\ea\ee
  \be\la{hq7}\ba &\sup_{t\in (0,\infty)} \int\left((\n-1)^2+\n |u|^2+\n(R\te-\overline P)^2\right)dx  \\&\quad   +\int_0^\infty\left( \|\na u\|^2_{L^2}+ \|\na \te\|_{L^2}^2 \right)dt\le CC_0^{1/4}, \ea\ee
\be\la{hq8}\ba
&\sup_{t\in (0,\infty)}\left(\si\|\na u \|^2_{L^6}+\si^2\|\te\|^2_{H^2}\right) \\
&+\int_0^\infty\left( \si\|u_t\|^2_{L^2}+\si^2\|\na \dot u\|_{L^2}^2+\si^2\|\te_t\|^2_{H^1}\right)dt\le C. \ea\ee
Moreover, $(\n,u,\te)$ satisfies (\ref{a1})$_3$ in the weak form, that is, for any test function  $\psi\in
\mathcal{D}(\O\times(-\infty,\infty)),$
\be\la{vu019}\ba &\frac{R}{\ga-1} \int\n_0 \te_0\psi(\cdot,0) dx+\frac{R}{\ga-1}\int_0^\infty\int\n \te \left(\psi_t+u \cdot\na\psi
\right)dxdt\\ & =  \ka \int_0^\infty\int \na\te \cdot\na\psi dxdt+R\int_0^\infty\int  \n\te  \div u \psi dxdt\\&\quad -\int_0^\infty\int \left(\lambda(\div u)^2+2\mu   |\mathfrak{D}(u)|^2  \right) \psi dxdt.\ea\ee

\end{theorem}

Next, as a direct application of \eqref{h11}, the following Corollary \ref{th3},  whose proof is similar to that of     \cite[Theorem 1.2]{C-L}, shows that the oscillation of the density will grow unboundedly in the long run with an exponential  rate provided  vacuum appears  (even at a point) initially.

\begin{cor} \la{th3}
In addition to the conditions  of Theorem \ref{th1}, assume further that there exists some point $x_0\in\Omega$ such that $\rho_0(x_0)=0.$ Then for any $\hat p>3$, there exists some positive constant $C$ depending on $\mu$, $\lambda$, $\ka,$ $ R,$ $ \ga,$  $\on,$ $\bt$, $\O$, $\hat p$ and $M$ such that the unique global classical solution $(\rho,u,\te)$ to the problem (\ref{a1})--(\ref{h1}) obtained in Theorem \ref{th1} satisfies that for any $t\geq 1$,
\be\notag\|\nabla \rho(\cdot,t)\|_{L^{\hat p} }\geq Ce^{Ct}.\ee
\end{cor}

A few remarks are in order:

\begin{remark}
It is easy to deduce from (\ref{h9}) and the Sobolev imbedding  theorem that for any $0<\tau<T<\infty,$
\be \la{hk1}
(\n,\, \na\n,\,u) \in C(\overline{\Omega} \times[0,T] ),\quad   (\te,\,\na\te,\,\na^2\te)  \in C(\overline{\Omega} \times(0,T] ),\ee
and
\be \ba\la{ono}
(\na u,\,\na^2 u) \in  C( [\tau,T];L^2 )\cap L^\infty(\tau,T;  W^{1,q})\hookrightarrow  C(\overline{\Omega} \times[\tau,T] ),
\ea \ee
which together with  (\ref{a1})$_1$  and (\ref{hk1}) shows
\be  \la{hk3} \n_t\in   C(\overline{\Omega}\times[\tau,T] ).\ee
Analogously, we have
\be\notag(u_t,\te_t)\in C(\overline{\Omega} \times[\tau,T] ),\ee
which along with (\ref{hk1})--(\ref{hk3}) arrives at
 the solution $(\n,u,\te)$ obtained in Theorem \ref{th1} is a classical one to  the   problem (\ref{a1})--(\ref{h1}) in $\Omega\times (0,\infty).$
\end{remark}

\begin{remark} It seems that
Theorem \ref{th1}, which extends the global existence result of the  barotropic flows studied in \cite{C-L} to the full compressible Navier-Stokes system,  is the first result concerning the global existence of classical solutions with initial vacuum to  (\ref{a0}) in general bounded domains. Although its energy is small, the oscillations could be arbitrarily large. 
 \end{remark}

\begin{remark}
	To obtain the global existence and uniqueness of classical solutions with vacuum, we only need the compatibility condition on the velocity (\ref{co2}) as in \cite{lxz}, which is much weaker than those in \cite{choe1, H-L,W-C} where not only (\ref{co2}) but also the following compatibility condition on the temperature
	\be\la{co1} \ka\Delta \te_0+\frac{\mu}{2}|\na u_0+(\na u_0)^{\rm tr}|^2+\lambda (\div u_0)^2=\sqrt{\n_0}g_1, \quad g_1 \in L^2\ee
	is needed. This reveals that the compatibility condition on the temperature (\ref{co1}) is not necessary for establishing the classical solutions with vacuum to the full Navier-Stokes equations, which is just the same as the barotropic case \cite{hulx, 2dlx}.
	\end{remark}
\begin{remark}
It should be  mentioned here that the boundary condition for velocity $u$:
	\be\la{bb}
	u\cdot n=0,\,\,\curl u \times n=0 \,\,\text {on}\,\, \p\O,
	\ee
	is a special case of the following general Navier-type slip condition (see Navier \cite{Nclm1}) 
	\be\notag
	u \cdot n = 0, \,\,(2\mathfrak{D}(u)\,n+ \vartheta u)_{tan}=0 \,\,\,\text{on}\,\,\, \partial\Omega,
	\ee which as indicated by \cite[Remark 1.1]{C-L},   is in fact  a particular case of   the following slip boundary one:
	\bn\label{Ns}
	u\cdot n=0,\,\,\curl u \times n=-Au \,\,\text {on}\,\, \p\O,
	\en
	where $\vartheta$ is a scalar friction function, the symbol $v_{tan}$ represents the projection of tangent plane of the vector $v$ on $\partial\Omega$, and $A = A(x)$ is a given $3\times3$ symmetric matrix defined on $\p\O$.	
	Indeed, our result still holds for more general slip boundary condition (\ref{Ns}) with $A$ being semi-positive and regular enough. The proof is similar to \cite{C-L} and omitted here.
\end{remark}

We now comment on the analysis of this paper.
We mainly take the strategy that we first extend the
standard local classical solutions with strictly positive initial density (see Lemma \ref{th0}) globally in time just under the condition that the initial energy is suitably small
(see Proposition \ref{pro2}), then let the lower bound of the initial density go to zero. To
do so, one needs to establish global a priori estimates, which are independent of
the lower bound of the density, on smooth solutions to  (\ref{a1})--(\ref{h1}) in suitable
higher norms.
As indicated in \cite{H-L,h101}, the key issue  is to obtain the time-independent upper bound of the density.

The first main difficulty arises in deriving the basic energy estimate, which indeed  is obtained directly for the Cauchy problem \cite{H-L}. 
However, in our case,  the basic energy equality reads:
\be\ba\la{11a}
&E'(t)+\int \left( \frac{\lambda(\div u)^2+2\mu |\mathfrak{D}(u)|^2}{{\te}}+\ka \frac{|\na \te|^2}{\te^2} \right)dx \\
&=- \mu \int \left(|\curl u|^2+2(\div u)^2 - 2 |\mathfrak{D}(u)|^2\right)dx,
\ea\ee where the basic energy $E(t)$ is defined by
\be \label{enet}E(t) \triangleq \int  \left( \frac{1}{2}\n |u|^2+R(1+\n\log
\n-\n)+\frac{R}{\ga-1}\n(\te-\log \te-1)\right)dx. \ee
Note that the right-hand term in \eqref{11a}
    is   sign-undetermined   due to the slip boundary condition (\ref{bb}), thus it seems difficult to obtain directly the usual standard energy estimate  \be \la{edwq1} E(t) \le C C_0,\ee   
where the smallness (with the same order of initial energy) of basic energy  plays a key role in the whole analysis of the global existence of classical solutions with vacuum not only for Cauchy problem/IBVP of barotropic flows \cite{2dlx, hulx, C-L} but also for Cauchy problem of  full compressible Navier-Stokes equations \cite{H-L}.
To overcome this difficulty,   we  first    assume that    $A_2(T)$ (see \eqref{AS2})  a priori satisfies $A_2(T)\le 2C_0^{1/4}$ (see \eqref{3.q2}) and  obtain the following ``weaker" basic energy estimate (see also \eqref{a2.8}):
\be\la{weaken} E(t) \le CC_0^{1/4},\ee which compared with \eqref{edwq1}, however,  is not enough and indeed will  bring us some essential difficulties to obtain all the a priori estimates (see Proposition \ref{pr1}).
Then, the first   observation is that the average of the pressure $\bp$ is uniformly bounded with positive lower and upper bounds (see \eqref{key}) by both ``weaker" basic energy estimate \eqref{weaken} and   Jensen's inequality (see Lemma \ref{je}). Combining this with the fact that the quantity $\bp$   plays a similar role as $\overline{\te}$ (see \eqref{pq}) implies that we can replace $\bar \te$ by $\bar P $ whose  positive  lower and upper  bounds play an important role  in further analysis.

Next, the second difficulty lies in the  estimation on the energy-like term $A_2(T)$ (see \eqref{AS1}) which includes the key bounds on the $L^2(\O\times(0,T))$-norm of the spatial derivatives of both the velocity and the temperature. To proceed, first, we adopt the ideas due to \cite{Hof1,C-L} to  estimate $\dot u$ and $\dot\te$ (see Lemma \ref{a113.4}), where $\dot f\triangleq f_t+u\cdot\nabla f $ denotes the material derivative of $f.$   Indeed, in this process, one needs to deal with the boundary  integrals in \eqref{bz8}, for example
$$\int_{\p\O}G(u\cdot\na)u\cdot\na n\cdot u dS,$$ which by
the classical trace theorem seems to be bounded by some good terms and the $L^p$-norm of $\na^2 u$, which is unavailable in this step. To overcome this difficulty, we adopt some idea  due to \cite{C-L}, that is, $u=u^{\perp}\times n$ with  $u^{\perp}\triangleq-u\times n$, which combined with the following fact:
\be\la{cd}\div(\na u^i\times u^{\perp})=-\na u^i\cdot\na\times u^{\perp},\ee yields that the above boundary integral can be indeed bounded by some suitable norms on both $\na u$ and  $\na G$ (see \eqref{bz3} for details). 
Next, after observing that  the evolution of $\bp$ can  be derived from the temperature equation (see \eqref{pt}), combining a careful analysis on the system \eqref{a1}  with the $L^1(0,\min\{1,T\};L^\infty)$-norm of the temperature  gives the desired basic energy estimate for small time (see Lemma \ref{a13}).
Moreover, we observe that  $R \te -\bp$ can be bounded by the combination of the initial energy with the spatial $L^2$-norm of the spatial derivatives of the temperature (see (\ref{a2.17})), which together  with a suitable combination of kinetic energy and thermal energy (see \eqref{a2.23}) yields $ A_2(T)\le CC_0^{7/24}$ (see \eqref{kyu1}) which implies $A_2(T)\le C_0^{1/ 4}$ (see \eqref{a2.34}), provided the initial energy is suitably small.

Next, the third difficulty is to obtain  the  key time-independent upper bound of the density. It should be noted that the methods used in Cauchy problem \cite{H-L}, which heavily relies on the nontrivial far field states, can not be applied to the IBVP directly. Here, by inserting the key quantity $\overline{P}$, we rewrite the continuity equation in the following way
\bnn\ba
(2\mu+\lambda) D_t \n&=-\bp \n(\n-1)- \n^2(R\te-\bp)-\n G,
\ea\enn
where $D_t$ is the material derivative and $G$ is the effective viscous flux, defined by
\be \la{hj1} D_t f \triangleq \dot f\triangleq f_t+u\cdot\nabla f,\quad G\triangleq(2\mu + \lambda)\div u -  (P-\bp), \ee
respectively.
With the aid of the uniform bound of $\overline{P}$ (see \eqref{key}),
the upper bound of the density follows directly by applying the Gr\"{o}nwall-type inequality (see Lemma \ref{le1}) and using the estimates on $R\te-\bp$ and $G$.

Finally, with the lower-order estimates including the time-independent upper bound of the density  at hand,  we can obtain the higher-order estimates 
just under the compatibility condition on the velocity \eqref{co2}. Note that all the a priori estimates are independent of the lower bound of the  density, thus after a standard approximate procedure, we can obtain the global existence of classical solutions with vacuum.  Moreover, we can as well establish the global weak solutions   almost  the same way as we established the classical one with a new modified approximate initial data.

The rest of the paper is organized as follows: In Section 2, we collect some basic facts and inequalities which will be used later. Section 3 is devoted to deriving the lower-order a priori estimates on classical solutions which are needed to extend the local solutions to all time. The higher-order estimates are established in Section 4. 
Finally, with all a priori estimates at hand, the main results, Theorems \ref{th1} and \ref{th2}, are proved in Section 5.


\section{Preliminaries}\la{se2}


First, the following local existence theory with strictly positive initial density can be shown by the standard contraction mapping arguments as in \cite{choe1,Tani,M1}.

\begin{lemma}   \la{th0} Let $\O$ be as in Theorem \ref{th1}. Assume  that
 $(\n_0,u_0,\te_0)$ satisfies \be \la{2.1} \begin{cases}
(\n_0,u_0,\te_0)\in H^3, \quad \inf\limits_{x\in\Omega}\n_0(x) >0, \quad \inf\limits_{x\in\Omega}\te_0(x)> 0,\\
u_0\cdot n=0,~~\curl u_0\times n=0,~~\na\te_0\cdot n=0~~~~\text{on}~\p\Omega.\end{cases}\ee Then there exist  a small time
$0<T_0<1$ and a unique classical solution $(\rho , u,\te )$ to the problem  (\ref{a1})--(\ref{h1}) on $\Omega\times(0,T_0]$ satisfying
\be\la{mn6}
  \inf\limits_{(x,t)\in\Omega\times (0,T_0]}\n(x,t)\ge \frac{1}{2}
 \inf\limits_{x\in\Omega}\n_0(x), \ee and
 \be\la{mn5}
 \begin{cases}
 ( \rho,u,\te) \in C([0,T_0];H^3),\quad
 \n_t\in C([0,T_0];H^2),\\   (u_t,\te_t)\in C([0,T_0];H^1),
 \quad (u,\te)\in L^2(0,T_0;H^4).\end{cases}\ee
 \end{lemma}

 \begin{remark} Applying the same arguments as in \cite[Lemma 2.1]{H-L}, one can deduce that the classical solution $(\rho , u,\te )$  obtained in Lemma \ref{th0} satisfies
\be \la{mn1}\begin{cases} (t u_{t},  t \te_{t}) \in L^2( 0,T_0;H^3)  ,\quad (t u_{tt},  t\te_{tt}) \in L^2(
0,T_0;H^1), \\  (t^2u_{tt},  t^2\te_{tt}) \in L^2( 0,T_0;H^2), \quad
 (t^2u_{ttt},  t^2\te_{ttt}) \in L^2(0,T_0;L^2).
\end{cases}\ee
Moreover, for any    $(x,t)\in \Omega\times [0,T_0],$ the following estimate holds:
   \be\la{mn2}
\te(x,t)\ge
\inf\limits_{x\in\Omega}\te_0(x)\exp\left\{-(\ga-1)\int_0^{T_0}
 \|\div u\|_{L^\infty}dt\right\}.\ee
\end{remark}

Next, we state the classical Jensen's inequality (see \cite[Theorem 3.3]{jes}), which guarantees the key uniform upper and lower bounds of $\overline P$.
\begin{lemma}\la{je}
Let $\mu$ be a positive measure on a $\si$-algebra $\mathfrak{M}$ in a set $\Omega$, so that $\mu(\Omega)=1$. If $f$ is a real function in $L^1(\mu)$,  $a<f(x)<b$ for all $x\in \Omega$, and  $\Phi$ is convex on $(a,b)$, then
\be\la{jen}\Phi\left(\int_\Omega fd\mu\right)\le \int_\Omega (\Phi\circ f)d\mu.\ee
\end{lemma}

Next, the following well-known Gagliardo-Nirenberg-Sobolev-type inequality (see \cite{nir}) will be used later frequently.

\begin{lemma}
 \la{11} Assume that $\O\subset\r^3$ is a bounded Lipschitz domain. For $r\in [2,6]$, $p\in(1,\infty)$, and $ q\in
(3,\infty),$ there exist  positive
 constants $C, \,C',$ and $C''$  which may depend  on $r,\,p,\, q$, and $\O$ such that for $f\in H^1$,
 $g\in L^p\cap W^{1,q}$,  and $\varphi,\,\psi\in H^2$,
 \be
\la{g1}\|f\|_{L^r} \le C \| f\|_{L^2}^{{(6-r)}/{2r}}  \|\na f\|_{L^2}^{{(3r-6)}/{2r}} + C' \| f\|_{L^2} ,\ee
\be \la{g2}\|g\|_{C\left(\overline{\Omega}\right)}
       \le C \|g\|_{L^p}^{p(q-3)/(3q+p(q-3))}\|\na g\|_{L^q}^{3q/(3q+p(q-3))} + C'' \|g\|_{L^2},\ee
\be\la{hs} \|\varphi\psi\|_{H^2}\le C \|\varphi\|_{H^2}\|\psi\|_{H^2}.\ee
Moreover, if $f\cdot n|_{\p \O}=0$ or $\overline{f}=0$, one has $C'=0$.
Similarly, if $g\cdot n|_{\p \O}=0$ or $\overline{g}=0$, it holds  $C''=0$.
\end{lemma}

Then, the following div-curl type inequality (see \cite[Theorem 3.2]{vww}) is used to get the estimates on the spatial derivatives of velocity. 
\begin{lemma}   \la{crle1}
	Assume that $\O\subset\r^3$ is a simply connected bounded domain with $C^{1,1}$ boundary $\partial\Omega$. Then, for $v\in W^{1,q}$ with $q\in(1,\infty)$ and $v\cdot n|_{\partial\Omega}=0$, there exists a positive constant $C=C(q,\Omega)$ such that
	\be\la{nn2}\|v\|_{W^{1,q}}\leq C\left(\|\div v\|_{L^{q}}+\|\curl v\|_{L^{q}}\right).\ee
\end{lemma}
%
%

Now, we deduce  from $(\ref{a1})_2$ that $G$ defined in \eqref{hj1}  satisfies the following elliptic equation:
 \be\la{h13}\begin{cases}
 \Delta G = \div (\rho\dot{u}),\,\,&x\in \O,\\
 \na G\cdot n=\n \dot u\cdot n,\,\,&x\in\p\O.
 \end{cases}\ee
The  standard $L^p$-estimate  for  (\ref{h13})  together with the div-curl type inequalities (see \cite{CANEHS,vww}) 
yields  the following essential estimates (see also \cite[Lemma 2.9]{C-L}).


\begin{lemma} \la{le4}
 Let  $\O\subset\r^3$ be the same as in Theorem \ref{th1} and $(\rho,u,\te)$ a smooth solution of
   (\ref{a1})--(\ref{h1}).
    Then there exists a generic positive
   constant $C$ depending only on $p$, $\mu,$   $\lambda,$  and $\O$ such that, for any $p\in [2,6],$
   \be\la{h18}
   \|\na u \|_{L^p} \le C \left( \|\div u\|_{L^p}  +\|\curl u\|_{L^p}\right),\ee
   \be\la{h19}\|\na G\|_{L^p} \le C\|\rho\dot{u}\|_{L^p},\ee
   \be\la{h191}  \|{\nabla \curl u}\|_{L^p}
   \le C (\|\rho\dot{u}\|_{L^p} + \|\na {u}\|_{L^2}),\ee
   \be  \la{h20}\|G\|_{L^p}
   \le C \|\rho\dot{u}\|_{L^2}^{(3p-6)/(2p) }
   \left(\|{\nabla u}\|_{L^2}
   +  \|P-\bp\|_{L^2}\right)^{(6-p)/(2p)},\ee
   \be  \la{h21} \|\curl u\|_{L^p}
   \le C \|\rho\dot{u}\|_{L^2}^{(3p-6)/(2p) }
   \|{\nabla u}\|_{L^2}
   ^{(6-p)/(2p)} + C\|{\nabla u}\|_{L^2}.\ee
  Moreover, it holds that
  \be \la{1h19}\ba
  \|G\|_{L^p} +\|\curl u\|_{L^p}\le C (\|\rho \dot{u}\|_{L^2}+\|\na u\|_{L^2}),
  \ea \ee
  \be\ba\la{h17} \|\na u\|_{L^p}\le& C  \|\rho\dot{u}\|_{L^2}^{(3p-6)/(2p) }  \left(\|{\nabla
  	u}\|_{L^2}+\|P-\bp\|_{L^2}\right)^{(6-p)/(2p)}\\
  &+ C (\|{\nabla	u}\|_{L^2} +  \|P-\bp\|_{L^p}).
  \ea\ee
\end{lemma}

Next, we state the following estimates  on $\dot u$ with $u \cdot n|_{\p \O}=0$, whose proof can be found in \cite[Lemma 2.10]{C-L}.

\begin{lemma}\la{uup1}
	Let $\Omega \subset \mathbb{R}^3$ with $C^{1,1}$ boundary.
	Assume that $u$ is smooth enough and $u \cdot n|_{\p \O}=0$,
	then there exists a generic positive constant $C$ depending only on   $\Omega$ such that
	\be\la{tb90}
	\ba\|\dot{u}\|_{L^6}\le C(\|\nabla\dot{u}\|_{L^2}+\|\nabla u\|_{L^2}^2),
	\ea\ee
	\be\la{tb11}\ba
	\|\nabla\dot{u}\|_{L^2}\le C(\|\div \dot{u}\|_{L^2}+\|\curl \dot{u}\|_{L^2}+\|\nabla u\|_{L^4}^2).
	\ea\ee
\end{lemma}

Furthermore, by the classical elliptic theory owing to  Agmon-Douglis-Nirenberg \cite{adn}, one has the following estimates for smooth solution to  the Lam\'{e}'s system:
\be\la{lames}\begin{cases}
	-\mu\Delta u-(\lambda+\mu)\nabla\div u=-\rho\dot{u}-\nabla P , \,\, &x\in\Omega, \\
	u\cdot n=0,\,\curl u\times n=0,\,\,&x\in\partial\Omega,
\end{cases} \ee
where 
$\Omega\subset\r^3$ is a bounded smooth domain, and $\mu,\lambda$ satisfy the condition \eqref{h3}.
\begin{lemma}  \la{zhle}
	Let $u$ be a smooth solution of the Lam\'{e}'s system (\ref{lames}). Then for $p\in[2,6]$ and $k\ge 2$,   there exists a positive constant $C$ depending only on $\lambda,\,\mu,\,p,\,k,$ and $\Omega$ such that
	\be\ba\la{rmk1}
	\| u\|_{W^{k,p}} 
	&\leq C(\|\rho\dot{u}\|_{W^{k-2,p}}+\|\nabla P\|_{W^{k-2,p}}+\|\nabla u\|_{L^2}).
	\ea\ee
\end{lemma}

Next, the following Gr\"{o}nwall-type inequality will be used to get the uniform (in time) upper bound of the density $\n$, whose proof can be found in \cite[Lemma 2.5]{H-L}.
\begin{lemma} \la{le1}
	Let the function $y\in W^{1,1}(0,T)$ satisfy
	\be \notag
	y'(t)+\al(t) y(t)\le  g(t)\mbox{  on  } [0,T] ,\quad y(0)=y_0,
	\ee
	where $0<\al_0\le \al(t)$ for any $t\in[0,T]$ and  $ g \in L^p(0,T_1)\cap L^q(T_1,T)$  for some $p,\,q\ge 1, $  $T_1\in [0,T].$ Then it has
	\be \la{2.34}
	\sup_{0\le t\le T} y(t) \le |y_0| + (1+\al_0^{-1}) \left(\|g\|_{L^p(0,T_1)} + \|g\|_{L^q(T_1,T)}\right).
	\ee
\end{lemma}


Next, to derive the exponential decay property of the solutions, we consider the following auxiliary problem
\be\la{e480}\begin{cases}
{\rm div}v=f,&x\in\Omega, \\
v=0,&x\in{\partial\Omega}.
\end{cases} \ee
\begin{lemma} \cite[Theorem III.3.1]{GPG} \la{th00}  There exists a linear operator $\mathcal{B} = [\mathcal{B}_1 , \mathcal{B}_2 , \mathcal{B}_3 ]$ enjoying
the properties:

1)The operator $$\mathcal{B}:\{f\in L^p\left|\right.\overline f =0\}\mapsto W^{1,p}_0$$ is a bounded linear one, that is,
\bnn \|\mathcal{B}[f]\|_{W^{1,p}_0}\le C(p)\|f\|_{L^p}, \mbox{ for any }p\in (1,\infty).\enn

2) The function $v = \mathcal{B}[f]$ solves the problem (\ref{e480}).

3) If, moreover, for $f = \div  g$ with a certain $g\in L^r,$ $g\cdot n|_{\pa\O}=0,$  then  for   any $r  \in (1,\infty),$
\bnn \|\mathcal{B}[f]\|_{L^{r}}\le C(r)\|g\|_{L^r} .\enn
\end{lemma}

Finally, in order to estimate $\|\nabla u\|_{L^\infty}$ for the further higher order estimates, we need the following Beale-Kato-Majda-type inequality, which was first proved in \cite{bkm,kato} when $\div u\equiv 0,$ 
whose detailed proof in the case of slip boundary condition  can be found in \cite[Lemma 2.7]{C-L} (see also \cite{h101,h1x}).
\begin{lemma}\la{le9}
Let $\O\subset\r^3$ be a  bounded smooth domain. For $3<q<\infty$, assume that $ u \in \{ f \in W^{2,q}|  f\cdot n=0, \curl f\times n=0 \text{ on } \partial \Omega \}$,  then there is a positive constant  $C=C(q,\O)$ such that
\bnn\ba
\|\na u\|_{L^\infty}\le C\left(\|{\rm div}u\|_{L^\infty}+\|\curl u\|_{L^\infty} \right)\ln(e+\|\na^2u\|_{L^q})+C\|\na u\|_{L^2} +C.
\ea\enn
\end{lemma}


\section{\la{se3} A priori estimates (I): lower-order estimates}

In this section,  we  will establish a priori bounds for the local-in-time smooth solution to problem (\ref{a1})--(\ref{h1}) obtained in Lemma \ref{th0}.

Let $(\n,u,\te)$ be a smooth solution to the  problem (\ref{a1})--(\ref{h1})  on $\Omega\times (0,T]$ for some fixed time $T>0,$ with  initial data $(\n_0,u_0,\te_0)$ satisfying \eqref{2.1}.
For
$\si(t)\triangleq\min\{1,t\}, $  we  define
$A_i(T)(i= 1,2,3)$ as follows:
  \be\la{As1}
  A_1(T) \triangleq \sup_{t\in[0,T] }\|\nabla u \|_{L^2}^2
  + \int_0^{T}\int  \rho|\dot{u}|^2dxdt,
  \ee
  \be\label{AS1}
  A_2(T) \triangleq \frac{1}{2(\ga-1)}\sup_{t\in[0,T] }\int\n (R\te-\bp)^2dx
  +\int_0^T\left( \|\na u\|_{L^2}^2+\|\na \te\|_{L^2}^2\right)dt,
  \ee
  \be\ba \label{AS2}
  A_3(T) \triangleq &\sup_{t\in(0,T]}\left(\si \|\na u\|_{L^2}^2+\si^2\int\n |\dot u|^2dx + \si^2\|\na\te\|_{L^2}^2 \right)\\
  & + \int_0^T\int\left(\si\n |\dot u|^2 +\sigma^2|\nabla\dot{u}|^2 +\sigma^2\n|\dot \te|^2 \right)dxdt.
  \ea\ee

We have the following key a priori estimates on $(\n,u,\te)$.
\begin{pro}\la{pr1}
For  given numbers $M>0$, $\on> 2,$  and $\bt> 1,$ assume further that $(\rho_0,u_0,\te_0) $  satisfies
\be \la{3.1}
0<\inf \rho_0 \le\sup \rho_0 <\on,\quad 0<\inf \te_0 \le\sup \te_0 \le \bt, \quad \|\na u_0\|_{L^2} \le M.
\ee
Then there exist  positive constants $K$,  $C^\ast$, $\al$,  $\te_\infty$, and $\ep_0$ all depending on $\mu,\,\lambda,\, \ka,\, R,\, \ga,$      $\on,$ $\bt$, $\O$, and $M$ such that if $(\rho,u,\te)$ is a smooth solution to the problem (\ref{a1})--(\ref{h1}) on $\O\times (0,T]$ satisfying
\be \la{z1}
0<  \rho\le 2\on, \,\,\, A_1(T)\le 3 K, \,\,\, A_2(T) \le 2C_0^{1/4}, \,\,\, A_3(T)  \le 2C_0^{1/6},
\ee
the following estimates hold:
\be \la{zs2}
0< \rho\le 3\on/2,\,\,\, A_1(T)\le 2 K, \,\,\, A_2(T) \le C_0^{1/4},  \,\,\, A_3(T)  \le  C_0^{1/6},
\ee
and for any $t\geq1$,
\be\la{h22}
\|\n-1\|_{L^2}+\|u\|_{W^{1,6}}^2+\|\te-\te_\infty\|_{H^2}^2\le C^\ast e^{-\al t},
\ee
provided \be\la{z01}C_0\le \ve_0.\ee      \end{pro}

\begin{proof}
Proposition \ref{pr1} is a straight consequence of
the following Lemmas \ref{le2}, \ref{le6}, \ref{le3}, \ref{le7}, and \ref{pr2}  with $\ve_0$ as in (\ref{t7}).
\end{proof}

In this section, we always assume that $C_0\le 1$ and let $C$ denote some generic positive constant depending only on $\mu$,  $\lambda$, $\ka$,  $R$, $\ga$, $\on$, $\bt$, $\O,$ and $M,$ and we write $C(\al)$ to emphasize that $C$ may depend  on $\al.$

To begin with, we have the following uniform estimate on $\bp$, which plays an important role in the whole analysis.

\begin{lemma}\la{a13.1} Under the conditions of Proposition \ref{pr1}, there exists a positive constant $C$ depending only on $\mu$, $R$, and $\on$ such that if $(\rho,u,\te)$ is a smooth solution to the problem (\ref{a1})--(\ref{h1})  on $\Omega\times (0,T] $ satisfying
\be\la{3.q2}
0<\n\le 2\on ,\quad A_2(T)\le 2C_0^{1/4},
\ee
the following estimates hold:
\be \la{a2.112}
\sup_{0\le t\le T}\int\left( \n |u|^2+(\n-\tn)^2\right)dx \le C  C_0^{1/4},
\ee
and
\be \la{key}
0<\pi_1  \le \overline P (t)\le \pi_2,~~~\mbox{for~any}~t\in [0,T],
\ee
where  $\pi_1 $ and  $\pi_2$ are  positive constants depending only on $\mu$, $\gamma$, $R$, and $\O$.
\end{lemma}

\begin{proof}
First, it follows from (\ref{3.1}) and  (\ref{mn2}) that, for all $ (x,t)\in \Omega\times(0,T),$
\be \la{3.2}
\te(x,t)>0 .
\ee

Note that
\be\notag
\Delta u = \na \div u - \na \times \curl u,
\ee
one can rewrite $(\ref{a1})_2$  as
\be\la{a11}\ba
\n (u_t+  u\cdot \na u )&=(2\mu+\lambda)\na{\rm div}u- \mu \na \times \curl u-\na P.\ea\ee
Adding $(\ref{a11})$ multiplied by $u$ to $(\ref{a1})_3$ multiplied by $1-\te^{-1}$ and  integrating the resulting equality over $\Omega$ by parts,  we obtain
after   using $(\ref{a1})_1$, \eqref{h3}, \eqref{3.2}, and  the boundary conditions \eqref{h1}    that
\be\la{la2.7}\ba
E'(t)
&=-\int \left( \frac{\lambda(\div u)^2+2\mu |\mathfrak{D}(u)|^2}{{\te}}+\ka \frac{|\na \te|^2}{\te^2} \right)dx \\
&\quad - \mu \int \left(|\curl u|^2+2(\div u)^2 - 2 |\mathfrak{D}(u)|^2\right)dx\\
&\le 2\mu \int  |\na u|^2 dx,
\ea\ee
where  $E(t)$ is  the basic energy defined by \eqref{enet}.

Then, integrating \eqref{la2.7} with respect to $t$ over $(0,T)$ and using \eqref{3.q2}, one has
 \be\la{a2.8}\ba
&\sup_{0\le t\le T} E(t)
\le C_0+ 2\mu \int_{0}^{T} \int  |\na u|^2 dxdt\le C C_0^{1/4},\ea\ee
which together with
\be\la{a2.9}\ba
 (\n-1)^2\ge 1+\n\log\n-\n&=(\n-1)^2\int_0^1\frac{1-\al}{\al (\n-1)+1}d\al  \ge \frac{(\n-1)^2}{ 2(2\on+1)  }
 \ea\ee
gives (\ref{a2.112}).

Next,  it is easy to deduce from $\eqref{a1}_1$  and \eqref{m} that for any $t\in[0,T]$,
\be \la{mmm}
\overline\rho(t)=\overline{\n_0}=1.
\ee
Denote $d\mu\triangleq|\Omega|^{-1}\n dx$. Then $d\mu$ is a positive measure satisfying $\mu(\Omega)=1$ due to \eqref{3.q2} and \eqref{mmm}. Moreover, observe that $y-\log y-1$ is a convex function in $(0,\infty)$,  it thus follows directly from Jensen's inequality \eqref{jen} that for any $t\in [0,T]$,
\be \notag
\overline {\rho \te }(t) - \log \overline {\rho \te }(t) -1 \le \int(\te-\log\te-1)\frac{\n dx}{|\Omega|}\le C
\ee due to   \eqref{a2.8}.
This in particular gives  \eqref{key} and finishes the proof of Lemma \ref{a13.1}.

\end{proof}

\begin{remark}
    It should be pointed out that the following term in (\ref{la2.7})
    \begin{equation}\label{xm}
        - \mu \int \left( |\curl u|^2+2(\div u)^2 - 2 |\mathfrak{D}(u)|^2 \right)dx
    \end{equation}
    is a sign-undetermined term due to the slip boundary condition (\ref{bb}), which is in sharp contrast to the Cauchy problem \cite{H-L} where the term (\ref{xm}) vanishes after integration by parts.
    Thus, in this case, we can not bound the basic energy only by the initial energy. However, this term obviously can be bounded by $C \int  |\nabla u|^2dx$, which implies a ``weaker" basic energy estimate (\ref{a2.8}).
\end{remark}


The next lemma provides an estimate on the term $A_1(T)$.
\begin{lemma}\la{le2}
	Under the conditions of Proposition \ref{pr1}, there exist positive constants  $K $  and $\ep_1 $ both depending only  on $\mu,\,\lambda,\, \ka,\, R,\, \ga,\, \on,\,\bt,\, \O,$ and $M$ such that if  $(\rho,u,\te)$ is a smooth solution to the problem  (\ref{a1})--(\ref{h1}) on $\Omega\times (0,T] $ satisfying
	\be\la{3.q1}  0<\n\le 2\on ,\quad A_2(T)\le 2C_0^{1/4},\quad A_1(T)\le 3K,\ee
	the following estimate holds:
	\be\la{h23} A_1(T)\le 2K ,  \ee
	provided   $C_0\le \ep_1.$
\end{lemma}

\begin{proof}
 First, integrating $(\ref{a11})$  multiplied by $2u_t $ over $\Omega $ by parts gives
 \be\ba \la{hh17}
 &\frac{d}{dt}\int \left(  {\mu} |\curl u|^2+ (2\mu+\lambda)(\div u)^2\right)dx+ \int\rho |\dot u|^2dx \\
&\le 2\int  P\div u_t dx+ \int \n|u\cdot \na u|^2dx\\
 &=  2\frac{d}{dt}\int  (P-\overline P) \div u  dx-2\int (P-\overline{P})_t \div u dx+\int \n|u\cdot \na u|^2dx\\
 &= 2\frac{d}{dt}\int  (P-\overline P) \div u  dx-\frac{1}{2\mu+\lambda}\frac{d}{dt}\int (P-\overline P)^2 dx\\
 &\quad-\frac{2}{2\mu+\lambda}\int (P-\overline{P})_t G dx+ \int \n|u\cdot \na u|^2dx ,
 \ea\ee
where in the last equality  we have used
(\ref{hj1}).

Next,   straight calculations show
that for any $p\in [2,6],$
\be\la{pq}\ba  \|R\te- \overline P \|_{L^p}
  \le R\|\te- \overline\te \|_{L^p}+C|R\overline\te-\overline P|\le C(\hat{\n}) \|\na \te\|_{L^2},\ea\ee
where one has used (\ref{g1}) and the following fact:
\be\notag\ba
|R\overline\te-\overline P|=&\frac{R}{|\O|}\xl|\int(1-\n)\te dx\xr|
=  \frac{R}{|\O|}\xl|\int(1-\n)(\te-\overline\te)dx\xr|\\
\le&C\|\n-1\|_{L^2}\| \te-\overline\te\|_{L^2}\\
\le&C(\on)C_0^{1/8}\|\na \te\|_{L^2}\ea\ee
due to   (\ref{mmm}) and \eqref{a2.112}. Thus, it follows from \eqref{key}, \eqref{3.q1}, \eqref{a2.112}, and \eqref{pq} that for any $p\in[2,6]$,
\be\la{p}\ba  \|P-\overline P \|_{L^p}&= \| \n(R\te-
\overline P )+ (
 \n -1)\overline P\|_{L^p}\\&\le \|\n(R\te-\overline P )\|_{L^2}^{(6-p)/(2p)}
 \|\n(R\te-\overline P )\|_{L^6}^{ 3(p-2)/(2p)}+ \pi_2 \|\n-1\|_{L^p}
 \\&\le C(\on)C_0^{(6-p)/(16p)}
 \|\na\te \|_{L^2}^{ 3(p-2)/(2p)}+ C(\hat\n)C_0^{1/(4p)},
  \ea\ee
which together with (\ref{h17}) and \eqref{3.q1} yields
\be \la{3.30}  \|\na
u\|_{L^6} \le C(\hat\n) \left(  \|\n^{1/2}\dot
u\|_{L^2}+\|\na u\|_{L^2}+ \|\na \te\|_{L^2}+C_0^{1/24}\right).
\ee

Note that (\ref{a1})$_3$ implies
\be \la{op3} \ba
P_t=&-\div (Pu) -(\gamma-1) P\div u+(\ga-1)\ka \Delta\te\\&+(\ga-1)\left(\lambda (\div u)^2+2\mu |\mathfrak{D}(u)|^2\right),
\ea\ee
which along with \eqref{h1} gives
\be \ba \la{pt}
\bp_t =&
-(\gamma-1) \overline{P\div u} +(\ga-1)\left(\lambda
\overline{(\div u)^2}+2\mu \overline{|\mathfrak{D}(u)|^2}\right).
\ea \ee
We thus obtain after  using  integration by parts, (\ref{key}),  (\ref{g1}), (\ref{h19}), and  (\ref{p})--(\ref{op3}) that
\be\la{a16}\ba
&\left|\int   P_t Gdx\right| \\
&\le C\int P(|G||\na u|+ |u||\na G|)dx+ C\int\left( |\na\te||\na G|+|\na u|^2|G|\right)dx \\
&\le  C\int |P-  \bp| (|G||\na u|+|u||\na G|)dx +C\bp\int (|G||\na u|+|u||\na G|)dx \\
&\quad + C \|\na G\|_{L^2} \|\na \te\|_{L^2} + C \| G\|_{L^6} \|\na u\|_{L^2}^{3/2}  \|\na u\|_{L^6}^{1/2} \\
&\le C(\hat \n)(\|\na \te\|_{L^2}^{1/2}+1)\|\na G\|_{L^2}\|\na u\|_{L^2}+ C \|\na G\|_{L^2} \|\na \te\|_{L^2} \\
& \quad+ C(\hat\n) \|\na G\|_{L^2} \|\na u\|_{L^2}^{3/2} \left(\|\n^{1/2}\dot u\|_{L^2}+\|\na u\|_{L^2}+\|\na\te\|_{L^2}+1 \right)^{1/2} \\
&\le \de \|\na G\|_{L^2}^2 +\de \|\rho^{1/2}\dot u\|^2_{L^2}+C(\de,\on) \left( \|\na u\|_{L^2}^2+ \|\na \te\|_{L^2}^2+ \|\na u\|^6_{L^2}\right) \\
&\le C(\on)\de\|\rho^{1/2} \dot u\|^2_{L^2}     +C(\de,\on) \left( \|\na u\|_{L^2}^2+ \|\na \te\|_{L^2}^2 +\|\na  u\|^6_{L^2}\right),
\ea\ee
and
\be\la{a161}\ba   \left|\int   \bp_t  Gdx\right|
\le& C(\on)(\| \na u\|_{L^2}+\| \na u\|_{L^2}^2)\| G\|_{L^2}
\\  \le& \de \|\na G\|_{L^2}^2  +C(\de,\on)(\|\na u\|_{L^2}^2+\|\na u\|_{L^2}^4)
\\\le&  C(\on)\de\|\rho^{1/2}
\dot u\|^2_{L^2}     +C(\de,\on) (\|\na u\|_{L^2}^2+\|\na u\|_{L^2}^4),
\ea\ee
where one has used
\be\la{511} \ba
|\overline P_t|
&\le C\| P-\bp\|_{L^2} \| \na u\|_{L^2} + C \|\na u\|_{L^2}^2 \le C(\on)(C^{1/8}_0\| \na u\|_{L^2}+\| \na u\|_{L^2}^2)
\ea \ee
owing to \eqref{pt} and \eqref{p}.

Then, it follows from (\ref{g1}), \eqref{3.q1}, and (\ref{3.30}) that
\be\la{op1}\ba
\int \n|u\cdot \na u|^2dx&\le C(\on)\|u\|_{L^6}^2 \|\na u\|_{L^2} \|\na u\|_{L^6}  \\
&\le \de\|\rho^{1/2} \dot u\|_{L^2}^2+ C(\de,\on)\left(\|\na u\|_{L^2}^2+\|\na\te\|_{L^2}^2+\|\na
u\|_{L^2}^6\right).
\ea\ee

Finally, substituting (\ref{a16}), \eqref{a161}, and (\ref{op1}) into (\ref{hh17}), one obtains after choosing $\de$ suitably small that
\be\ba \la{1hh17}
&\frac{d}{dt}\int \left(  {\mu} |\curl u|^2+ (2\mu+\lambda)(\div u)^2\right)dx+ \frac{1}{2\mu+\lambda}\frac{d}{dt}  \|P-\overline P\|_{L^2}^2  + \frac{1}{2}\int\rho |\dot u|^2dx \\
& \le 2\frac{d}{dt}\int  (P-\overline P) \div u  dx + C(\on)\left(\|\na u\|_{L^2}^2+\|\na\te\|_{L^2}^2+\|\na
u\|_{L^2}^6\right).
\ea\ee

Note that it holds
\be\la{cz}\ba \|P-\overline P\|_{L^2}^2(0)&=R^2\int(\n_0\te_0-\overline{\n_0\te_0})^2dx\\
&\le C(\on)\int\n_0(\te_0-\overline{\n_0\te_0})^2dx+C\int(\n_0-1)^2dx\\
&\le C(\on, \bt)C_0,\ea\ee
where we have used \eqref{key}, \eqref{a2.9}, and the following fact:
\be \ba \label{jia10} \int \n_0(\te_0- \overline{\n_0\te_0})^2dx &\le C\int \n_0(\te_0-1)^2dx+C|1-\overline{\n_0\te_0}|^2\\
&\le  C\int \n_0(\te_0-1)^2dx+C\left|\int\n_0(1-\te_0)dx\right|^2\\
&\le  C(\hat\n,\bt)\int \n_0(\te_0-\log\te_0-1) dx\\
&\le  C(\hat\n, \bt)C_0
\ea\ee
due to \eqref{m} and
 \be\la{cz1}\te-\log\te-1 =(\te-1)^2\int_0^1\frac{\al}{\al (\te-1)+1}d\al\geq\frac{1}{2(\|\te(\cdot,t)\|_{L^{\infty}}+1)}(\te-1)^2.\ee
Then, integrating (\ref{1hh17}) over $(0,T)$, one deduces from (\ref{h18}), \eqref{cz}, (\ref{3.q1}), and \eqref{p} that
\bnn\la{h81} \ba
&\sup_{0\le t\le T}\|\na u\|_{L^2}^2+ \int_0^{T}\int\rho|\dot{u}|^2dxdt\\
&\le CM^2+C(\on,\hat\te)C_0^{1/4} + C(\on ) C_0^{1/4}\sup_{0\le t\le T}\|\na u\|_{L^2}^4 + C(\on ) C_0^{1/8}\sup_{0\le t\le T}\|\na u\|_{L^2}\\
&\le CM^2+C(\on,\hat\te)C_0^{1/12}+ C(\on ) C_0^{1/4}\sup_{0\le t\le T}\|\na u\|_{L^2}^4 \\
&\le K+9C(\on)C_0^{1/4}K^2 \\
&\le 2K,
\ea \enn
with $K\triangleq CM^2+C(\on,\hat\te) +1$, provided
\be\notag C_0\le \ep_1 \triangleq \min\left\{1,\xl(9C(\on)K\xr)^{-4}\right\}.\ee
The proof of Lemma \ref{le2} is completed.
\end{proof}

Next,  to estimate $A_3(T)$, we adopt the approach due to Hoff \cite{Hof1} (see also Huang-Li \cite{H-L})  to establish the following elementary estimates on $\dot u$ and $\dot \te$, where the boundary terms are handled by the ideas due to \cite{C-L}.
The estimate of $A_3(T)$ will be postponed to Lemma \ref{le6}.

\begin{lemma}\la{a113.4}
	Under the conditions of Proposition \ref{pr1}, let $(\rho,u,\te)$ be a smooth solution to the problem (\ref{a1})--(\ref{h1}) on $\Omega\times (0,T] $ satisfying (\ref{z1}) with $K$ as in Lemma \ref{le2}.  Then  there exist positive constants $C$, $ C_1$, and $C_2$ depending only on $\mu,\,\lambda, \,k,\, R,\, \ga,\, \on,\,\bt,\,\O,$ and $ M$  such that, for any $\eta\in (0,1]$ and $m\geq0,$
the following estimates hold:
\be\ba  \la{an1}
(\sigma B_1)'(t) + \frac{1}{2}\sigma \int \rho |\dot u|^2dx
\le   C C_0^{1/4} \sigma' + C\left(\|\na u\|_{L^2}^2+\|\na\te\|_{L^2}^2\right),
\ea\ee
\be\la{ae0}\ba
&\left(\sigma^{m}\|\rho^{1/2}\dot{u}\|_{L^2}^2\right)_t+C_1 \sigma^{m}\|\na\dot{u}\|_{L^2}^2\\
&\le - 2\left(\int_{\p \O}  \sigma^m (u \cdot \na n \cdot u) G dS\right)_t + C(\si^{m-1}\si'+\si^m)  \|\rho^{1/2} \dot u\|_{L^2}^2 \\&\quad+
C_2  \si^m \|\rho^{1/2} \dot \te\|_{L^2}^2+ C\|\na u\|^2_{L^2}+C \si^m \|\na u\|^4_{L^4} + C \si^m  \|\te \na u\|_{L^2}^2,\\\ea\ee
  and
 \be\la{nle7}\ba  &(\si^mB_2 )'(t)+\si^m \int\n|\dot \te|^2dx\\
&\le C \eta \si^m\|\na\dot u\|_{L^2}^2+C \|\na
\te \|_{L^2}^2+C\si^m \|\na u\|_{L^4}^4+C(\eta)\si^m \|\te\na u\|_{L^2}^2,\ea\ee
where
\be \la{an2} \ba B_1(t)\triangleq& \mu\|\curl u\|_{L^2}^2+ (2\mu+\lambda) \|\div u\|_{L^2}^2 \\&+ \frac{1}{2\mu+\lambda}  \|P-\overline P\|_{L^2}^2 -2 \int \div u(P-\bp) dx, \ea\ee
and
 \be\la{e6}
B_2(t)\triangleq\frac{\ga-1}{R}\left(\ka \|\na
\te\|_{L^2}^2-2 \int (\lambda (\div u)^2+2\mu|\mathfrak{D}(u)|^2)\te dx\right).\ee
\end{lemma}

\begin{proof}   First, multiplying \eqref{1hh17} by $\sigma$ and using \eqref{p} and \eqref{z1} give \eqref{an1} directly.

Now, we will prove  (\ref{ae0}).

First, one can rewrite $(\ref{a1})_2$  as
\be\la{a111}\ba\n \dot u&=\na G- \mu \na \times \curl u,
\ea\ee
with $G$ defined in \eqref{hj1}.
For $m\ge 0,$ operating $ \si^m\dot u^j[\pa/\pa t+\div (u\cdot)]$ to $ (\ref{a111})^j$ and integrating the resulting equality over $\Omega$ by parts lead to
\be\la{m4} \ba
& \left(\frac{\sigma^m}{2}\int\rho|\dot{u}|^2dx \right)_t -\frac{m}{2}\sigma^{m-1}\si'\int\rho|\dot{u}|^2dx\\
 &=  \int_{\p \O}  \sigma^m \dot{u} \cdot n G_t dS - \int  \sigma^m [\div \dot{u} G_t + u \cdot \na \dot u \cdot \na G]dx \\
&\quad- \mu \int\sigma^m\dot{u}^j\xl[(\na \times \curl u)_t^j + \div (u (\na \times \curl u)^j)\xr] dx  \triangleq\sum_{i=1}^{3}N_i.
\ea \ee

Noticing that
\be\la{pzw1} u\cdot\nabla u\cdot n=-u\cdot\nabla n\cdot u ~~\quad \mbox{on}~\p \O\ee
due to $u \cdot n|_{\p \O}=0$,
one can deduce from \eqref{h1} and \eqref{pzw1} that
\be \la{bz8}\ba
N_1&=- \int_{\p \O}  \sigma^m (u \cdot \na n \cdot u) G_t dS\\
&=- \left(\int_{\p \O} \sigma^m (u \cdot \na n \cdot u) G dS\right)_t + m \si^{m-1} \si' \int_{\p \O}( u \cdot \na n \cdot u) G dS \\
&\quad+  \int_{\p \O}  \sigma^m (\dot u \cdot \na n \cdot u) G dS+ \int_{\p \O}  \sigma^m ( u \cdot \na      n \cdot \dot u) G dS\\
&\quad-  \int_{\p \O}  \sigma^m G( u \cdot \na) u \cdot \na n \cdot u  dS -\int_{\p \O}  \sigma^m G u \cdot \na n \cdot (u\cdot \na )u  dS\\
&\le - \left(\int_{\p \O}  \sigma^m (u \cdot \na n \cdot u) G dS\right)_t + C\si^{m-1} \si' \| \na u\|_{L^2}^2\|\na G\|_{L^2} \\
&\quad+\de \si^m \|\dot u\|_{H^1}^2+ C(\de) \si^m \|\na u\|_{L^2}^2 \|\na G\|_{L^2}^2\\
& \quad- \int_{\p \O}  \sigma^m G( u \cdot \na) u \cdot \na n \cdot u  dS -\int_{\p \O}  \sigma^m G u \cdot \na n \cdot (u\cdot \na )u  dS,
\ea \ee
where one has used
\be \ba\notag
\left|\int_{\partial \O}  (\dot u\cdot \na n\cdot u+ u\cdot \na n\cdot\dot u) G dS\right| &\le C \|\dot u\|_{H^1} \| u\|_{H^1} \|G\|_{H^1} \\&\le C \|\dot u\|_{H^1} \|\na u\|_{L^2} \|\na G\|_{L^2},
\ea \ee
and
\be\la{b2}
\left|\int_{\partial \O}  ( u\cdot \na n\cdot u) G dS  \right| \le C \|\na u\|_{L^2} ^2\|\na G\|_{L^2}.
\ee
Now, we will adopt the idea in \cite{C-L} to deal with the last two boundary terms in \eqref{bz8}. In fact, denote $u^{\perp}\triangleq-u\times n$, it follows from
$u \cdot n|_{\p \O}=0$ that
\begin{align}\notag
u=u^{\perp} \times n~~~~~\text{on}\, \p \O,
\end{align}
which along with \eqref{g1}, \eqref{cd}, and integration by parts yields
\be \la{bz3}\ba &- \int_{\partial\Omega} G (u\cdot \na) u\cdot\na n\cdot u dS \\&= -\int_{\partial\Omega}  G u^\bot\times n \cdot\na u^i \nabla_i n\cdot u  dS \\&= - \int_{\partial\Omega} G n\cdot ( \na u^i \times  u^\bot)    \nabla_i n\cdot u dS\\
&= - \int\div( G( \na u^i \times  u^\bot)   \nabla_i n\cdot u) dx \\
&= - \int \na (\nabla_i n\cdot u G) \cdot ( \na u^i \times  u^\bot)   dx  - \int \div( \na u^i \times  u^\bot)    \nabla_i n\cdot u   G  dx \\
&= - \int \na (\nabla_i n\cdot u G) \cdot ( \na u^i \times  u^\bot)   dx  + \int  G \na  u^i \cdot \na\times  u^\bot     \nabla_i n\cdot u     dx \\
& \le C \int |\na G||\na u||u|^2dx+C \int |G| (|\na u|^2|u|+|\na u||u|^2)dx
\\& \le C  \|\na G\|_{L^6}\|\na u\|_{L^2}\|u\|^2_{L^6}
+C  \| G\|_{L^3}\|\na u\|^2_{L^4}\|u\|_{L^6}+C\| G\|_{L^{6}} \|\na      u\|_{L^2}\|u\|_{L^6}^2
\\& \le \de \|\na G\|_{L^6}^2+C(\de) \|\na u\|^6_{L^2}+C\|\na u\|^4_{L^4}+ C  \|  \na G\|_{L^2}^2 (\|\na u\|^2_{L^2} +1 ).\ea\ee
Similarly, it holds that
\be \la{bz4}\ba &-  \int_{\partial\Omega}G  u\cdot\na n\cdot ({u}\cdot\na) u dS\\& \le \de \|\na G\|_{L^6}^2+C(\de) \|\na u\|^6_{L^2}+C\|\na u\|^4_{L^4}+ C \|  \na G\|_{L^2}^2(\|\na u\|^2_{L^2} +1 ).\ea\ee

Next, it follows from \eqref{hj1} that
\be \ba \la{pt11}
G_t =& (2\mu+\lambda)\div u_t - (P_t-\bp_t )\\
=& (2\mu+\lambda) \div \dot u - (2\mu+\lambda) \div (u\cdot \na u) - R\rho \dot\te + \div(Pu) + R\overline{\rho \dot\te}\\
=& (2\mu+\lambda) \div \dot u - (2\mu+\lambda) \na u : (\na u)^{\rm tr} -  u \cdot  \na G + P \div u - R\rho \dot\te + R\overline{\rho \dot\te},
\ea \ee
where one has used
\be \ba \label{jia1}
P_t=(R\n \te)_t=R\n \dot{\te}-\div (Pu),~~~\bp_t=R\overline{\n\dot{\te}}.
\ea \ee
Then, integration by parts combined with \eqref{pt11} gives
\be\la{m5} \ba
N_2 =&  - \int  \sigma^m [\div \dot{u} G_t + u \cdot \na \dot u \cdot \na G]dx \\
=& - (2\mu+\lambda) \int  \sigma^m (\div \dot{u} )^2 dx +  (2\mu+\lambda)  \int \sigma^m \div \dot{u} \na u : (\na u)^{\rm tr} dx \\
&+\int  \sigma^m \div \dot{u}  u \cdot  \na G  dx - \int \sigma^m \div \dot{u} P \div u dx  \\
&+ R \int \sigma^m \div \dot{u} \rho \dot\te dx -  R \overline{\rho \dot\te}  \int \sigma^m \div \dot{u}  dx - \int  \sigma^m u \cdot \na \dot u \cdot \na Gdx \\
\le &   - (2\mu+\lambda) \int  \sigma^m (\div \dot{u} )^2 dx\\
& + C \si^m \|\na \dot{u}\|_{L^2} \|\na u\|_{L^4}^2 + C \si^m \|\na \dot{u}\|_{L^2} \|\na G\|_{L^2}^{1/2} \|\na G\|_{L^6}^{1/2} \|u\|_{L^6} \\
&+ C(\on) \si^m \|\na \dot{u}\|_{L^2} \|\te \na u\|_{L^2}+ C(\on) \si^m \|\na \dot{u}\|_{L^2} \|\rho^{1/2} \dot \te\|_{L^2}.\ea \ee

Note that
$$\curl u_t=\curl \dot u-u\cdot \na \curl u-\na u^i\times \nabla_iu,$$
which together with some straight calculations yields
\be\la{ax3999}\ba
N_3 &=- \mu \int\sigma^{m}|\curl\dot{u}|^{2}dx+\mu \int\sigma^{m}\curl\dot{u}\cdot(\nabla u^i\times\nabla_i u) dx \\&\quad+\mu \int\sigma^{m} u\cdot\na  \curl u \cdot\curl\dot{u} dx  +\mu \int\sigma^{m}    u \cdot \na \dot{u}\cdot (\nabla\times\curl u) dx  \\
&\le - \mu \int\sigma^{m}|\curl\dot{u}|^{2}dx+\de  \sigma^{m}(\|\na \dot u\|_{L^2}^2 +
\|\na \curl u\|_{L^6}^2)\\&\quad +C(\de) \sigma^{m} \|\na u\|_{L^4}^4 +C(\de)  \sigma^{m}\|\na  u\|_{L^2}^4\|\na \curl u\|_{L^2}^2.\\
\ea\ee

Finally, it is easy to deduce from  Lemmas  \ref{le4} and \ref{uup1} that
\be \la{bz5} \|G\|_{H^1}+\|\curl u\|_{H^1}\le C(\|\n \dot u\|_{L^2}+\|\na  u\|_{L^2}), \ee
and that
\be \la{bz6}\ba  \|\na G\|_{L^6}+\|\na\curl u\|_{L^6} +\| \dot u\|_{H^1}
 &\le C(\|\n \dot u\|_{L^6}+\|\na  u\|_{L^2})+\| \dot u\|_{H^1}
\\&\le C(\on)  (\|\na \dot u\|_{L^2}+  \|\na  u\|_{L^2}+  \|\na  u\|_{L^2}^2).\ea \ee

Hence, submitting  \eqref{bz8},  \eqref{m5}, and \eqref{ax3999} into \eqref{m4}, one obtains after  using \eqref{bz3}, \eqref{bz4}, \eqref{z1},  \eqref{bz5}, and \eqref{bz6} that
\be\la{ax40}\ba
&\left(\frac{\sigma^{m}}{2}\|\rho^{1/2}\dot{u}\|_{L^2}^2\right)_t+(2\mu+\lambda)\sigma^{m}\|\div\dot{u}\|_{L^2}^2+\mu\sigma^{m}\|\curl\dot{u}\|_{L^2}^2\\
&\le -\left(\int_{\p \O}  \sigma^m (u \cdot \na n \cdot u) G dS\right)_t+ C(\on) \de \si^m \|\na \dot{u}\|_{L^2}^2 \\
&\quad+ C(\de,\on) \si^m \|\rho^{1/2} \dot \te\|_{L^2}^2 + C(\de, \on, M)(\si^{m-1}\si'+\si^m)  \|\rho^{1/2} \dot u\|_{L^2}^2\\
&\quad+C(\de, \on, M) \|\na u\|^2_{L^2} +C(\de)\si^m \|\na u\|^4_{L^4}+ C(\de,\on) \si^m  \|\te \na u\|_{L^2}^2.
\ea\ee
Applying  \eqref{tb11} to \eqref{ax40} and choosing $\de$ small enough infer \eqref{ae0} directly.

Finally, we will prove   (\ref{nle7}).

For $m\ge 0,$
multiplying $(\ref{a1})_3 $ by $\sigma^m \dot\te$ and integrating
the resulting equality over $\Omega $ yield  that
 \be\la{e1} \ba &\frac{\ka
{\sigma^m}}{2}\left( \|\na\te\|_{L^2}^2\right)_t+\frac{R\sigma^m}{\ga-1} \int\rho|\dot{\te}|^2dx
\\&=-\ka\sigma^m\int\na\te\cdot\na(u\cdot\na\te)dx
+\lambda\sigma^m\int  (\div u)^2\dot\te dx\\&\quad
+2\mu\sigma^m\int |\mathfrak{D}(u)|^2\dot\te dx-R\si^m\int\n\te \div
u\dot\te dx \triangleq \sum_{i=1}^4I_i . \ea\ee

 First,   combining   (\ref{g1}) and (\ref{z1}) gives
\be\la{e2} \ba
I_1
&\le C\sigma^{m}   \|\na u\|_{L^2}\|\na\te\|^{1/2}_{L^2}
\|\na^2\te\|^{3/2}_{L^2}  \\
&\le  \de\sigma^{m} \|\n^{1/2}\dot\te\|^2_{L^2}+\si^m \left(\|\na u\|_{L^4}^4+\|\te\na u\|_{L^2}^2\right) +C(\de,\on,M)\sigma^{m}
\|\na\te\|^2_{L^2}   ,\ea\ee
where in the last inequality we have used the following estimate:
\be  \la{lop4}\ba
	\|\na^2\te\|_{L^2} &\le C (\on)\left(\|\n^{1/2}\dot \te\|_{L^2}+ \|\na u\|_{L^4}^2+\|\te\na u\|_{L^2}\right),
	\ea\ee
which is derived from the standard $L^2$-estimate to the following elliptic problem:
	\be\la{3.29}\begin{cases}
		\ka\Delta \te=\frac{R}{\ga-1}\n\dot\te +R\n\te\div
		u-\lambda (\div u)^2-2\mu |\mathfrak{D}(u)|^2 ,\\
		\na \theta\cdot n|_{\p\O \times (0,T)}=0 .
	\end{cases}\ee

Next, it holds  that  for any $\eta\in (0,1],$
\be\la{e3}\ba I_2 =&\lambda\si^m\int (\div u)^2 \te_t
dx+\lambda\si^m\int (\div u)^2u\cdot\na\te
dx\\=&\lambda\si^m\left(\int (\div u)^2 \te
dx\right)_t-2\lambda\si^m \int \te \div u \div (\dot u-u\cdot\na u)
dx\\&+\lambda\si^m\int (\div u)^2u\cdot\na\te
dx \\=&\lambda\si^m
\left(\int (\div u)^2 \te dx\right)_t-2\lambda\si^m\int \te \div u
\div \dot udx\\&+2\lambda\si^m\int \te \div u \pa_i u^j\pa_j  u^i dx
+ \lambda\si^m\int u \cdot\na\left(\te   (\div u)^2 \right)dx
 \\
\le &\lambda\left(\si^m\int (\div u)^2 \te dx\right)_t-\lambda
m\si^{m-1}\si'\int (\div u)^2 \te dx\\& +\eta\si^m\|\na \dot
u\|_{L^2}^2+C(\eta)\si^m\|\te\na u\|_{L^2}^2+C\si^m\|\na u\|_{L^4}^4,\ea\ee
and
 \be \la{e5}\ba I_3&\le 2\mu\left(\si^m\int
|\mathfrak{D}(u)|^2 \te dx\right)_t-2\mu m\si^{m-1}\si'\int
|\mathfrak{D}(u)|^2 \te dx
 \\&\quad+ \eta\si^m\|\na \dot
u\|_{L^2}^2+C(\eta)\si^m\|\te\na u\|_{L^2}^2+C\si^m\|\na u\|_{L^4}^4 .    \ea\ee

Finally, Cauchy's inequality gives
 \be\la{e39}\ba
 |I_4|    \le  \delta \si^m \int \n |\dot\te|^2dx+C( \delta,\on)\si^m \|\te\na u\|_{L^2}^2.  \ea\ee

Substituting  (\ref{e2}) and \eqref{e3}--(\ref{e39}) into (\ref{e1}), we obtain \eqref{nle7}
after using  \eqref{h3} and choosing $\de$ suitably
small.
The proof of Lemma \ref{a113.4} is completed.
\end{proof}

With the estimates \eqref{an1}--\eqref{nle7} (see Lemma  \ref{a113.4}) at hand, we are now in a position to prove the following estimate on $A_3(T).$

\begin{lemma}\la{le6} Under the conditions of Proposition \ref{pr1},   there exists a positive constant $\ve_2$
depending only on   $\mu,\,\lambda,\, \ka,\, R,\, \ga,\, \on,\,\bt,\,\O,$ and $M$
such that if $(\rho,u,\te)$ is a smooth solution to the problem (\ref{a1})--(\ref{h1})  on $\Omega\times (0,T] $ satisfying      (\ref{z1}) with $K$ as
in Lemma \ref{le2}, the following estimate holds: \be\la{b2.34}  A_3(T) \le C_0^{1/6},\ee provided $C_0\le
\ve_2.$
\end{lemma}

\begin{proof}
	First, it follows from \eqref{an2}, (\ref{h18}), and (\ref{p}) that
	\be\ba \notag B_1(t)
	&\ge  C   \|\na u\|_{L^2}^2- C \|P-\bp\|^2_{L^2}  \ge C  \|\nabla u\|_{L^2}^2-C(\on) C_0^{1/4},\ea\ee
	which together with \eqref{an1} and \eqref{z1} implies that
	\begin{align}\label{dd}
		\sup_{t\in(0,T]}\left(\si \|\na u\|_{L^2}^2\right) + \int_0^T\int \si\n |\dot u|^2 dx dt \le C(\hat{\rho}, M)C_0^{1/4}.
	\end{align}

	For $C_2$ as in  (\ref{ae0}), adding  (\ref{nle7}) multiplied by $C_2+1$ to (\ref{ae0}) and choosing $ \eta$ suitably small give
	\be\la{e8}\ba
	&\left(\sigma^m\varphi\right)'(t) + \sigma^m \int\left(\frac{C_1}{2} |\nabla\dot{u}|^2 +\n|\dot \te|^2\right)dx\\
	&\le - 2\left(\int_{\p \O}  \sigma^m( u \cdot \na n \cdot u )G dS\right)_t + C( \on, M) (\si^{m-1}\si'+\si^m ) \|\rho^{1/2} \dot u\|_{L^2}^2\\
	&\quad+ C(\on,  M) (\|\na u\|_{L^2}^2 + \|\na \te\|_{L^2}^2)+C\si^m \|\na u\|^4_{L^4} + C (\on)\si^m \|\te\na u\|_{L^2}^2,\\
	\ea\ee
	where  $\varphi(t)$ is defined by
	\be\la{wq3} \varphi(t) \triangleq   \|\n^{1/2}\dot u\|_{L^2}^2+(C_2+1)  B_2(t).\ee
	Then it follows from (\ref{e6}) that
	\be\la{wq2}  \varphi(t) \ge  \frac{1}{2}\|\n^{1/2}\dot u\|_{L^2}^2 +\frac{\ka(\ga-1)}{2R}\|\na\te\|_{L^2}^2-C(\on,  M)\|\na u\|_{L^2}^2, \ee
	where one has used that for any $\de\in(0,1]$,
	\be \la{2.48}\ba   \int
	\te |\na u|^2dx   & \le
	C\int|R\te-\bp||\na u|^2dx+C\bp \int |\na u|^2dx\\ &\le C
	\|R\te-\bp \|_{L^6}\|\na u\|_{L^2}^{3/2}
	\|\na u\|_{L^6}^{1/2}+   C \|\na u\|_{L^2}^2
	\\ &\le C(\on)\|\na\te\|_{L^2}\|\na u\|_{L^2}^{3/2}
	\left(\|\n^{1/2} \dot u\|_{L^2}+\|\na u\|_{L^2}+ \|\na\te\|_{L^2}+1\right)^{1/2}\\
	&\quad+   C \|\na  u\|_{L^2}^2\\ &\le \de
	\left(  \|\na\te\|^2_{L^2}  + \|\n^{1/2}  \dot
	u\|_{L^2}^2 \right) + C(\de,\on,M)  \|\na
	u\|_{L^2}^2 \ea\ee
	due to  \eqref{key}, (\ref{pq}), (\ref{3.30}),   and (\ref{z1}).

	Next, it follows from \eqref{key},  (\ref{pq}), \eqref{z1}, and  (\ref{3.30}) that
	\be \la{m20}\ba
	\|\te\na u\|_{L^2}^2
	&\le C\|R\te-\bp \|_{L^6}^2 \|\na u\|_{L^2} \|\na u\|_{L^6} + C\bp^2 \|\na u\|_{L^2}^2 \\
	&\le C(\on,M) \left( \|\na u\|_{L^2}^2+\|\na\te\|_{L^2}^2\right)\left( \|\n^{1/2}\dot u\|^2_{L^2} +
	\|\na\te\|^2_{L^2} +1\right).\ea\ee
	And, by virtue of (\ref{h17}), (\ref{pq}), and (\ref{z1}), one gets
	\be\la{ae9}\ba
	&\|\na u\|_{L^4}^4\\
	&\le C\|\n\dot u\|_{L^2}^3(\|\na u\|_{L^2}+\|P-\bp\|_{L^2})+C(\|\na u\|_{L^2}^4+\|P-\bp\|_{L^4}^4)\\
	&\le C(\on)\|\n^{1/2}\dot u\|_{L^2}^3\left(\|\na u\|_{L^2}+1\right) +C(\on)\|\na\te\|_{L^2}^3+C\|\n-1\|_{L^4}^4 +C\|\na u\|_{L^2}^4\\
	&\le C(\on,M)\left( \|\n^{1/2}\dot u\|_{L^2}^3 + \|\na \te\|_{L^2}^3\right)  +C(\on)\|\n-1\|_{L^2}^2+C(\on,M)\|\na u\|_{L^2}^2,
	\ea\ee
	which together with (\ref{z1})  yields
	\be \la{m22}\ba
	\si \|\na u\|_{L^4}^4   &\le C(\on,M) \left(\|\n^{1/2}\dot u\|_{L^2}^2  + \|\na \te\|_{L^2}^2+\|\na u\|_{L^2}^2\right)
	+C(\on)\si\|\n-1\|_{L^2}^2.
	\ea\ee
Thus, taking $m=2$ in \eqref{e8}, one obtains after  using  \eqref{z1}, \eqref{m20}, and \eqref{m22}  that
\be\la{e25}\ba
& \left(\sigma^2\varphi\right)'(t) + \sigma^2 \int\left(\frac{C_1}{2} |\nabla\dot{u}|^2 +\n|\dot \te|^2\right)dx\\
&\le  - 2\left(\int_{\p \O}  \sigma^2 (u \cdot \na n \cdot u) G dS\right)_t+ C(\on,M)\si\|\n^{1/2}\dot u\|_{L^2}^2\\
&\quad +  C(\on,  M)\left(\|\na u\|_{L^2}^2+\|\na\te\|_{L^2}^2\right)+C(\on)\si\|\n-1\|_{L^2}^2.
	\ea\ee

	Now, we  deduce from \eqref{h19},  \eqref{z1}, and \eqref{b2} that
	\be \ba\label{jia4}
	\sup_{ 0\le t\le T} \left|\int_{\p \O}  \sigma^2 u \cdot \na n \cdot u G dS\right|
	\le & C(\on) \sup_{ 0\le t\le T} (\si \|\na u\|_{L^2}^2) \sup_{ 0\le t\le T} (\si \|\rho^{1/2} \dot u\|_{L^2})\\
	\le &  C(\on) C_0^{1/4}.
	\ea \ee
	Furthermore, note that (\ref{hj1}) is equivalent to
 \be\notag
\bp(\n-1)=-G+(2\mu+\lambda)\div u-\n(R\te-\bp),
\ee
 this together with \eqref{key}, \eqref{z1}, \eqref{1h19}, \eqref{pq}, and \eqref{dd} implies
\be\ba\la{67}
&\int_0^T\si\|\n-1\|_{L^2}^2 dt\\
&\le C\int_0^T\si(\|G\|_{L^2}^2+\|\na u\|_{L^2}^2) dt+C(\on)\int_0^T\|R\te-\bp\|_{L^2}^2 dt\\
&\le C(\on)\int_0^T\left(\si\|\n^{1/2}\dot u\|_{L^2}^2+\|\na u\|_{L^2}^2+\|\na \te\|_{L^2}^2\right) dt\\
&\le C(\on,M)C_0^{1/4}.
\ea\ee
	
	Thus, integrating \eqref{e25} over $(0,T)$, one obtains after using \eqref{dd}, \eqref{wq2}, (\ref{jia4}), (\ref{67}), and (\ref{z1}) that
	\be\ba\notag
	A_3(T)\le C(\on,M)C_0^{1/4}\le C_0^{1/6},
	\ea\ee
 provided
	\be \ba\notag C_0\le\ve_2\triangleq \min\{1,(C(\on,M))^{-12}\}.\ea\ee

  The proof of Lemma \ref{le6} is completed.
 \end{proof}



Next, in order to control $A_2(T)$, we first re-establish the basic energy estimate for short time $[0, \si(T)]$, and then show that the spatial $L^2$-norm of $R\te-\overline P$ could be  bounded by the combination of the initial energy and the spatial $L^2$-norm of $\na \te$, which is indeed the key ingredient to   estimate   $A_2(T)$.

\begin{lemma}\la{a13} Under the conditions of Proposition \ref{pr1},
	there exist  positive constants $C$ and $\varepsilon_{3,1}$
	depending only on
	$\mu,\,\lambda,\, \ka,\, R,\, \ga,\, \on,\,\bt,\,\O, $ and $M$ such
	that  if $(\rho,u,\te)$ is a smooth solution to the problem (\ref{a1})--(\ref{h1})  on $\Omega\times (0,T] $ satisfying      (\ref{z1}) with $K$ as
in Lemma \ref{le2},
	the following estimates  hold:
	\be \la{a2.121} \ba
	&\sup_{0\le t\le \si(T)}\int\left( \n |u|^2+(\n-1)^2 + \n(\te-\log \te-1) \right)dx\le C C_0,\ea\ee
and 
\be  \la{a2.17}  \ba
\|(R\te-\bp)(\cdot,t)\|_{L^2} \le C \left(C_0^{1/2} +C_0^{1/3}\|\na\te(\cdot,t)\|_{L^2}\right),
	\ea\ee
for all $t\in(0,\si(T)]$, provided $C_0\le\varepsilon_{3,1}.$
	\end{lemma}
\begin{proof}
	The proof is divided into the following two steps.
	
{\it Step 1: The proof of (\ref{a2.121}).}
	
	First, multiplying (\ref{a11}) by $u$, one deduces from integration by parts, \eqref{a1}$_1$, and \eqref{cz1} that
	\be \la{a2.12} \ba
	&\frac{d}{dt}\int\left(\frac{1}{2}\n |u|^2+R(1+\n\log \n-\n)
	\right)dx+ \int(\mu|\curl u|^2+(2\mu+\lambda)(\div u)^2)dx \\
	&=  R \int \rho (\te -1) \div u dx\\
	&\le \de \|\na u\|_{L^2}^2 + C(\de, \on) \int \n(\te-1)^2dx\\
&\le\de \|\na u\|_{L^2}^2 + C(\de, \on)(\|\te(\cdot,t)\|_{L^{\infty}} +1)  \int \rho (\te -\log \te -1)dx.
	\ea\ee
Using \eqref{h18} and choosing $\de$ small enough in \eqref{a2.12}, it holds that
	\be \la{a2.222} \ba
	&\frac{d}{dt}\int\left(\frac{1}{2}\n |u|^2+R(1+\n\log \n-\n)
	\right)dx + C_3 \int|\na u|^2dx \\
&\le C(\on) (\|\te(\cdot,t)\|_{L^{\infty}} +1)  \int \rho  (\te -\log \te -1)dx.
	\ea\ee
Then, adding \eqref{a2.222} multiplied by $(2\mu+1) {C_3}^{-1}$  to \eqref{la2.7}, one has
	\be \la{a2.22} \ba
	&
	\xl((2\mu+1) {C_3}^{-1}+1\xr)\frac{d}{dt}\int\left(\frac{1}{2}\n |u|^2+R(1+\n\log \n-\n)\right)dx
	\\&+ \frac{R}{\ga-1} \frac{d}{dt} \int \n(\te-\log \te-1)dx+\int|\na u|^2dx\\
	 &\le C(\on) (\|\te(\cdot,t)\|_{L^{\infty}} +1)  \int \rho (\te -\log \te -1)dx.
	\ea\ee
	
Next, we claim that
	\be \ba\la{k}
	\int_0^{\si(T)}\|\te\|_{L^\infty}dt \le C(\on, M).
	\ea \ee
	Combining this with \eqref{a2.22}, \eqref{a2.9}, and Gr\"onwall inequality  implies \eqref{a2.121} directly.
	
	Finally, it remains to prove \eqref{k}. 
	Taking $m=1$ in \eqref{e8} and integrating the resulting inequality, one deduces from \eqref{wq2}, \eqref{m20}, \eqref{m22}, (\ref{z1}), \eqref{67}, \eqref{h19}, and \eqref{b2} that
\bnn  \ba
& \sigma \varphi +  \int_0^t \sigma\int\left(\frac{C_1}{2} |\nabla\dot{u}|^2 +\n|\dot \te|^2\right)dxd\tau\\
&\le  2\sigma\left|\int_{\p \O}   (u \cdot \na n \cdot u )G dS\right|(t)
      + C(\on,M)\int_0^t (\|\n^{1/2}  \dot u\|_{L^2}^2 +\|\na u\|_{L^2}^2+\|\na \te\|_{L^2}^2)d\tau \\
&\quad+  C(\on)\int_0^t \si \|\n-1\|_{L^2}^2d\tau+C(\on)\int_0^t \left(\|\na u\|_{L^2}^2+\|\na \te\|_{L^2}^2\right) \si\varphi d\tau\\
& \le  C(\on) (\sigma \|\na u\|_{L^2}^2 \|\rho^{1/2} \dot u\|_{L^2})(t)+ C(\on,M)\\
&\quad+C(\on,M)\int_0^t \left(\|\na u\|_{L^2}^2+\|\na \te\|_{L^2}^2\right) \si\varphi d\tau\\
& \le C(\on,M)+C(\on,M)\int_0^t \left(\|\na u\|_{L^2}^2+\|\na \te\|_{L^2}^2\right) \si\varphi d\tau.
\ea\enn
Then Gr\"onwall inequality together with (\ref{z1}) and \eqref{wq2} yields
\be\ba\label{ae26}
	\sup_{0\le t\le T}\si \left(\int\rho|\dot{u}|^2dx+\|\na\te\|_{L^2}^2\right)+\int_0^T\si \int\left(|\na\dot u|^2+\n|\dot\te|^2\right)dxdt
\le  C(\on,M).\ea\ee


	 Next, it follows from (\ref{lop4}), \eqref{m22}, \eqref{m20},  \eqref{ae26}, (\ref{z1}), and (\ref{67})  that
	\be\ba  \la{k1}
	\int_0^{T }\si \|\na^2\te\|_{L^2}^2dt
	&\le   C(\on,M)\int_0^{T } \left(\si \|\n^{1/2}\dot\te\|_{L^2}^2+
	\|\n^{1/2}\dot u \|_{L^2}^2  \right) dt\\
	&\quad+ C(\on,M)\int_0^{T } \left( \| \na u\|_{L^2}^2+\| \na\te\|_{L^2}^2+\si\|\n-1\|_{L^2}^2\right) dt \\
	&\le C(\on,M).
	\ea\ee
	Furthermore, one deduces from \eqref{g2}, \eqref{g1}, and \eqref{pq} that
	\begin{equation}
	    \label{6yue}\ba
	    \|R\te-\bp\|_{L^\infty} \le& C \|R\te-\bp\|_{L^6}^{1/2} \|\na\te\|_{L^6}^{1/2} + \|R\te-\bp\|_{L^2}\\
	    \le& C(\hat{\rho}) \|\na \te\|_{L^2}^{1/2} \|\na^2 \te\|_{L^2}^{1/2} + C(\hat{\n}) \|\na \te\|_{L^2},
	    \ea
	\end{equation}
	which together with  (\ref{z1}) and \eqref{k1}  gives that
	\be\la{3.88}\ba
	& \int_0^{\si(T)}\|R\te-\bp\|_{L^\infty}dt \\
	&\le C(\hat{\rho}) \int_0^{\si(T)}\|\na\te\|_{L^2}^{1/2} \left(\si\|\na^2\te\|^2_{L^2}\right)^{1/4}\si^{-1/4}dt + C(\hat{\rho}) \left(\int_0^{\si(T)} \|\na\te\|_{L^2}^2 dt\right)^{1/2}\\
	&\le C(\hat{\rho}) \left(\int_0^{\si(T)} \|\na \te\|_{L^2}^2dt \int_0^{\si(T)}\si\|\na^2\te\|_{L^2}^2dt\right)^{1/4} 
+ C(\on)C_0^{1/8}\\
	&\le C(\on,M)C_0^{1/16}.
	\ea\ee
 Combining this  with \eqref{key} yields \eqref{k} directly.

{\it Step 2: The proof of (\ref{a2.17}).}

Direct calculations together with \eqref{cz1} lead to
  \be\notag\ba
 \te-\log\te-1
\ge \frac{1}{8} (\te-1)1_{(\te(\cdot,t)>2)
}+\frac{1}{12}(\te-1)^21_{(\te(\cdot,t)<3)},  \ea\ee with
$(\te(\cdot,t)> 2)\triangleq \left.\left\{x\in
\Omega\right|\te(x,t)> 2\right\}$ and  $(\te(\cdot,t)< 3)\triangleq
\left.\left\{x\in \Omega\right|\te(x,t)<3\right\}.$
   Combining this with \eqref{a2.121} gives
	\be \la{a2.11}\ba
	\sup_{0\le t\le \si(T)}\int \left(\n(\te-1)1_{(\te(\cdot,t)>2)}+\n(\te-1)^21_{(\te(\cdot,t)<3)}\right)dx \le C(\hat\n,M) C_0.
	\ea\ee
	
	Next, it follows from \eqref{a2.11}, \eqref{a2.121}, and the Sobolev inequality that for $t \in (0,\sigma(T)]$,
	\begin{equation}\la{la2.19}
	\begin{aligned}
	 &\|\te -1\|_{L^2(\te(\cdot,t)<3)}^2 \\
	 &\le\int  \n (\te-1)^2 1_{(\te(\cdot,t)<3)}dx  +  \left|\int  (\n-1) (\te -1)^2 dx \right|\\
	&\le C(\on,M)C_0  +  C\|\n-1\|_{L^2} \| \te-1 \|_{L^2}^{1/2} \|\te -1\|_{L^6}^{3/2}\\
	&\le C(\on,M)C_0   +  C(\on,M)C_0^{1/2}  \| \te-1 \|_{L^2}^{1/2} \left(\| \te-1 \|_{L^2}+ \|\na \te\|_{L^2}\right)^{3/2}\\
	&\le C(\on,M)\left(C_0+ C(\delta) C_0^{2/3} \|\na \te\|_{L^2}^2 + (\delta+C_0^{1/2}) \| \te-1 \|_{L^2}^2 \right),
	\end{aligned}
	\end{equation}
	and
	\be\la{a2.18}\ba
	 &\|\te-1\|_{L^2(\te(\cdot,t)> 2)}^2\\
	&\le \|\te-1\|^{4/5}_{L^1(\te(\cdot,t)> 2)} \| \te-1\|_{L^6}^{6/5}\\
	&\le C(\on,M) \left(C_0 +C_0^{1/2} \|\te-1\|_{L^2}\right)^{4/5} (\| \te-1\|_{L^2}+ \|\na \te\|_{L^2})^{6/5}\\
	&\le  C(\on,M)\left(C_0+C(\delta)C_0^{2/3}\|\na \te\|_{L^2}^2 + (\delta+C_0^{2/5} ) \|\te -1\|_{L^2}^2\right),
	\ea\ee
	where in the second inequality one has used
	\be \notag\ba
	\|\te-1\|_{L^1(\te(\cdot,t)>2)} \le &\int  \n (\te-1) 1_{(\te(\cdot,t)>2)}dx  +   \int  \left|(\n-1) (\te -1) \right|dx \\
	\le & C(\on,M)(C_0 +C_0^{1/2} \|\te-1\|_{L^2}).
	\ea\ee
	Hence, adding \eqref{la2.19} with \eqref{a2.18} together and choosing $\delta$ small enough in the resulting inequality, one has   for any $t\in (0,\sigma(T)],$
	\be\ba\notag
	\|\te-1\|_{L^2}^2 \le  C(\on,M)\left(C_0+C_0^{2/3}\|\na \te\|_{L^2}^2 + C_0^{2/5}  \|\te -1\|_{L^2}^2\right),
	\ea\ee
	which implies that
	\be\ba\la{nnn1}
	\|\te-1\|_{L^2}^2 \le  C(\on,M)\left(C_0+C_0^{2/3}\|\na \te\|_{L^2}^2 \right),
	\ea\ee
	provided
	\begin{equation}\label{31}
	C_0\le \ve_{3,1} \triangleq\min\left\{1,(2C(\on,M))^{-5/2}\right\}.
	\end{equation}
	
	Finally, note that
	\be \ba\notag
	\| R\te-\bp \|_{L^2} &\le R \|\te  -1\|_{L^2}	+C|1-\overline{\n\te}|\\
	&\le R \|\te  -1\|_{L^2}+C\left|\int\n(1-\te)dx\right|\\
	&\le C(\on) \|\te  -1\|_{L^2},
	\ea\ee
	this together with \eqref{nnn1} yields (\ref{a2.17}).
	
	The proof of Lemma \ref{a13}  is completed.
\end{proof}


Next, with the help of \eqref{a2.17}, the estimate on $A_2(T)$ will be handled smoothly.

\begin{lemma}\la{le3}
Under the conditions of Proposition \ref{pr1},  there exists a positive constant $\ve_3$ depending only on $\mu,\,\lambda,\, \ka,\, R,\, \ga,\, \on,\,\bt,\,\O$, and $M$ such that if $(\rho,u,\te)$ is a smooth solution to the problem (\ref{a1})--(\ref{h1})  on $\Omega\times (0,T] $ satisfying  (\ref{z1})   with $K$ as in Lemma \ref{le2}, the following estimate holds:
\be\la{a2.34} A_2(T) \le C_0^{1/4},\ee
provided $C_0\le \ve_3.$
\end{lemma}

\begin{proof}
To begin with, multiplying (\ref{a11}) by $u$ and integrating by parts give that
\be \la{a2.225}\ba
&\frac{d}{dt}\int\left(\frac{1}{2}\n |u|^2+\bp(1+\n\log \n-\n)\right)dx\\
&\quad+ \int\left(\mu|\curl u|^2+(2\mu+\lambda)(\div u)^2\right)dx \\
&= \bp_t \int  (1+\n\log \n-\n) dx +  \int \rho (R \te -\bp) \div u dx.
\ea\ee
Next, multiplying $(\ref{a1})_3$ by $\bp^{-1}(R\te-\bp)$, one obtains after  integrating the resulting
equality over $\Omega $ by parts that
\be\la{a2.231} \ba
& \frac{1}{2(\ga-1)} \frac{d}{dt}\int
{\bp}^{-1}\n {(R\te-\bp)^2}dx+ {\ka}R \bp^{-1}  \|\na\te\|_{L^2}^2\\
&= -  \frac{1}{\ga-1} {\bp}^{-1} \bp_t \int
\n {(R\te-\bp)}  dx - \frac{1}{2(\ga-1)} \bp^{-2} \bp_t \int \n {(R\te-\bp)^2}  dx\\
& \quad- {\bp}^{-1} \int \n {(R\te-\bp)^2} \div u dx -  \int \rho (R\te -\bp) \div u dx \\
&\quad+ {\bp}^{-1} \int  {(R\te-\bp)} (\lambda (\div u)^2+2\mu |\mathfrak{D}(u)|^2) dx.
\ea\ee
Adding \eqref{a2.225} and \eqref{a2.231} together yields that
\be\la{a2.23} \ba
&  \frac{d}{dt}\int
\left(\frac{1}{2}\n |u|^2+\bp (1+\n\log \n-\n)+\frac{1}{2(\ga-1)} \n \bp^{-1}{(R\te-\bp)^2} \right)dx\\
&\quad+ \mu \|\curl u\|_{L^2}^2 +(2\mu+\lambda)\|\div u\|_{L^2}^2+  \ka R \bp^{-1}  \|\na\te\|_{L^2}^2\\
&= -  \frac{1}{\ga-1} {\bp}^{-1} \bp_t \int
\n {(R\te-\bp)}  dx - \frac{1}{2(\ga-1)} {\bp}^{-2} \bp_t \int \n {(R\te-\bp)^2} dx\\
& \quad+\bp_t \int  (1+\n\log \n-\n) dx - \bp^{-1}\int  \n {(R\te-\bp)^2} \div u dx \\
&\quad+ \bp^{-1} \int  {(R\te-\bp)} (\lambda (\div u)^2+2\mu |\mathfrak{D}(u)|^2) dx\triangleq \sum_{i=1}^{5} J_i.
\ea\ee

The terms $J_i \,(i=1,\cdots,5)$ can be estimated as follows.

It follows from \eqref{key}, \eqref{511}, \eqref{z1}, and \eqref{pq} that
\be \ba\la{aa1}
J_1+J_2 \le & C|\overline P_t|\left(\| \n (R\te-\bp)\|_{L^2}+\| \n^{1/2} (R\te-\bp)\|_{L^2}^2\right)\\
\le &C(\on)\left(C_0^{1/8} \|\na u \|_{L^2}+\| \na u\|_{L^2}^2 \right)\|\n^{1/2}(R\te-\bp)\|_{L^2}\\
\le & C(\on) C_0^{1/8} \left( \|\na u\|_{L^2}^2+\| \na \te \|_{L^2}^2\right),
\ea \ee
and
\be \ba\la{aa4}
J_4 \le & C   \|\n^{1/2} {(R\te-\bp)}\|_{L^2}^{1/2} \|\n^{1/2} {(R\te-\bp)}\|_{L^6}^{3/2} \|\na u\|_{L^2} \\
\le & C(\on) A_2^{1/4}(T) \|\na \te\|_{L^2}^{3/2} \|\na u\|_{L^2}\\
\le & C(\on,M) C_0^{1/16} (\|\na u\|_{L^2}^2+\|\na \te\|_{L^2}^{2}).
\ea \ee
Furthermore, by virtue of \eqref{a2.9}, \eqref{511}, and \eqref{a2.112}, we have
\be \ba\la{aa3}
J_3 \le & |\overline P_t|\left|\int (1+\n \log\n-\n )dx\right|\\
\le & C(\on)\left(  C_0^{1/8}\|\na u \|_{L^2}+\| \na u\|_{L^2}^2 \right)\| \n -1\|_{L^2}^2\\
\le & C(\on) C_0^{1/4}  \|\na u\|_{L^2}^2 + C(\on) C_0^{1/4} \| \n -1\|_{L^2}^2.
\ea \ee


Now, we will estimate the term $J_5$ for the short time $t\in [0, \si(T))$ and the large time $t\in [\si(T),T]$, respectively.

For $t\in[0, \si(T))$, it follows from \eqref{key}, \eqref{3.30}, \eqref{pq}, \eqref{a2.17}, and \eqref{z1} that
\be \ba\la{aa5}
J_5 \le& C \int  |R\te-\bp| |\na u|^2dx\\
\le & C \| {R\te-\bp}\|_{L^2}^{1/2} \|R\te-\bp\|_{L^6}^{1/2}  \|\na u\|_{L^2} \|\na u\|_{L^6}\\
\le & C(\on)  \|{R\te-\bp}\|_{L^2}^{1/2} \|\na \te\|_{L^2}^{1/2}  \|\na u\|_{L^2}\\
 & \left(\|\n^{1/2}\dot u\|_{L^2}+\|\na u\|_{L^2}+\|\na\te\|_{L^2}+C_0^{1/24}\right)\\
\le & C (\on) \| {R\te-\bp}\|_{L^2}^{1/2} \|\na \te\|_{L^2}^{1/2}  \|\na u\|_{L^2}\|\n^{1/2}\dot u\|_{L^2}\\
&+C(\on,M) C_0^{1/24} (\|\na \te\|_{L^2}^{2}+  \|\na u\|_{L^2}^2)\\
\le & C(\on,M) C_0^{7/24} \|\rho^{1/2}\dot u\|_{L^2}^2 + C(\on,M)C_0^{1/24} (\|\na u\|_{L^2}^2 +  \|\na \te\|_{L^2}^2),
\ea \ee
where we have used following calculations:
\be \ba\notag
&\| {R\te-\bp}\|_{L^2}^{1/2} \|\na \te\|_{L^2}^{1/2}  \|\na u\|_{L^2} \|\n^{1/2}\dot u\|_{L^2} \\
&\le C(\on,M)(C_0^{1/4}\|\na\te\|_{L^2}^{1/2}+C_0^{1/6}\|\na\te\|_{L^2})\|\na u\|_{L^2}\|\n^{1/2}\dot u\|_{L^2}\\
&\le C(\on,M) C_0^{7/24} \|\rho^{1/2}\dot u\|_{L^2}^2 + C(\on,M)C_0^{1/12}  \|\na u\|_{L^2}^2 + C(\on,M)C_0^{1/24} \|\na \te\|_{L^2}^2
\ea \ee
owing to \eqref{a2.17}.

For $t\in[\si(T), T]$, it holds that
\be \ba\la{aa6}
J_5 
\le &C\|R\te-\bp\|_{L^3}\|\na u\|_{L^2}\|\na u\|_{L^6}
\le  C(\on)C_0^{1/24}(\|\na u\|_{L^2}^2+\|\na \te\|_{L^2}^2),
\ea \ee
where one has used \eqref{pq} and the following fact:
\be\notag\ba \sup_{0\le t\le T}\xl(\si\|\na
u\|_{L^6}\xr)
&\le C(\on)C_0^{1/24} \ea\ee
due to \eqref{z1} and \eqref{3.30}.

Finally, substituting \eqref{aa1}--\eqref{aa6}  into \eqref{a2.23}, one obtains after using \eqref{z1}, \eqref{key}, and \eqref{jia10} that
\be \ba  \label{jia9}
&  \sup_{ 0\le t\le T} \int \left(\frac{1}{2}\n |u|^2+\bp (1+\n\log \n-\n)+\frac{1}{2(\ga-1)} \n \bp^{-1}{(R\te-\bp)^2}
\right)dx\\
&+ \int_{0}^{T} (\mu \|\curl u\|_{L^2}^2 +(2\mu+\lambda)\|\div u\|_{L^2}^2+ {R\ka}{\bp}^{-1}  \|\na\te\|_{L^2}^2)dt\\
&\le  C(\on,M) C_0^{1/24} \int_{ 0}^{T}  ( \|\na u\|_{L^2}^2+\| \na \te \|_{L^2}^2 ) dt+ C(\on) C_0^{1/4}\int_{ 0}^{T} \| \n -1\|_{L^2}^2 dt
\\&\quad + C(\on,M) C_0^{7/24} \int_{0}^{\si(T)} \|\rho^{1/2}\dot u\|_{L^2}^2 dt+C(\on,\bt)C_0\\
&\le  C(\on,\bt, M) C_0^{7/24},
\ea \ee
where one has used
\be\la{xx}
\int_{ 0}^{T} \| \n -1\|_{L^2}^2 dt \le \sup_{0\le t\le\si(T)} \|\n -1\|_{L^2}^2 + \int_{\si(T)}^{T} \| \n -1\|_{L^2}^2 dt \le C(\on,M) C_0^{1/4}
\ee
due to \eqref{a2.121} and \eqref{67}. Thus, one deduces from \eqref{jia9}, \eqref{h18}, and \eqref{key} that
\be \la{kyu1} A_2(T)\le C(\on,\hat{\theta}, M) C_0^{7/24}  \ee which implies \eqref{a2.34}
provided \be \ba \notag C_0\le \ve_3\triangleq\min
\left\{\ve_{3,1},  (C(\on,\hat{\theta},M))^{-24}\right\},\ea\ee
with $\ve_{3,1}$ as in \eqref{31}. The proof
of Lemma \ref{le3} is completed.
\end{proof}

\begin{remark}\la{r2}
	It's worth noticing that the energy-like estimate $A_2(T)$ is a little subtle, since $A_2(T)$ is not a conserved quantity for the full Navier-Stokes system owing to the nonlinear coupling of $\te$ and $u$. Thus, further consideration is needed to handle this issue. 
	More precisely,
	\begin{itemize}
		\item on the one hand, while deriving the kinetic energy (see (\ref{a2.225})), we need to deal with the following term
		$$\int \rho (R\te -\bp) \div u dx.$$
		Unfortunately, this term is troublesome for large time $t\in [\sigma(T), T]$. In fact, this term could be bounded by  $\|\na \te \|_{L^2}\|\na u \|_{L^2}$, which will only be  of the same order as $C_0^{1/4}$ with the help of all a priori estimates (\ref{z1}). Therefore, we can not handle this term directly. Here, based on careful analysis on system (\ref{a1}), we find that this term can be cancelled by a suitable combination of kinetic energy and thermal energy, see (\ref{a2.225})--(\ref{a2.23});
		\item on the other hand, while deriving the thermal energy (see (\ref{a2.231})), we need to handle the following term
		$$\int (R\te - \bp) (\div u)^2 dx.$$
		Note that for short time $t\in [0, \sigma(T))$, the ``weaker" basic energy estimate (\ref{a2.112}) is not enough, hence it's necessary to re-establish  the basic energy estimate (\ref{a2.121}), which is obtained by the a priori $L^1(0,\sigma(T);L^\infty)$-norm of $\te$ (see (\ref{k})). Consequently, we can obtain (\ref{a2.17}) as a consequence of (\ref{a2.121}) and then handle this term for short time $t\in [0, \sigma(T))$ (see (\ref{aa5})).	
	\end{itemize}
	
	Moreover, it should be mentioned that the uniform positive lower and upper bounds  of $\bp$  also play a critical role in estimating $A_2(T)$.
\end{remark}

We now proceed to derive a uniform (in time) upper bound for the density, which turns out to be the key to obtaining all the higher order estimates and thus extending the classical solution globally.

\begin{lemma}\la{le7}
Under the conditions of Proposition \ref{pr1},  there exists a positive constant $\ve_4$ depending only on $\mu,\,\lambda,\, \ka,\, R,\, \ga,\, \on,\,\bt,\,\O$, and $M$ such that if $(\rho,u,\te)$ is a smooth solution to the problem (\ref{a1})--(\ref{h1})  on $\Omega\times (0,T] $ satisfying  (\ref{z1})   with $K$ as in Lemma \ref{le2}, the following estimate holds:
\be \la{a3.7}
\sup_{0\le t\le T}\|\n(\cdot,t)\|_{L^\infty}  \le
\frac{3\on }{2},
\ee
provided $C_0\le \ve_4$.
\end{lemma}

\begin{proof}
First, it follows from \eqref{k1}, \eqref{6yue}, and \eqref{z1} that
\be\la{3.89}\ba
\int_{\si(T)}^T\|R\te-\bp\|^2_{L^\infty}dt
\le& C(\on) \left(\int_{\si(T)}^T\|\na\te\|^2_{L^2}dt\right)^{1/2}\left(\int_{\si(T)}^T \|\na^2\te\|^2_{L^2}dt\right)^{1/2} \\
&+ C(\on) \int_{\si(T)}^T\|\na\te\|^2_{L^2}dt \\
\le& C(\on,M) C_0^{1/8}.
\ea\ee

Next, it follows from  (\ref{h19}), \eqref{tb90}, (\ref{ae26}), and (\ref{z1}) that
\be\la{3.90}\ba &\int_0^{\si(T)}\|G\|_{L^\infty}dt\\
&\le C\int_0^{\si(T)}\|\na G\|_{L^2}^{1/2} \|\na G\|_{L^6}^{1/2}dt\\
&\le C(\on)\int_0^{\si(T)}\|\n \dot u\|_{L^2}^{1/2}(\|\na\dot u\|_{L^2}+ \|\na u\|_{L^2}^2)^{1/2}dt\\
&\le C(\on)\int_0^{\si(T)}\left(\si\|\n \dot u\|_{L^2}\right)^{1/4} \left(\si\|\n \dot u\|^{2}_{L^2}\right)^{1/8} \left(\si\|\na \dot u\|^2_{L^2}\right)^{1/4}\si^{-5/8}dt\\
& \quad+ C(\on) \int_0^{\si(T)}\left(\si \|\n \dot u\|_{L^2}\right)^{1/2} \|\na u\|_{L^2} \si^{-1/2} dt\\
&\le C(\on,M)C_0^{1/48}\left(\int_0^{\si(T)} \si\|\na \dot u\|^2_{L^2} dt\right)^{1/4}\left(\int_0^{\si(T)} \si^{-5/6}dt\right)^{3/4} \\
&\quad+ C(\on,M) C_0^{1/24} \int_0^{\si(T)}
\si^{-1/2} dt\\
&\le C(\on,M)C_0^{1/48},
\ea\ee
and
 \be\la{3.91}\ba  \int_{\si(T)}^T\|G\|^2_{L^\infty}dt
  &\le C\int_{\si(T)}^T\|\na G\|_{L^2} \|\na G\|_{L^6} dt
   \\ &\le C(\on,M)\int_{\si(T)}^T\left(\|\n^{1/2} \dot u\|^2_{L^2}+
 \|\na\dot u\|_{L^2}^2 +  \|\na u\|_{L^2}^2\right)dt\\ &\le C(\on,M)C_0^{1/6}.
    \ea\ee

Denoting $ D_t\n=\n_t+u \cdot\nabla \n $ and using (\ref{hj1}), one can rewrite   $(\ref{a1})_1$  as follows
\bnn\ba
(2\mu+\lambda) D_t \n&=-\bp \n(\n-1)- \n^2(R\te-\bp)-\n G\\
&\le -\bp (\n-1)+C(\on)\|R\te-\bp\|_{L^\infty}+C(\on)\| G\|_{L^\infty},
\ea\enn
which gives
\be\la{3.92}\ba
D_t (\n-1)+\frac{\bp }{2\mu+\lambda} (\n-1)\le C(\on)\|R\te-\bp\|_{L^\infty}+C(\on)\| G\|_{L^\infty}.
\ea\ee

Finally, applying Lemma \ref{le1} with
$$y=\n-1, \quad\al=\frac{\bp}{2\mu+\lambda} ,\quad g=C(\on)\|R\te-\bp\|_{L^\infty}+C(\on)\| G\|_{L^\infty},\quad T_1=\si(T),$$
 we thus deduce from (\ref{3.92}), (\ref{3.88}), \eqref{3.89}--(\ref{3.91}), (\ref{2.34}), and \eqref{key}  that
 \bnn\ba\n
 & \le \on+1 +C\left(\|g\|_{L^1(0,\si(T))}+\|g\|_{L^2(\si(T),T)}\right) \le \on+1 +C(\on,M)C_0^{1/48} ,
 \ea\enn
 which gives \eqref{a3.7}
 provided 
  \be \notag C_0\le \ve_4\triangleq\min\left\{1,\left(\frac{\hat \n-2 }{2C(\on,M) }\right)^{48}\right\}.\ee
 The proof of Lemma \ref{le7} is completed.
\end{proof}

Next, we summarize some uniform estimates on $(\n,u,\te)$ which will be useful for higher-order ones in the next section.
\begin{lemma}\la{le8}
Under the conditions of Proposition \ref{pr1},  there exists a  positive constant    $C $     depending only   on  $\mu,\,\lambda,\, \ka,\, R$, $\ga,\, \on,\,\bt,\, \O$, and $M$  such that if $(\rho,u,\te)$  is a smooth solution to the problem (\ref{a1})--(\ref{h1}) on $\Omega\times (0,T] $  satisfying (\ref{z1}) with $K$ as in Lemma \ref{le2}, the following estimate holds:
	\be \la{ae3.7}\sup_{0< t\le T}\si^2\int \n|\dot\te|^2dx + \int_0^T\si^2 \|\na\dot\te\|_{L^2}^2dt\le C.\ee
Moreover, it holds that
\be\la{vu15}\ba
&\sup_{0< t\le T}\left(  \si\|\na u \|^2_{L^6}+\si^2\|\te\|^2_{H^2}\right)\\
&+\int_0^T(\si \|\na u \|_{L^4}^4+\si\|\na\te \|_{H^1}^2+\si\|u_t\|_{L^2}^2+\si^2\|\te_t\|^2_{H^1}+\|\n -1\|_{L^2}^2)dt\le C.
\ea\ee
\end{lemma}

\begin{proof}
First, applying the operator $\pa_t+\div(u\cdot) $ to (\ref{a1})$_3 $ and using  (\ref{a1})$_1$, one   gets
\be\la{3.96}\ba
&\frac{R}{\ga-1} \n \left(\pa_t\dot \te+u\cdot\na\dot \te\right)\\
&=\ka \Delta  \te_t +\ka \div (\Delta \te u)+\left( \lambda (\div u)^2+2\mu |\mathfrak{D}(u)|^2\right)\div u +R\n \te  \pa_ku^l\pa_lu^k\\
&\quad -R\n \dot\te \div u-R\n \te\div \dot u +2\lambda \left( \div\dot u-\pa_ku^l\pa_lu^k\right)\div u\\
&\quad + \mu (\pa_iu^j+\pa_ju^i)\left( \pa_i\dot u^j+\pa_j\dot u^i-\pa_iu^k\pa_ku^j-\pa_ju^k\pa_ku^i\right).
\ea\ee
Direct calculations show that
\be \ba\la{bea}
 \int  (\Delta  \te_t + \div (\Delta \te u)) \dot \te dx &=  - \int  (\na  \te_t \cdot \na \dot\te + \Delta \te u \cdot \na \dot \te) dx\\
&= - \int  |\na \dot\te|^2 dx  + \int ( \na(u\cdot \na \te) \cdot \na \dot \te - \Delta \te u \cdot \na \dot \te) dx.
\ea \ee
Multiplying (\ref{3.96}) by $\dot \te$ and integrating the resulting equality over $\O$, it holds that
\be\la{3.99}\ba
& \frac{R}{2(\ga-1)}\left(\int \n |\dot\te|^2dx\right)_t + \ka   \|\na\dot\te\|_{L^2}^2 \\
&\le  C  \int|\na \dot \te|\left( |\na^2\te||u|+ |\na \te| |\na u|\right)dx+C\int  \n|R\te-\bp| |\na\dot u| |\dot \te|dx\\
&\quad +C(\on)  \int|\na u|^2|\dot\te|\left(|\na u|+|R\te-\bp| \right)dx+C   \int |\na\dot u|\n|\dot \te| dx \\
&\quad +C (\on)  \int\left( |\na u|^2|\dot \te|+\n  |\dot
\te|^2|\na u|+|\na u| |\na\dot u| |\dot \te|\right)dx \\
&\le C\|\na u\|^{1/2}_{L^2}\|\na u\|^{1/2}_{L^6}\|\na^2\te\|_{L^2}\|\na \dot \te\|_{L^2}+C(\on)\|\na\te\|_{L^2} \|\na\dot u\|_{L^2} \|\dot\te\|_{L^6}\\
&\quad+C(\on)  \|\na u\|_{L^2}\|\na u\|_{L^6}\left(\|\na u\|_{L^6}+\|\na \te\|_{L^2}\right)
\|\dot\te\|_{L^6} +C  \|\na\dot u\|_{L^2} \|\n\dot\te\|_{L^2} \\
&\quad+C(\on)  \|\na u\|^{1/2}_{L^6}\|\na u\|^{1/2}_{L^2} \|\dot\te\|_{L^6}\left(\|\na u\|_{L^2}
+\|\n\dot\te\|_{L^2}+\|\na\dot u\|_{L^2}\right)  \\
&\le\frac{\ka}{2}\|\na\dot\te\|_{L^2}^2+C(\on)\|\na u\|_{L^2}^2\left(\|\na u\|_{L^6}^4+\|\na\te \|_{L^2}^4\right)+C(\on,M) \|\na u\|_{L^6}\|\na u\|_{L^2}^2\\
&\quad+C(\on,M) \left(1+\|\na u\|_{L^6}+\|\na\te\|_{L^2}^2\right) \left(\|\na^2\te\|_{L^2}^2+\|\na\dot u\|_{L^2}^2+\|\rho^{1/2} \dot \te\|_{L^2}^2\right),
\ea\ee
where we have used \eqref{bea}, \eqref{g1}, \eqref{g2}, \eqref{z1}, \eqref{pq}, and the following Poincar\'e-type inequality (\cite[Lemma 3.2]{feireisl1}):
\be \la{kk}
\|f\|_{L^p}\le C(\on)(\|\n^{1/2}f\|_{L^2}+\|\na f \|_{L^2}),~~~p\in[2,6],
\ee
for any $f\in\{h\in H^1 \left|\n^{1/2}h\in L^2\}\right.$.


Multiplying (\ref{3.99}) by $\si^2$ and integrating the resulting inequality over $(0,T),$
we obtain after integrating by parts that
\bnn\ba
& \sup_{0\le t\le T}\si^2\int \n|\dot\te|^2dx + \int_0^T\si^2 \|\na\dot\te\|_{L^2}^2dt  \\
&\le C(\on) \sup_{0\le t\le T} \left(\si^2(\|\na u\|_{L^6}^4+\|\na\te\|_{L^2}^4)\right)\int_0^T\|\na u\|_{L^2}^2dt \\
&\quad + C(\on,M) \sup_{0\le t\le T} \left(\si\left(1+\|\na u\|_{L^6}+\|\na\te\|_{L^2}^2\right)\right)\\
&\quad \quad \quad \quad \quad \quad \quad \cdot\int_0^T\si\left(\|\na^2\te\|_{L^2}^2+\|\na\dot u\|_{L^2}^2+\|\rho^{1/2} \dot \te\|_{L^2}^2\right)dt\\
&\quad +C(\on,M) \sup_{0\le t\le T} \left(\si\|\na u\|_{L^6} \right) \int_0^T\|\na u\|_{L^2}^2dt+C\int_0^T\si\|\rho^{1/2} \dot \te\|_{L^2}^2dt\\
&\le  C(\on,M), \ea\enn
where we have used   (\ref{z1}), (\ref{ae26}),  (\ref{k1}), and the following fact:
\be\ba\la{ong}\sup_{0\le t\le T}(\si\|\na u\|_{L^6}^2)\le C(\on,M)\ea\ee
due to \eqref{3.30}, \eqref{ae26}, and \eqref{z1}.

Next, it follows from (\ref{z1}),  \eqref{ae26}, (\ref{lop4}), (\ref{m20}), (\ref{m22}), (\ref{ae3.7}), (\ref{xx}), \eqref{nnn1}, and  (\ref{k1}) that
\be \la{vu02}\ba
\sup_{0\le t\le T}\left(\si^2\|\te  \|^2_{H^2}\right)+\int_0^T \left(\si\|\na u  \|_{L^4}^4+\si\|\na\te \|_{H^1}^2+\|\n -1\|_{L^2}^2\right)dt
 \le C(\on,M), \ea\ee
which along with (\ref{z1}), \eqref{ae26}, (\ref{k1}), \eqref{kk}, \eqref{ong}, and \eqref{ae3.7} gives
 \be\la{vu12}\ba    \int_0^T  \si \|u _t\|_{L^2}^2dt
 &\le C\int_0^T  \si(\| \dot u \|_{L^2}^2+\|u\cdot\na  u \|_{L^2}^2)dt\\
 &\le C(\hat\n)\int_0^T  \si(\|\n^{1/2} \dot u \|_{L^2}^2+\|\na\dot u\|_{L^2}^2+\|u \|_{L^\infty}^2\|\na  u \|_{L^2}^2)dt\\
 &\le C(\hat \n,M) ,\ea\ee
\be\la{vu11}\ba
\int_0^T  \si^2 \|  \te _t\|_{L^2}^2dt
&\le C\int_0^T  \si^2(\| \dot \te \|_{L^2}^2+\|u \cdot\na  \te \|_{L^2}^2)dt\\
&\le C(\hat\n)\int_0^T  \si^2(\|\n^{1/2} \dot \te\|_{L^2}^2+\|\na\dot\te\|_{L^2}^2+\|u\|_{L^6}^2\|\na\te\|_{L^3}^2)dt\\
&\le C(\hat \n,M) ,\ea\ee
 and
\be\la{vu01}\ba
\int_0^T  \si^2 \|  \na\te _t\|_{L^2}^2dt
&\le C\int_0^T  \si^2\|\na \dot \te \|_{L^2}^2  dt+ C\int_0^T  \si^2\|\na(u \cdot\na  \te )\|_{L^2}^2dt\\
&\le C(\on,M) +C\int_0^T\si^2\left(\|\na u \|_{L^3}^2+\|u \|_{L^\infty}^2\right)\|\na^2 \te \|_{L^2}^2dt  \\ &\le C(\on,M).\ea\ee
Hence, (\ref{vu15}) is derived from (\ref{ong})--\eqref{vu01} immediately.
The proof of Lemma \ref{le8} is finished.
\end{proof}

Finally, we end this section by establishing the exponential decay-in-time for the classical solutions.
\begin{lemma}\la{pr2}
Under the conditions of Proposition \ref{pr1}, there exist  positive constants $\ve_0$, $C^\ast$, $\al$, and $\te_\infty$ depending only on   $\mu,\,\lambda,\, \ka,\, R,\, \ga,\, \on,\, \bt,\,\O,$ and $M$ such that if $(\rho,u,\te)$ is a smooth solution to the problem (\ref{a1})--(\ref{h1}) on $\Omega\times (0,T] $ satisfying (\ref{z1}) with $K$ as in Lemma \ref{le2}, (\ref{h22}) holds for any $t\geq 1$, provided $C_0\le\ve_0.$
\end{lemma}
\begin{proof}
First, it follows from \eqref{a2.23}, \eqref{z1}, \eqref{key}, \eqref{a2.9}, \eqref{511}, \eqref{pq}, \eqref{6yue}, and \eqref{vu15} that for any $t\geq 1$,
\be\la{ao}\ba
&\frac{1}{2}W'(t)+ \mu \|\curl u\|_{L^2}^2 +(2\mu+\lambda)\|\div u\|_{L^2}^2+  \ka R \bp^{-1}  \|\na\te\|_{L^2}^2\\
&\le C|\bp_t|(\|R\te-\bp\|_{L^2}+\|R\te-\bp\|_{L^2}^2)+|\bp_t| \|\n-1\|_{L^2}^2 \\
&\quad+C\|R\te-\bp\|_{L^\infty}(\|R\te-\bp\|_{L^2}^2+\|\na u\|_{L^2}^2)\\
&\le C(C_0^{1/8}\|\na u\|_{L^2}+\|\na u\|_{L^2}^2)(\|\na\te\|_{L^2}+\|\na\te\|_{L^2}^2+ \|\n-1\|_{L^2}^2)\\
&\quad+C(\|\na\te\|_{L^2}^{1/2}\|\na^2\te\|_{L^2}^{1/2}+\|\na\te\|_{L^2})(\|\na\te\|_{L^2}^2+\|\na u\|_{L^2}^2)\\
&\le CC_0^{1/24}(\|\na u\|_{L^2}^2+\|\na\te\|_{L^2}^2+ \|\n-1\|_{L^2}^2),
\ea\ee
where
\be\la{t3}\ba
W(t)&\triangleq\int\left(\n |u|^2+2\bp (1+\n\log \n-\n)+\frac{1}{\ga-1} \n \bp^{-1}{(R\te-\bp)^2} \right)dx\\
&\le \hat C_3(\|\na u\|_{L^2}^2+\|\na\te\|_{L^2}^2+ \|\n-1\|_{L^2}^2)
\ea\ee
owing to \eqref{z1}, \eqref{key}, \eqref{a2.9}, and \eqref{pq}.
Combining \eqref{ao} with \eqref{h18} and \eqref{key} yields that
\be\la{t5}\ba &W'(t)+\hat C_1(\|\na u\|_{L^2}^2+\|\na \te\|_{L^2}^2)
\\&\le \hat C_2C_0^{1/24}(\|\na u\|_{L^2}^2+\|\na\te\|_{L^2}^2+ \|\n-1\|_{L^2}^2).\ea\ee

Next, rewriting $\eqref{a1}_2$   as
\bnn\la{t2}\ba &(\n u)_t+\div(\n u\otimes u)\\&=\mu\Delta u+(\mu+\lambda)\na(\div u)-\na(\n(R\te-\bp))-\bp\na(\n-1),\ea\enn
  multiplying this by $\mathcal{B}[\n-1]$ and using Lemma \ref{th00}, \eqref{z1}, and \eqref{pq},  one gets that for any $t\geq 1$,
\be\notag \ba
&\bp\int(\n-1)^2 dx \\
&= \left(\int\rho u\cdot\mathcal{B}[\n-1] dx\right)_t -\int\rho u\cdot\mathcal{B}[\n_t]dx
  -\int\rho u\cdot\nabla\mathcal{B}[\n-1]\cdot u dx \\
& \quad  +\mu\int\p_j u\cdot\p_j\mathcal{B}[\n-1] dx +(\mu+\lambda)\int(\rho-1)\div udx -\int\n(R\te-\bp)(\n-1)dx\\
& \le\left(\int\n u\cdot\mathcal{B}[\n-1]dx\right)_t+C\|\n u\|_{L^2}^2+C\|u\|_{L^{4}}^{2}\|\n-1\|_{L^2}\\
& \quad  +C\|\rho-1\|_{L^2}\|\na u\|_{L^2}+C\|\rho-1\|_{L^2}\|R\te-\bp\|_{L^2} \\
& \leq \left(\int\rho u\cdot\mathcal{B}[\n-1] dx\right)_t+\frac{\pi_1}{2}\|\n-1\|_{L^2}^2+C(\|\na u\|_{L^2}^2+\|\na \te\|_{L^2}^2),
\ea\ee
which as well as \eqref{key} leads to
\be\la{t6}\|\rho-1\|_{L^2}^2\le\frac{2}{\pi_1}\left(\int\rho u\cdot\mathcal{B}[\n-1] dx\right)_t+\hat C_4(\|\na u\|_{L^2}^2+\|\na \te\|_{L^2}^2).\ee
By virtue of \eqref{key}, \eqref{a2.9}, and Lemma \ref{th00}, it holds
\be \la{t4} \ba
\left|\int\rho u\cdot\mathcal{B}[\n-1] dx\right|
&\leq  C \left(\|\n u\|^2_{L^2}+\|\n-1\|_{L^2}^2\right)\\
&\le \hat C_5\left(\|\n^{1/2} u\|^2_{L^2}+2\bp (1+\n\log \n-\n)\right).
\ea\ee
Adding \eqref{t5} to \eqref{t6} multiplied by $\hat C_6$ with $\hat C_6=\min\{\frac{\pi_1}{4\hat C_5},\frac{\hat C_1}{4\hat C_4}\}$ yields
\be\ba\la{t1}
&W_1'(t)+\frac{3\hat C_1}{4}(\|\na u\|_{L^2}^2+\|\na \te\|_{L^2}^2)+\hat C_6\|\n-1\|_{L^2}^2\\
&\le \hat C_2C_0^{1/24}(\|\na u\|_{L^2}^2+\|\na\te\|_{L^2}^2+ \|\n-1\|_{L^2}^2),\ea\ee
where
$$W_1(t)\triangleq W(t)-\frac{2\hat C_6}{\pi_1}\int\rho u\cdot\mathcal{B}[\n-1] dx,$$
satisfies
\be\la{t8}\frac{1}{2}W(t)\le W_1(t)\le 2W(t)\ee
due to \eqref{t4}.
Thus we infer from \eqref{t1} that
\be\la{t9} W_1'(t)+\frac{\hat C_1}{2}(\|\na u\|_{L^2}^2+\|\na \te\|_{L^2}^2)+\frac{\hat C_6}{2}\|\n-1\|_{L^2}^2\le 0,\ee
provided
\be\la{t7}C_0\le\ve_0\triangleq \min\left\{\ve_1,\cdots,\ve_4,\left(\frac{\hat C_6}{2\hat C_2}\right)^{24},\left(\frac{\hat C_1}{4\hat C_2}\right)^{24}\right\}.\ee
Then by  \eqref{t3}, one derives that for $\alpha=\frac{1}{3}\min\{\frac{\hat C_1}{2\hat C_3},\frac{\hat C_6}{2\hat C_3}\}$,
\be\notag W_1'(t)+3\alpha W_1(t)\le 0,\ee
which along with \eqref{t8},  \eqref{a2.9}, \eqref{key}, \eqref{t3}, and \eqref{z1} shows that for any $t\geq 1$,
\be\la{t10}\|\n^{1/2} u\|_{L^2}^2+\|\n-1\|_{L^2}^2+\|\n^{1/2}(R\te-\bp)\|_{L^2}^2\le C W_1(t)\le Ce^{-3\al t}.\ee
Moreover, we deduce from \eqref{t9} and \eqref{t10} that for any $1\le t\le T<\infty$,
\be\la{t11}\int_1^Te^{\al t}(\|\na u\|_{L^2}^2+\|\na \te\|_{L^2}^2)dt\le C.\ee

Next, multiplying \eqref{1hh17} by $e^{\al t}$ and using \eqref{z1} imply that for $B_1$ defined in \eqref{an2},
\be\ba \la{t12}
&(e^{\al t}B_1(t))'+ \frac{1}{2}e^{\al t}\int\rho |\dot u|^2dx\le Ce^{\al t}\left(\|\na u\|_{L^2}^2+\|\na\te\|_{L^2}^2+\|P-\overline P\|_{L^2}^2\right).
\ea\ee
Note that by \eqref{z1}, \eqref{key}, and \eqref{t10},
\be\la{t16}\|P-\overline P\|_{L^2}\le \|\n(R\te-\bp)\|_{L^2}+\bp\|\n-1\|_{L^2}\le Ce^{-\al t},\ee which together with  \eqref{t11}, \eqref{t12}, and \eqref{h18} gives for any $1\le t\le T<\infty$,
\be\la{t13}\ba
\sup_{1\leq t\leq T}\left(e^{\al t}\|\na u\|_{L^2}^2\right)+\int_1^T e^{\al t}\|\n^{1/2}\dot u\|^2_{L^2}dt\le C.
\ea \ee

Furthermore, choosing $m=0$ in \eqref{e8},  it follows from \eqref{m20}, \eqref{ae9}, and \eqref{z1} that for any $t\geq1$,
\be\notag\ba
&\varphi'(t)+ 2\left(\int_{\p \O}( u \cdot \na n \cdot u )G dS\right)_t +\int\left(\frac{C_1}{2} |\nabla\dot{u}|^2 +\n|\dot \te|^2\right)dx\\
&\le C(\|\n^{1/2}\dot u\|_{L^2}^2+\|\na u\|_{L^2}^2 + \|\na \te\|_{L^2}^2+\|\n-1\|_{L^2}^2),
\ea\ee
where  $\varphi(t)$ is defined in \eqref{wq3}. Multiplying this by $e^{\al t}$ along with \eqref{e6}, \eqref{wq3}, \eqref{wq2}, \eqref{b2}, \eqref{t10}, \eqref{t11}, \eqref{t13}, \eqref{h19}, and \eqref{z1} yields that for any $1\le t\le T<\infty$,
\be\la{t15}\ba
\sup_{1\leq t\leq T}\left(e^{\al t}(\|\n^{1/2}\dot u\|_{L^2}^2+\|\na\te\|_{L^2}^2)\right)+\int_1^T e^{\al t}(\|\na\dot u\|_{L^2}^2+\|\n^{1/2}\dot \te\|^2_{L^2})dt\le C.
\ea \ee
Adopting the analogous method and applying \eqref{3.99}, \eqref{vu15}, \eqref{t10}, \eqref{t11}, \eqref{t13}, \eqref{t15}, \eqref{lop4}, \eqref{m20},  \eqref{z1}, and \eqref{ae9}, we obtain that
\be\la{t17}\ba
\sup_{1\leq t\leq T}\left(e^{\al t}(\|\n^{1/2}\dot\te\|_{L^2}^2+\|\na^2\te\|_{L^2}^2)\right)+\int_1^T e^{\al t}\|\na\dot \te\|_{L^2}^2dt\le C.
\ea \ee

Finally, it remains to determine the limit of $\te$ as $t$ tends to infinity. Combining \eqref{pt}, \eqref{t16}, and \eqref{t13} shows that for any $t\geq 1$,
\be\ba\notag
|\bp_t|\le C(\|\na u\|_{L^2}^2+\|P-\bp\|_{L^2}^2)\le C e^{-\al t},
\ea\ee
which implies there exists a constant $P_\infty$ such that $\lim_{t\rightarrow \infty}\overline P=P_\infty$ and
\be\la{t18}|\overline P-P_\infty|\le Ce^{-\al t}.\ee
Denoting $\te_\infty\triangleq P_\infty/R$,   we have
\be\la{t19}\|\te-\te_\infty\|_{L^2}^2\le C\|R\te-\bp\|_{L^2}^2+C|\bp-R\te_\infty|^2\le Ce^{-\al t},\ee
where we have used \eqref{pq}, \eqref{t15}, and \eqref{t18}.
Therefore, the combination of  \eqref{t10}, \eqref{t13}--\eqref{t17}, \eqref{t19}, \eqref{h17}, and \eqref{p} concludes \eqref{t7} and finishs the proof of Lemma \ref{pr2}.
\end{proof}

\section{\la{se4} A priori estimates (II): higher-order estimates}

In this section, we will derive the higher-order estimates of smooth solution $(\rho, u, \te)$ to problem (\ref{a1})--(\ref{h1})  on $ \Omega\times (0,T]$ with initial data $(\n_0 ,u_0,\te_0)$ satisfying (\ref{co3}) and (\ref{3.1}).

We shall assume that (\ref{z1}) and (\ref{z01}) both hold as well. To proceed,
we define $\tilde g $ as
\be \la{co12}\tilde g\triangleq\n_0^{-1/2}\left(
-\mu \Delta u_0-(\mu+\lambda)\na\div u_0+R\na (\n_0\te_0)\right).\ee
Then it follows from (\ref{co3}) and (\ref{3.1}) that
\be\la{wq01}\tilde g\in L^2.\ee
From now on, the generic constant $C $ will depend only  on \bnn
T, \,\, \| \tilde g\|_{L^2},    \,\|\n_0\|_{W^{2,q}}  ,   \,  \,\|\na u_0\|_{H^1},  \ \,
\| \na\te_0\|_{L^2} , \enn
besides  $\mu,\,\lambda,\, \ka,\, R,\, \ga,\, \on,\,\bt,\,\O,$ and $M.$

We begin with the following estimates on the spatial gradient of
the smooth solution $(\rho,u,\te).$

\begin{lemma}\la{le11}
	The following estimates hold:
	\be\label{lee2}\ba
	&\sup_{0\le t\le T} \left(\|\rho^{1/2}\dot u\|_{L^2}^2 + \sigma\|\rho^{1/2}\dot \te\|_{L^2}^2 +\|\te\|_{H^1}^2 + \sigma \|\na^2 \theta\|_{L^2}^2
 \right)  \\
	  &\quad+\ia\left(  \|\nabla\dot u\|_{L^2}^2  + \|\rho^{1/2}\dot \te\|_{L^2}^2+ \|\nabla^2 \theta\|_{L^2}^2 +\sigma \|\nabla\dot \te\|_{L^2}^2 \right) dt\le C,
	\ea\ee
	and
	\be\la{qq1}
	\sup_{0\le t\le T}\left(\|u\|_{H^2} +\|\n\|_{H^2}\right)
	+ \int_0^{T}\left( \|\nabla u\|_{L^{\infty}}^{3/2} + \si \| \na^3 \te\|_{L^2}^2+\|u\|_{H^3}^2 \right)dt\le C.
	\ee
\end{lemma}

\begin{proof}
The proof is divided into the following two steps.
	
{\it Step 1: The proof of (\ref{lee2}).}
First, for $\varphi(t)$ as in \eqref{wq3}, taking $m=0$ in \eqref{e8}, one gets
\be\la{ae8}\ba
& \varphi'(t) +   \int\left( \frac{C_1}{2}|\nabla\dot{u}|^2   +\n|\dot \te|^2\right)dx \\
&\le -2 \left(\int_{\p\O} G\left( u\cdot \na n \cdot u \right)dS\right)_t+C\left(\|\n^{1/2} \dot u\|_{L^2}^2+\|\na u\|_{L^2}^2+\|\na    \te\|_{L^2}^2\right)\\
&\quad +C\left(\|\n^{1/2}\dot u\|_{L^2}^3+\|\na\te\|_{L^2}^3+\|\na u\|_{L^2}^2+\|\n-1\|_{L^2}^2\right) \\
& \quad +C \left(\|\na u\|_{L^2}^2+\|\na\te\|_{L^2}^2\right)\left( \|\n^{1/2}\dot u\|_{L^2}^2+ \|\na \te \|_{L^2}^2+1 \right) \\
&\le -2 \left(\int_{\p\O} G\left( u\cdot \na n \cdot u \right)dS\right)_t+C   \left( \|\n^{1/2}\dot u\|_{L^2}^2 + \|\na\te\|_{L^2}^2\right) (\varphi+1)+C
\ea  \ee
due to   (\ref{z1}), \eqref{wq2}, (\ref{ae9}),  and \eqref{m20}.
Taking into account the compatibility condition \eqref{co2}, we can define
\bnn \ba
\sqrt{\n} \dot u(x,t=0) \triangleq
-\tilde g,
\ea\enn
which along with    (\ref{e6}), (\ref{2.48}), and \eqref{wq01} yields that
\be \la{wq02}
|\varphi(0)|\le C\| \tilde g\|_{L^2}^2+C\le C.
\ee
Then, integrating \eqref{ae8} over $(0,t)$, one obtains after using \eqref{z1}, \eqref{b2}, \eqref{h19}, \eqref{wq2}, and \eqref{wq02} that
\be\ba \la{q1}
& \varphi(t)+\int_0^t\int\left(\frac{C_1}{2} |\nabla\dot{u}|^2   +\n|\dot \te|^2\right)dxds\\
&\le  2\left| \int_{\p\O} G\left( u\cdot \na n \cdot u \right)dS\right|(t)+C\int_0^t \left(\|\n^{1/2}\dot u\|_{L^2}^2+ \|\na \te \|_{L^2}^2\right)\varphi ds+C\\
&\le  C (\|\na u\|_{L^2}^2\|\n^{1/2}\dot u\|_{L^2})(t)+C\int_0^t \left(\|\n^{1/2}\dot u\|_{L^2}^2+ \|\na \te \|_{L^2}^2\right)\varphi ds+C\\
&\le  \frac{1}{2} \varphi(t) + C \int_{0}^{t} \left(\|\n^{1/2}\dot u\|_{L^2}^2+ \|\na \te \|_{L^2}^2\right)\varphi ds+C.
\ea\ee
Applying Gr\"{o}nwall's inequality to \eqref{q1} and 
 using \eqref{z1} and (\ref{wq2}), it holds
\be\label{lee3}
\sup_{0\le t\le T} \left(\|\rho^{1/2}\dot u\|_{L^2}^2
+\|\na \te\|_{L^2}^2 \right) + \ia\int\left(|\nabla\dot
u|^2+\n|\dot\te|^2\right)dxdt\le C,
\ee
which together with  \eqref{key} and \eqref{pq} implies
\be\la{ff1} \ba
\|\te\|_{L^2} &\le C( \|R \te-\overline P \|_{L^2} +\overline P ) \le C(\|\na \te\|_{L^2}+1) \le C.
\ea \ee

Next,
multiplying (\ref{3.99}) by $\sigma$ and integrating over $(0,T)$  lead to
\be  \ba \la{a5}
&\sup\limits_{0\le t\le T} \si \int \n|\dot\te|^2dx+\int_0^T \si \|\na\dot\te\|_{L^2}^2dt\\
&\le
C\int_0^T\left(  \|\na^2\te\|_{L^2}^2+\|\na\dot u\|_{L^2}^2+ \|\n^{1/2}\dot\te\|_{L^2}^2 \right)dt + C\\
&\le  C,
\ea\ee
where we have used (\ref{lee3}),  (\ref{z1}), (\ref{m20}), \eqref{ae9},   (\ref{lop4}), and the following fact:
\be\la{w1}\|\na u\|_{L^6}\le C\ee
due to \eqref{3.30}, \eqref{z1}, and \eqref{lee3}.
Then, it  follows from (\ref{lop4}), \eqref{lee3}, \eqref{a5}, \eqref{w1}, (\ref{m20}), and (\ref{z1}) that
\be \notag \sup\limits_{0\le t\le T} \si\|\na^2\te\|_{L^2}^2 + \int_{0}^{T}\|\na^2\te\|_{L^2}^2 dt \le C,\ee
which along with \eqref{lee3}--\eqref{a5} gives \eqref{lee2}.


{\it Step 2: The proof of (\ref{qq1}).}
First, standard calculations show  that for $ 2\le p\le 6$,
\be\la{L11}\ba
\partial_t\norm[L^p]{\nabla\rho}
&\le C\norm[L^{\infty}]{\nabla u} \norm[L^p]{\nabla\rho}+C\|\na^2u\|_{L^p}\\
&\le C\left(1+\|\na u\|_{L^{\infty}}+\|\na^2\te \|_{L^2}\right)
\norm[L^p]{\nabla\rho}\\
&\quad +C\left(1+\|\na\dot u\|_{L^2}+\|\na^2\te \|_{L^2}\right), \ea\ee
where we have used
\be\ba\la{ua1}
\|\na^2 u\|_{L^p} & \le C(\|\rho \dot u\|_{L^p} + \|\na P\|_{L^p} + \|\na u\|_{L^2}  ) \\
& \le   C\left(1+\|\na\dot u\|_{L^2}+\|\na \te\|_{L^p}+\|\te\|_{L^\infty}\|\nabla\n\|_{L^p}\right)\\
&\le  C\left(1+\|\na\dot u\|_{L^2}+\|\na^2\te\|_{L^2}+(\|\na^2\te \|_{L^2} + 1)\|\nabla\n\|_{L^p}\right)
\ea\ee
due to \eqref{g1}, \eqref{tb90}, \eqref{rmk1},  \eqref{z1},  and  \eqref{lee2}.
It follows from Lemma \ref{le9}, \eqref{z1}, and (\ref{ua1})  that
\be\la{u13}\ba
\|\na u\|_{L^\infty }
&\le C\left(\|{\rm div}u\|_{L^\infty}+\|\curl u\|_{L^\infty}\right)\log(e+\|\na^2 u\|_{L^6}) +C\|\na u\|_{L^2}+C \\
&\le C\left( \|{\rm div}u\|_{L^\infty } + \|\curl u\|_{L^\infty }
\right)\log(e+ \|\na\dot u\|_{L^2 } + \|\na^2\te \|_{L^2})\\
&\quad +C\left(\|{\rm div}u\|_{L^\infty }+ \|\curl u\|_{L^\infty } \right)
\log\left(e  + \|\na \n\|_{L^6}\right)+C.
\ea\ee
Denote
\bnn\la{gt}\begin{cases}
	f(t)\triangleq  e+\|\na	\n\|_{L^6},\\
	h(t)\triangleq 1+  \|{\rm div}u\|_{L^\infty }^2+ \|\curl u\|_{L^\infty }^2
	+ \|\na\dot u\|_{L^2 }^2 +\|\na^2\te \|_{L^2}^2.
\end{cases}\enn
One obtains after submitting \eqref{u13} into (\ref{L11}) with  $p=6$ that
\bnn f'(t)\le   C h(t) f(t)\ln f(t) ,\enn
which implies \be\la{w2}  (\ln(\ln f(t)))'\le  C h(t).\ee
Note that by virtue of
(\ref{hj1}),  (\ref{key}), (\ref{lee2}),  (\ref{z1}), \eqref{bz5}, and (\ref{bz6}),  one gets
\be \la{p2}\ba
& \int_0^T\left(\|\div u\|^2_{L^\infty}+\|\curl u\|^2_{L^\infty} \right)dt \\
& \le  C\int_0^T\left(\|G\|^2_{L^\infty}+ \|\curl u\|^2_{L^\infty}+\|P-\bp\|^2_{L^\infty}\right)dt \\
&\le  C\ia\left(\| G\|^2_{  W^{1,6} } + \| \curl u\|^2_{W^{1,6}} + \|\te\|_{L^\infty}^2\right)dt + C \\
& \le   C\ia\left(\|\na G\|^2_{L^6}+\|\na\curl  u\|^2_{L^6}+\|\na^2\te \|_{L^2}^2\right)dt+C \\
&\le C\ia(\|\na \dot u\|^2_{L^2}+\|\na^2\te \|_{L^2}^2)dt+C \\
&\le  C,
\ea\ee
which as well as \eqref{w2} and \eqref{lee2}  yields that
\be \la{u113} \sup\limits_{0\le t\le T}\|\nabla \rho\|_{L^6}\le C.\ee
Combining this with (\ref{u13}),  \eqref{p2}, and (\ref{lee2}) leads to
\be \la{v6}\ia\|\nabla u\|_{L^\infty}^{3/2}dt \le C.\ee
Moreover, it follows from (\ref{rmk1}), \eqref{z1}, \eqref{u113}, and (\ref{lee2})  that
\be\ba\la{aa95}
\sup\limits_{0\le t\le T} \| u\|_{H^2}
\le &C \sup\limits_{0\le t\le T}\left(\|\n\dot u\|_{L^2}+\|\nabla P\|_{L^2}+\| \na u\|_{L^2}\right)\le C.
\ea\ee

Next, applying operator $\p_{ij}~(1\le i,j\le 3)$ to $(\ref{a1})_1$ gives
\be\la{4.52}  (\p_{ij} \n)_t+\div (\p_{ij} \n u)+\div (\n\p_{ij} u)+\div(\p_i\n\pa_j u+\p_j\n\p_i u)=0. \ee
Multiplying (\ref{4.52}) by $2\p_{ij} \n$ and  integrating the resulting equality over $\O,$ it holds
\be\la{ua2}\ba
\frac{d}{dt}\|\na^2\n\|^2_{L^2}
& \le C(1+\|\na u\|_{L^{\infty}})\|\na^2\n\|_{L^2}^2+C\|\na u\|^2_{H^2}\\
& \le C(1+\|\na u\|_{L^{\infty}} +\|\na^2\te \|_{L^2}^2)(1+\|\na^2\n\|_{L^2}^2) +C\|\na\dot u \|^2_{L^2}, \ea\ee
where one has used \eqref{z1}, \eqref{u113}, and the following estimate:
\be\ba\la{va2}
\| u\|_{H^3}&\le C\left(\|\na(\n\dot u)\|_{L^2}+\|\n\dot u\|_{L^2}+\| \na P\|_{H^1}+\|\na u\|_{L^2}\right)
\\&\le C\|\na\n\|_{L^3}\|\dot u\|_{L^6}+C\|\na \dot u\|_{L^2}+C\|\na \te\|_{H^1}
\\&\quad+C\||\na\n||\na\te|\|_{L^2}+C\|\te\|_{L^\infty}\|\na \n\|_{H^1}+C
\\ &\le  C\|\na\dot u \|_{L^2}+ C  (1+\|\na^2\te \|_{L^2})(1+\|\na^2\n\|_{L^2}) +C
\ea\ee
due to (\ref{rmk1}), (\ref{tb90}), (\ref{lee2}), (\ref{u113}),  and (\ref{z1}).
Then applying Gr\"{o}nwall's inequality to \eqref{ua2} and using  (\ref{lee2}), (\ref{v6}) yield
\be\la{ja3} \sup_{0\le t\le T} \|\na^2\n \|_{L^2}  \le C,\ee
which together with \eqref{va2} and \eqref{lee2} gives
\be\la{ja4} \int_0^T\|u\|_{H^3}^2 dt\le C . \ee

Finally,
applying the standard $H^1$-estimate to   elliptic problem (\ref{3.29}), one derives from \eqref{z1},  \eqref{lee2}, \eqref{u113}, \eqref{kk}, and \eqref{aa95} that
\be\la{ex4}\ba
\|\na^2\te\|_{H^1}
& \le C\left(\|\n \dot \te\|_{H^1}+\|\n\te\div u\|_{H^1}+\||\na u|^2\|_{H^1}\right)\\
& \le C\left(1+ \|\na \dot \te\|_{L^2} +  \|\rho^{1/2} \dot \theta\|_{L^2}  + \|\na(\n\te\div u)\|_{L^2}+ \||\na u||\na^2u|\|_{L^2} \right) \\
& \le C\left(1+ \|\na \dot \te\|_{L^2} +  \|\rho^{1/2} \dot \theta\|_{L^2}  +\|\na^2 \theta\|_{L^2}+\|\na^3 u\|_{L^2} \right),
\ea\ee
which along with \eqref{z1}, \eqref{ja3}, \eqref{ja4}, \eqref{u113}--\eqref{aa95}, and \eqref{lee2} yields (\ref{qq1}).


 The proof of Lemma \ref{le11} is finished.
\end{proof}

\begin{lemma}\la{le9-1}
	The following estimates hold:
	\be\la{va5}\ba&
	\sup\limits_{0\le t\le T}
	\|\n_t\|_{H^1}
	+\int_0^T\left(\|  u_t\|_{H^1}^2+\si \| \te_t\|_{H^1}^2+\| \n u_t\|_{H^1}^2+\si \|\n \te_t\|_{H^1}^2
	\right)dt\le C,
	\ea  \ee
	and
	\be\la{vva5}\ba
	\int_0^T \sigma \left( \|(\n u_t)_t\|_{H^{-1}}^2+\|(\n \te_t)_t\|_{H^{-1}}^2
	\right)dt\le C.\ea\ee
\end{lemma}

\begin{proof}

First, it   follows from  (\ref{lee2}) and (\ref{qq1})   that
\be\label{va1}\ba
&\sup_{0\le t\le T}\int\left( \n|u_t|^2 + \si
\n\te_t^2\right)dx +\int_0^T \left(\|\na u_t\|_{L^2}^2+ \si \|\na\te_t\|_{L^2}^2\right)dt\\
&\le \sup\limits_{0\le t\le T}\int \left(\n|\dot u|^2+ \si\n|\dot\te|^2 \right)dx
+\int_0^T\left(\|\na\dot u\|_{L^2}^2+\si \|\na\dot\te\|_{L^2}^2 \right)dt\\
&\quad + \sup\limits_{0\le t\le T}\int \n\left(|u\cdot\na u|^2+\si |u\cdot\na\te|^2\right)dx\\
&\quad +\int_0^T\left(\|\na(u\cdot\na u)\|_{L^2}^2+\si \|\na(u\cdot\na \te)\|_{L^2}^2\right)dt\\
&\le C,\ea\ee
which together with (\ref{lee2}) and (\ref{qq1}) gives
\be\la{vva1} \ba
& \int_0^T\left(\|\na(\n u_t)\|_{L^2}^2+ \si \|\na(\n\te_t)\|_{L^2}^2\right)dt\\
& \le  C\int_0^T\left(\|  \na u_t \|_{L^2}^2+\|  \na \n\|_{L^3}^2\| u_t \|_{L^6}^2
+ \si \|  \na \te_t \|_{L^2}^2+ \si \|  \na \n\|_{L^3}^2\| \te_t \|_{L^6}^2 \right)dt\\
&\le C,\ea\ee
where we have used
\be\la{w3}\|\te_t\|_{L^6}\le C\|\n^{1/2}\te_t\|_{L^2}+C\|\na\te_t\|_{L^2}
\ee
due to \eqref{kk}.

Next, one deduces from $(\ref{a1})_1$, (\ref{qq1}), and \eqref{hs} that
\bnn\ba\la{sp1}
\|\n_t\|_{H^1}\le& \|\div (\rho u)\|_{H^1}
\le  C \|u\|_{H^2}\|
\n\|_{H^2}
\le C,
\ea \enn
which as well as (\ref{va1})--\eqref{w3} shows    (\ref{va5}).

Finally, differentiating $(\ref{a1})_2$ with respect to $t$ yields that
\be   \la{va7}\ba (\n u_t)_t
=-(\n u\cdot\na u)_t + \left(\mu\Delta u+(\mu+\lambda)\na\div u \right)_t -\na P_t.\ea\ee
It follows from (\ref{va5}), \eqref{qq1}, \eqref{lee2}, \eqref{lop4}, and \eqref{w3} that
\be   \la{va9}\ba  \|(\n u\cdot\na u)_t  \|_{L^{2}}
&=\| \n_t u\cdot\na u+ \n u_t\cdot\na u + \n u\cdot\na u_t  \|_{L^{2}}\\
&\le C\|\n_t\|_{L^6} \|\na u\|_{L^3}+  C\|u_t\|_{L^6} \|\na u\|_{L^3}+  C\|u \|_{L^\infty} \|\na u_t\|_{L^2}\\
&\le C+   C\|\na u_t\|_{L^2},\ea\ee
and
\be   \la{va10}\ba  \|\na P_t  \|_{L^2}
&=R\|\n_t\na\te +\n \na\te_t +\na\n_t\te +\na\n\te_t\|_{L^2}\\
&\le C\left(\|\n_t\|_{L^6}\|\na\te\| _{L^3}+\|\na\te_t\|_{L^2}
+\|\te\|_{L^\infty}\|\na \n_t\|_{L^2}+\|\na\n\|_{L^6}\|\te_t\| _{L^3}\right)\\
&\le C+C\|\na\te_t\|_{L^2}+C\|\n^{1/2}\te_t\|_{L^2}.\ea\ee
Combining (\ref{va7})--(\ref{va10}) with   (\ref{va5}) shows
\be\la{vva04}\ba\int_0^T \si \|(\n u_t)_t\|_{H^{-1}}^2dt\le C.\ea\ee
Similarly, we have \bnn\int_0^T \si  \|(\n \te_t)_t\|_{H^{-1}}^2dt\le C,  \enn
which combined with (\ref{vva04}) implies (\ref{vva5}). The proof of Lemma \ref{le9-1} is completed.
\end{proof}

\begin{lemma}\la{pe1}
	The following estimate holds:
	\be\la{nq1}
	\sup\limits_{0\le t\le T} \si\left(\|\nabla u_t\|^2_{L^2}+\|\n_{tt} \|^2_{L^2} + \|u\|_{H^3}^2\right)
	+ \int_0^T\si \left(\|\rho^{1/2} u_{tt}\|_{L^2}^2+ \|\nabla u_t\|_{H^1}^2\right)dt\le C.
	\ee
\end{lemma}
\begin{proof}
Differentiating  $(\ref{a11})$  with respect to $t$ leads to \be\la{nt0}\ba \begin{cases}
(2\mu+\lambda)\na\div u_t-\mu \na\times \curl u_t\\
= \n u_{tt} +\n_tu_t+\n_tu\cdot\na u+\n u_t\cdot\na u +\n u\cdot\na u_t+\na P_t\triangleq \tilde f
,&\, \text{in}\,\O\times[0,T],\\
u_{t}\cdot n=0,\  \curl u_{t}\times n=0,  &\,\text{on}\,\p\O\times[0,T]. \end{cases}\ea\ee
Multiplying (\ref{nt0})$_1$ by $u_{tt}$   and integrating  the resulting equality  by parts, one gets  \be\la{sp9} \ba&
\frac{1}{2}\frac{d}{dt}\int \left(\mu|\curl u_t|^2 + (2\mu +\lambda)({\rm div}u_t)^2\right)dx
+\int \rho| u_{tt}|^2dx\\
&=\frac{d}{dt}\left(-\frac{1}{2}\int_{ }\rho_t |u_t|^2 dx- \int_{}\rho_t u\cdot\nabla u\cdot u_tdx
+ \int_{ }P_t {\rm div}u_tdx\right)\\
& \quad + \frac{1}{2}\int_{ }\rho_{tt} |u_t|^2 dx+\int_{ }(\rho_{t} u\cdot\nabla u )_t\cdot u_tdx
-\int_{ }\rho u_t\cdot\nabla u\cdot u_{tt}dx\\
&\quad - \int_{ }\rho u\cdot\nabla u_t\cdot u_{tt}dx - \int_{ }\left(P_{tt}-
\ka(\ga-1)\Delta\te_t\right){\rm div}u_tdx\\
&\quad +\ka(\ga-1)\int_{ } \na\te_t\cdot\na {\rm div}u_tdx
\triangleq\frac{d}{dt}\tilde{I}_0+ \sum\limits_{i=1}^6 \tilde{I}_i. \ea \ee

 Each term $\tilde{I}_i(i=0,\cdots,6)$ can be estimated as follows:

First, it follows from simple calculations, $(\ref{a1})_1,$   (\ref{va5}), \eqref{qq1}, \eqref{lee2}, and (\ref{va1}) that
\be \ba \la{sp10}
|\tilde{I}_0|&=\left|-\frac{1}{2}\int\rho_t |u_t|^2 dx- \int \rho_t u\cdot\nabla
u\cdot u_tdx+ \int P_t {\rm div}u_tdx\right|\\
&\le C\int  \n |u||u_t||\nabla u_t| dx+C\norm[L^3]{\rho_t}\norm[L^2]
{\nabla u}\norm[L^6]{u_t}+C\|(\n\te)_t\|_{L^2}\|\nabla u_t\|_{L^2}\\
&\le C \|\n^{1/2}u_t\|_{L^2}  \|\nabla u_t\|_{L^2} +C(1+\|\n^{1/2} \te_t\|_{L^2}+\|\n_t\|_{L^3}\|\te\|_{L^6})\|\nabla u_t\|_{L^2}\\
&\le C (1+\|\n^{1/2}\te_t\|_{L^2}) \|\nabla u_t\|_{L^2} ,\ea\ee
\be \la{sp11}\ba
2|\tilde{I}_1|&=\left|\int \rho_{tt} |u_t|^2 dx\right|  \le C \|\n_{tt}\|_{L^2}\|u_t\|_{L^4}^{2}
 \le  C\|\n_{tt}\|_{L^2}^2+C \|\na u_t\|_{L^2}^4,
\ea \ee
\be \la{sp12}\ba
|\tilde{I}_2|&=\left|\int \left(\rho_t u\cdot\nabla u \right)_t\cdot u_{t}dx\right|\\
& = \left|  \int\left(\rho_{tt} u\cdot\nabla u\cdot u_t +\rho_t
u_t\cdot\nabla u\cdot u_t+\rho_t u\cdot\nabla u_t\cdot u_t\right)dx\right|\\
&\le   C\norm[L^2]{\rho_{tt}}\norm[L^6]{\na u}\norm[L^6]{u}\norm[L^6]{u_t}
+C\norm[L^2]{\rho_t}\norm[L^6]{u_t}^2\norm[L^6]{\nabla u} \\
&\quad+C\norm[L^3]{\rho_t}\norm[L^{\infty}]{u}\norm[L^2]{\nabla u_t}\norm[L^6]{u_t}\\
& \le C\norm[L^2]{\rho_{tt}}^2 + C\norm[L^2]{\nabla u_t}^2, \ea \ee
and
\be\ba\la{sp13}
|\tilde{I}_3|+|\tilde{I}_4|&= \left| \int \rho u_t\cdot\nabla u\cdot u_{tt} dx\right|
+\left| \int \rho u\cdot\nabla u_t\cdot u_{tt}dx\right|\\
& \le   C\|\n^{1/2}u_{tt}\|_{L^2}\left(\|u_t\|_{L^6}\|\na u\|_{L^3}
+\|u\|_{L^\infty}\|\na u_t\|_{L^2}\right) \\
& \le  \frac{1}{4}\norm[L^2]{\rho^{{1/2}}u_{tt}}^2 + C \norm[L^2]{\nabla u_t}^2.\ea\ee
Then, by virtue of
(\ref{op3}), (\ref{va5}), \eqref{va10}, and Lemma \ref{le11}, it holds
\bnn\ba & \|P_{tt}-\ka(\ga-1)\Delta \te_t\|_{L^2}\\
&\le  C\|(u\cdot\na P)_t\|_{L^2}+C\|(P\div u)_t\|_{L^2}+C\||\na u||\na u_t|\|_{L^2}\\
&\le  C\|u_t\|_{L^6}\|\na P\|_{L^3}+C\|u\|_{L^\infty}\|\na P_t\|_{L^2}
+C\|P_t\|_{L^6}\|\na u\|_{L^3}\\
&\quad  +C\|P\|_{L^\infty}\|\na u_t\|_{L^2}+C\|\na u\|_{L^\infty}\|\na u_t\|_{L^2}\\
&\le C\left(1+\|\na u\|_{L^\infty}+\|\na^2\te \|_{L^2}\right)\|\na u_t\|_{L^2} +C\left(1+\|\na\te_t\|_{L^2}+\|\n^{1/2}\te_t\|_{L^2}\right),
\ea\enn
which yields
\be\ba\la{sp15}
|\tilde{I}_5|&=\left|\int\left(P_{tt}-\ka(\ga-1)\Delta \te_t\right){\rm div}u_tdx\right|\\
&\le\norm[L^2]{P_{tt}-\ka(\ga-1)\Delta \te_t}\norm[L^2]{\na u_t}\\
&\le C\left(1+\|\na u\|_{L^\infty}+\|\na^2\te \|_{L^2}\right)\|\na u_t\|_{L^2}^2\\
&\quad+C\left(1+\|\na\te_t\|^2_{L^2}+\|\n^{1/2}\te_t\|_{L^2}^2\right).
\ea\ee
Next, combining Lam\'{e}'s system \eqref{nt0} with Lemma \ref{zhle},   \eqref{qq1}, \eqref{va5}, and (\ref{va10}) gives
\be\la{nt4}\ba
\|\na^2u_t\|_{L^2}&\le C\|\tilde f\|_{L^2}+C\|\na u_t\|_{L^2}\\
&\le C\|\n u_{tt}\|_{L^2}
+C\|\n_t\|_{L^3}\|u_t\|_{L^6}+C\|\n_t\|_{L^3}\|\na u\|_{L^6}\|u\|_{L^\infty}\\
&\quad  +C\|u_t\|_{L^6}\|\na u\|_{L^3}+C\|\na u_t\|_{L^2}\|u\|_{L^\infty}+C\|\na P_t\|_{L^2}\\
&\le C\left(\|\n u_{tt}\|_{L^2}+\|\na u_t\|_{L^2}+\|\n^{1/2} \te_{t}\|_{L^2}+\|\na \te_t\|_{L^2}+1\right),\ea\ee
which immediately leads  to\be \la{asp16}\ba
|\tilde{I}_6| &=\ka(\ga-1) \left| \int_{ } \na\te_t\cdot\na{\rm div}u_tdx \right|   \\
& \le  C \|\na^2u_t\|_{L^2}\|\na\te_t\|_{L^2}\\
& \le \frac{1}{4} \|\n^{1/2} u_{tt}\|^2_{L^2}+ C\left(1+\|\na u_t\|^2_{L^2}+\|\n^{1/2} \te_t\|^2_{L^2}+\|\na\te_t\|^2_{L^2}\right).
\ea\ee

Putting   (\ref{sp11})--(\ref{sp15})  and (\ref{asp16}) into
(\ref{sp9}) yields
\be\la{4.052} \ba
& \frac{d}{dt}\int \left(\mu|\curl u_t|^2 + (2\mu +\lambda)({\rm div}u_t)^2-2\tilde{I}_0\right)dx
+\int \rho| u_{tt}|^2dx\\
&\le  C\left(1+\|\na u\|_{L^\infty}+\|\na u_t\|_{L^2}^2+\|\na^2\te \|_{L^2}^2 \right)\|\na u_t\|_{L^2}^2 \\
&\quad +C\left(1+\|\n_{tt}\|_{L^2}^2+\|\n^{1/2} \te_t\|^2_{L^2}+\|\na \te_t\|_{L^2}^2\right).\ea\ee
Furthermore, we infer from $(\ref{a1})_1$, \eqref{qq1}, and (\ref{va5}) that
\be \la{s4} \ba
\|\n_{tt}\|_{L^2} &= \|\div(\rho u)_t\|_{L^2} \\
& \le C\left(\|\n_t\|_{L^6}\|\nabla u\|_{L^3}+ \|\nabla u_t\|_{L^2}
+\|u_t\|_{L^6}\|\nabla \n\|_{L^3}+\|\nabla \n_t\|_{L^2}\right) \\ &\le C+C\|\na u_t\|_{L^2}.\ea\ee
Multiplying \eqref{4.052} by $\sigma$ and integrating the resulting inequality over $(0,T)$, one thus  deduces from \eqref{nn2}, \eqref{lee2}, (\ref{qq1}),  (\ref{va1}), (\ref{sp10}),     (\ref{s4}),
and Gr\"{o}nwall's inequality that
\be\la{nq11}
\sup\limits_{0\le t\le T} \si \|\nabla u_t\|^2_{L^2}
+ \int_0^T\si\int\rho |u_{tt}|^2dxdt
\le C.
\ee

Finally, by Lemma \ref{le11}, \eqref{s4}, (\ref{va2}), \eqref{nt4}, \eqref{va1}, and  (\ref{nq11}), we have
\be\notag
\sup\limits_{0\le t\le T}\si\left(\|\n_{tt}\|_{L^2}^2+\|u\|^2_{H^3}\right) + \int_0^T \si\|\nabla u_t\|_{H^1}^2 dt\le C,
\ee
which along with \eqref{nq11} gives (\ref{nq1}).
We  complete  the proof of Lemma \ref{pe1}.
\end{proof}

\begin{lemma}\la{pr3} For $q\in (3,6)$ as in Theorem \ref{th1}, it holds that
	\be\la{y2}\ba
	&\sup_{0\le t\le T} \|\n\|_{W^{2,q}} +\int_0^T
	 \|\na^2u\|_{W^{1,q}}^{p_0}  dt\le C,
	\ea \ee
	where  \be 1< \la{pppppp} p_0<\frac{4q }{5q-6} \in (1,4/3).\ee
\end{lemma}

\begin{proof}
First, it follows from  \eqref{rmk1}, \eqref{tb90}, and Lemma \ref{le11} that
\be\la{a4.74}\ba
\|\na^2u\|_{W^{1,q}}
\le & C\left(\|\n \dot u\|_{W^{1,q}}+\|\na P\|_{W^{1,q}}+\|\na u\|_{L^2}\right)\\
\le & C
(\|\na\dot u\|_{L^2}+\|\na(\n\dot u)\|_{L^q}+  \|\na^2\te\|_{L^2}+ \|\te\na^2\n\|_{L^q}\\
& +\| |\na\n||\na\te|\|_{L^q}+  \|\na^2\te\|_{L^q}+1)\\
\le & C\left(\|\na\dot u\|_{L^2}+\|\na(\n\dot u)\|_{L^q}  + \| \na^2 \theta\|_{L^q}+1 \right)\\
&+ C(1+ \|\na^2 \theta\|_{L^2})\| \na^2 \n\|_{L^q}.
\ea\ee
Next, multiplying (\ref{4.52}) by $q|\p_{ij} \n|^{q-2}\p_{ij} \n$ and  integrating the resulting equality over $\O,$ we obtain after using  (\ref{qq1}) and (\ref{a4.74}) that
\be\la{sp28}\ba
&\frac{d}{dt}\|\na^2\n\|_{L^q}^q\\
&\le C\|\na u\|_{L^\infty}\|\na^2\n\|_{L^q}^q
+C\|\na^{2} u\|_{W^{1,q}}\|\na^2\n\|_{L^q}^{q-1}(\|\na\n\|_{L^q}+1)\\
&\le C\| u\|_{H^3}\|\na^2\n\|_{L^q}^q+C\|\na^2 u\|_{W^{1,q}}\|\na^2\n\|_{L^q}^{q-1}\\
&\le C\left(\| u\|_{H^3}   + \|\na\dot u\|_{L^2} +\|\na(\n\dot u)\|_{L^q}+\|\na^2 \theta\|_{L^q}+1\right)\left(\|\na^2 \n\|_{L^q}^q+1\right).
\ea\ee

Note that Lemma \ref{le11}, (\ref{g1}), \eqref{tb90}, and \eqref{nq1} imply
\be\la{4.49}\ba     \|\na(\n\dot u)\|_{L^q}
&\le C\|\na \n\|_{L^6}\|\dot u \|_{L^{6q/(6-q)}}+C\|\na\dot u \|_{L^q}\\
&\le C\|\dot u \|_{W^{1,6q/(6+q)}}+C\|\na\dot u \|_{L^q}\\
&\le C\|\na u_t \|_{L^q}+C\|\na(u\cdot \na u ) \|_{L^q}+C\\
&\le C\|\na u_t \|_{L^2}^{(6-q)/2q}\|\na u_t \|_{L^6}^{3(q-2)/{2q}}\\
& \quad+C\|\na u \|_{L^q}\| \na u \|_{L^\infty}+C\| u \|_{L^\infty}\|\na^2 u \|_{L^q}+C\\
&\le C\si^{-1/2} \left(\si\|\na u_t \|^2_{H^1}\right)^{3(q-2)/{4q}}
+C\|u\|_{H^3}+C,\ea \ee
and
\begin{equation}\notag
    \begin{aligned}
    \| \na^2 \theta\|_{L^q} \le& C \|\na^2 \theta\|_{L^2}^{(6-q)/2q} \|\na^3 \theta\|_{L^2}^{3(q-2)/2q} +C\|\na^2 \theta\|_{L^2}\\
    \le & C \si^{-1/2} \left(\si\|\na^3 \theta \|^2_{L^2} \right)^{3(q-2)/{4q}} +C \|\na^2 \theta\|_{L^2},
\end{aligned}
\end{equation}
which combined with Lemma \ref{le11} and \eqref{nq1} shows that, for $p_0$ as in (\ref{pppppp}),
\be \la{4.53}\int_0^T \left(\|\na(\n \dot u)\|^{p_0}_{L^q} + \|\na^2 \theta\|_{L^q}^{p_0} \right) dt\le C. \ee

Finally, applying  Gr\"{o}nwall's
inequality to (\ref{sp28}), we obtain after using Lemma \ref{le11} and (\ref{4.53}) that
\bnn  \sup\limits_{0\le t\le T}\|\na^2 \n\|_{L^q}\le C,\enn
which together with   Lemma \ref{le11},    (\ref{4.53}),
and (\ref{a4.74}) gives (\ref{y2}).  We finish  the proof of Lemma \ref{pr3}.
\end{proof}

\begin{lemma}\la{sq90} For $q\in (3,6)$ as in Theorem \ref{th1}, the following estimate holds:
	\be \ba\la{eg17}
	&\sup_{ 0\le t\le T}\si \left(\|\te_t\|_{H^1}+\|\na^2\te\|_{H^1}+\| u_t\|_{H^2}
	+\| u\|_{W^{3,q}}\right)   +\int_0^T   \si^2\|\na u_{tt}\|_{L^2}^2 dt\le C.\ea  \ee
	
\end{lemma}

\begin{proof}
First,  differentiating $(\ref{nt0})$ with
respect to $t$ gives
\be\la{sp30}\ba \begin{cases}
\n u_{ttt}+\n u\cdot\na u_{tt}-(2\mu+\lambda)\nabla{\rm div}u_{tt} +\mu\na \times \curl u_{tt}\\
= 2{\rm div}(\n u)u_{tt} +{\rm div}(\n u)_{t}u_t-2(\n u)_t\cdot\na u_t\\\quad -(\n_{tt} u+2\n_t u_t) \cdot\na u
- \n u_{tt}\cdot\na u-\na P_{tt}, & \text{in}\,\O\times[0,T],\\
u_{tt} \cdot n=0,\quad \curl u_{tt} \times n=0, &  \text{on}\,\p\O\times[0,T].
\end{cases}\ea \ee
Multiplying (\ref{sp30})$_1$ by $u_{tt}$ and integrating the resulting equality over ${\Omega}$ by parts imply that
\be \la{sp31}\ba
&\frac{1}{2}\frac{d}{dt}\int \n |u_{tt}|^2dx
+\int \left((2\mu+\lambda)({\rm div}u_{tt})^2+\mu|\curl u_{tt}|^2\right)dx\\
&=-4\int  u^i_{tt}\n u\cdot\na u^i_{tt} dx
-\int (\n u)_t\cdot \left(\na (u_t\cdot u_{tt})+2\na u_t\cdot u_{tt}\right)dx\\
&\quad -\int (\n_{tt}u+2\n_tu_t)\cdot\na u\cdot u_{tt}dx
-\int   \n u_{tt}\cdot\na u\cdot  u_{tt} dx\\
& \quad+\int  P_{tt}{\rm div}u_{tt}dx\triangleq\sum_{i=1}^5\tilde{J}_i.\ea\ee
It follows from Lemmas \ref{le11}--\ref{pe1}, (\ref{va1}),   \eqref{w3}, and \eqref{s4} that, for $\eta\in(0,1],$
\be \la{sp32} \ba
|\tilde{J}_1| &\le C\|\n^{1/2}u_{tt}\|_{L^2}\|\na u_{tt}\|_{L^2}\| u \|_{L^\infty} \le \eta \|\na u_{tt}\|_{L^2}^2+C(\eta) \|\n^{1/2}u_{tt}\|^2_{L^2},\ea\ee
\be \la{sp33}\ba
|\tilde{J}_2| &\le C\left(\|\n u_t\|_{L^3}+\|\n_t u\|_{L^3}\right)
\left(\| \na u_{tt}\|_{L^2}\| u_t\|_{L^6}+\| u_{tt}\|_{L^6}\| \na u_t\|_{L^2}\right)\\
&\le C\left(\|\n^{1/2} u_t\|^{1/2}_{L^2}\|u_t\|^{1/2}_{L^6}+\|\n_t\|_{L^6}\| u\|_{L^6}\right)\| \na u_{tt}\|_{L^2}\| \na u_t\|_{L^2}\\
&\le \eta \|\na u_{tt}\|_{L^2}^2+C(\eta)\| \na u_t\|_{L^2}^{3}+C(\eta)\\
&\le \eta \|\na u_{tt}\|_{L^2}^2+C(\eta)\si^{-3/2}  ,\ea\ee
\be  \la{sp34}\ba
|\tilde{J}_3| &\le C\left(\|\n_{tt}\|_{L^2}\|u\|_{L^6}+
\|\n_{t}\|_{L^2}\|u_{t}\|_{L^6} \right)\|\na u\|_{L^6}\|u_{tt}\|_{L^6} \\
&\le \eta \|\na u_{tt}\|_{L^2}^2+C(\eta)\si^{-1}  ,\ea\ee
and\be  \la{sp36}\ba &
|\tilde{J}_4|+|\tilde{J}_5|\\
&\le  C\|\n u_{tt}\|_{L^2} \|\na u\|_{L^3}\|u_{tt}\|_{L^6}
+C \|(\n_t\te+\n\te_t)_t\|_{L^2}\|\na u_{tt}\|_{L^2}\\
&\le  \eta \|\na u_{tt}\|_{L^2}^2+C(\eta) \left(\|\n^{1/2}u_{tt}\|^2_{L^2}
+\|\n_{tt}\te\|_{L^2}^2+\|\n_{t}\te_t\|_{L^2}^2 +\|\n^{1/2}\te_{tt}\|_{L^2}^2\right) \\
&\le  \eta \|\na u_{tt}\|_{L^2}^2+C(\eta)\left( \|\n^{1/2}u_{tt}\|^2_{L^2}
+\|\na\te_t\|_{L^2}^2 +\|\n^{1/2}\te_{tt}\|_{L^2}^2+\sigma^{-2}\right). \ea\ee
Substituting (\ref{sp32})--(\ref{sp36}) into (\ref{sp31}), we obtain after using \eqref{nn2} and choosing $\eta$ suitably small  that
\be \la{ex12}\ba
& \frac{d}{dt}\int \n |u_{tt}|^2dx+C_4 \int|\na u_{tt}|^2dx \\
& \le  C\si^{-2} +C\|\n^{1/2}u_{tt}\|^2_{L^2}
+C\|\na\te_t\|_{L^2}^2+C_5\|\n^{1/2}\te_{tt}\|_{L^2}^2.\ea\ee

Next, differentiating \eqref{3.29} with respect to $t$ infers
\be\la{eg1}\ba \begin{cases}
-\frac{\ka(\ga-1)}{R}\Delta \te_t+\n\te_{tt}\\
=-\n_t\te_{t}- \n_t\left(u\cdot\na \te+(\ga-1)\te\div u\right)-\n\left( u\cdot\na
\te+(\ga-1)\te\div u\right)_t\\
\quad+\frac{\ga-1}{R}\left(\lambda (\div u)^2+2\mu |\mathfrak{D}(u)|^2\right)_t,  \\
\na \te_t\cdot n|_{ \p\O\times(0,T)}=0  .
\end{cases}\ea\ee
Multiplying  (\ref{eg1})$_1$ by $\te_{tt}$ and integrating the resulting
equality over $\Omega$ lead to
\be\la{ex5}\ba
& \left(\frac{\ka(\ga-1)}{2R}\|\na \te_t\|_{L^2}^2+H_0\right)_t+ \int\n\te_{tt}^2dx \\
&=\frac{1}{2}\int\n_{tt}\left( \te_t^2
+2\left(u\cdot\na \te+(\ga-1)\te\div u\right)\te_t\right)dx\\
&\quad + \int\n_t\left(u\cdot\na\te+(\ga-1)\te\div u \right)_t\te_{t}dx\\
& \quad-\int\n\left(u\cdot\na\te+(\ga-1)\te\div u\right)_t\te_{tt}dx\\
& \quad -\frac{\ga-1}{R}\int \left(\lambda (\div u)^2+2\mu |\mathfrak{D}(u)|^2\right)_{tt}\te_t dx
\triangleq\sum_{i=1}^4H_i,\ea\ee
where
\bnn\ba H_0\triangleq & \frac{1}{2}\int \n_t\te_{t}^2dx
+\int\n_t\left(u\cdot\na\te+(\ga-1)\te\div u\right) \te_tdx\\
&- \frac{\ga-1}{R}\int\left(\lambda (\div u)^2+2\mu |\mathfrak{D}(u)|^2 \right)_t\te_t dx. \ea\enn
It follows from  $(\ref{a1})_1,$   (\ref{va1}),    \eqref{w3}, \eqref{s4}, and Lemmas \ref{le11}--\ref{pe1} that
\be\la{ex6}\ba
|H_0|\le & C\int \n|u||\te_{t}||\na\te_{t}|dx+C\|\n_t\|_{L^3}\|\te_t\|_{L^6}\left( \|\na\te\|_{L^2} \|u\|_{L^\infty}+ \|\theta\|_{L^6} \|\na u\|_{L^3}\right)\\
&+ C\|\na u\|_{L^3}\|\na u_t\|_{L^2} \|\te_t\|_{L^6} \\
\le &C  (\|\rho^{1/2}\theta_t \|_{L^2}+\|\na\te_t\|_{L^2})\left(\|\n^{1/2}\te_t\|_{L^2}+\|\na u_t\|_{L^2}+1\right)\\
\le &\frac{\ka(\ga-1)}{4R} \|\na\te_t\|_{L^2}^2+C\si^{-1},\ea\ee
and
\be\la{ex7}\ba
|H_1|& \le C\|\n_{tt}\|_{L^2}\left(\|\te_t\|_{L^4}^{2}
+\|\te_t\|_{L^6}\left(\|u\cdot\na \te\|_{L^3}+\| \te\div u\|_{L^3} \right)\right)\\
&\le C\|\n_{tt}\|_{L^2}\left(\|\rho^{1/2} \te_t\|_{L^2}^{2} + \|\na \te_t\|_{L^2}^{2}
+\si^{-1}  \right) \\
& \le  C(1+\|\na u_{t}\|_{L^2} )\|\na \te_t\|^2_{L^2}+C\si^{-3/2}.\ea\ee
Combining  Lemma \ref{le11} with (\ref{w3}) gives
\be\la{eg12}\ba
 &\|\left(u\cdot\na\te+(\ga-1)\te\div u \right)_t\|_{L^2}\\
& \le  C\left(\|u_t\|_{L^6}\|\na\te\|_{L^3}+\|\na\te_t\|_{L^2}+\|\te_t\|_{L^6}\|\na u\|_{L^3}
+\|\te\|_{L^\infty}\|\na u_t\|_{L^2}\right)\\
& \le  C\|\na u_t\|_{L^2}(\|\na^2 \te\|_{L^2}+1)+ C\|\na \te_t\|_{L^2}+C\|\rho^{1/2} \theta_t \|_{L^2},\ea\ee
which together with \eqref{lee2}, \eqref{va5},  (\ref{va1}), and \eqref{w3} shows
\be\la{ex9}\ba
|H_2|+|H_3|&\le C\left(\si^{-1/2}(\|\na u_t\|_{L^2}+1)+\|\na\te_t\|_{L^2}\right)
\left(\|\n_t\|_{L^3} \|\te_t\|_{L^6}+\|\n \te_{tt}\|_{L^2}\right)\\
&\le \frac{1}{2}\int\n\te_{tt}^2dx+C\|\na\te_t\|_{L^2}^2+C\si^{-1 } \|\na u_t\|^2_{L^2}+C\si^{-1 } . \ea\ee
One deduces from \eqref{qq1}, (\ref{va1}), \eqref{w3}, and (\ref{nq1}) that
\be\la{ex10}\ba
|H_4|&\le C\int \left(|\na u_t|^2+|\na u||\na u_{tt}|\right)|\te_t|dx\\
&\le C\left(\|\na u_t\|_{L^2}^{3/2}\|\na u_t\|_{L^6}^{1/2}
+ \|\na u\|_{L^3} \|\na u_{tt}\|_{L^2}\right)\|\te_t\|_{L^6}\\
&\le \de\|\na u_{tt}\|^2_{L^2}+C\|\na^2 u_t\|^2_{L^2}
+C(\de)(\|\na\te_t\|_{L^2}^2+\si^{-1 })+C\si^{-2}\|\na u_t\|_{L^2}^2.\ea\ee
Substituting (\ref{ex7}), (\ref{ex9}), and (\ref{ex10}) into (\ref{ex5}) gives
\be\la{ex11}\ba
& \left(\frac{\ka(\ga-1)}{2R}\|\na \te_t\|_{L^2}^2+H_0\right)_t
+\frac{1}{2}\int\n\te_{tt}^2dx \\
&\le  \de\|\na u_{tt}\|^2_{L^2}+C(\de)((1+\|\na u_{t}\|_{L^2}) \|\na \te_t\|^2_{L^2}+\si^{-3/2})\\
&\quad +C(\|\na^2 u_t\|^2_{L^2} +\si^{-2}\|\na u_t\|_{L^2}^2).\ea\ee

Finally, for $C_5$ as in (\ref{ex12}), adding (\ref{ex11}) multiplied by
$2 (C_5+1) $ to  (\ref{ex12}) and choosing $\de$ suitably small yield that
\bnn\la{ex13}\ba
& \left[ 2 (C_5+1)\left(\frac{\ka(\ga-1)}{2R}\|\na \te_t\|_{L^2}^2+H_0\right)
+\int \n |u_{tt}|^2dx\right]_t\\
&\quad + \int\n\te_{tt}^2dx+\frac{C_4}{2}\int |\na u_{tt}|^2dx\\
&\le C (1+\|\na u_{t}\|_{L^2}^2) (\si^{-2} +\|\na \te_t\|^2_{L^2})
+C\|\n^{1/2}u_{tt}\|^2_{L^2} + C\|\na^2 u_t\|^2_{L^2}.\ea\enn
Multiplying this by $\si^2$ and integrating the resulting inequality over $(0,T),$
we  obtain after using (\ref{ex6}),  (\ref{nq1}),  (\ref{va5}), and Gr\"{o}nwall's inequality that
\be \la{eg10}
\sup_{ 0\le t\le T}\si^2\int \left(|\na\te_t|^2+\n |u_{tt}|^2\right)dx
+\int_{0}^T\si^2\int \left(\n\te_{tt}^2+|\nabla u_{tt}|^2\right)dxdt\le C,\ee
which together with Lemmas \ref{le11}, \ref{pe1}, \ref{pr3}, (\ref{nt4}), (\ref{ex4}),  \eqref{va1}, (\ref{a4.74}),
and (\ref{4.49}) gives
\be\la{sp20} \sup_{ 0\le t\le T}\si \left(\|\na u_t\|_{H^1}
+ \|\na^2\te\|_{H^1}+\|\na^2u\|_{W^{1,q}} \right)\le C.\ee

We thus derive (\ref{eg17}) from  (\ref{eg10}),   (\ref{sp20}), \eqref{va1}, \eqref{w3},
and (\ref{qq1}). The proof of Lemma \ref{sq90} is completed.
\end{proof}

\begin{lemma}\la{sq91} The following estimate holds:
	\be \la{egg17}\sup_{ 0\le t\le T}\si^2  \left(\|\na^2\te\|_{H^2}+\| \te_t\|_{H^2}+\|\n^{1/2}\te_{tt}\|_{L^2}  \right)
	+\int_0^T\si^4\|\na \te_{tt}\|_{L^2}^2 dt\le C.\ee
\end{lemma}

\begin{proof}
First, differentiating $(\ref{eg1})$ with respect to $t$ yields
\be\la{eg2}\ba \begin{cases}
\n\te_{ttt}-\frac{\ka(\ga-1)}{R}\Delta \te_{tt}\\
=-\n u
\cdot\na\te_{tt}+ 2\div(\n u)\te_{tt} - \n_{tt}\left(\te_t+ u\cdot\na \te+(\ga-1)\te\div u\right)\\
\quad - 2\n_t\left(u\cdot\na \te+(\ga-1)\te\div u\right)_t\\
\quad - \n\left(u_{tt}\cdot\na \te+2u_t\cdot\na\te_t+(\ga-1)(\te\div u)_{tt}\right)\\
\quad +\frac{\ga-1}{R}\left(\lambda (\div u)^2+2\mu |\mathfrak{D}(u)|^2 \right)_{tt},  \\
\na\te_{tt}\cdot n|_{\p\O\times(0,T)}=0   .
\end{cases}\ea\ee
Multiplying (\ref{eg2})$_1$ by $\te_{tt}$ and integrating the resulting equality over $\Omega$ yield that
\be\la{eg3}\ba
&\frac{1}{2}\frac{d}{dt}\int\n|\te_{tt}|^2dx +\frac{\ka(\ga-1)}{R}\int|\na \te_{tt}|^2dx\\
&=-4\int \te_{tt}\n u\cdot\na\te_{tt}dx  -\int \n_{tt}\left(\te_t
+ u\cdot\na \te+(\ga-1)\te\div u\right)\te_{tt}dx\\
&\quad - 2\int\n_t\left(u\cdot\na \te+(\ga-1)\te\div u\right)_t\te_{tt}dx\\
& \quad - \int\n\left(u_{tt}\cdot\na \te+2u_t\cdot\na\te_t
+(\ga-1)(\te\div u)_{tt}\right)\te_{tt}dx\\
& \quad +\frac{\ga-1}{R}\int  \left(\lambda (\div u)^2+2\mu |\mathfrak{D}(u)|^2\right)_{tt}\te_{tt}dx
\triangleq \sum_{i=1}^5K_i.\ea\ee

It follows from
Lemmas \ref{le11}--\ref{pe1}, \ref{sq90},  \eqref{kk}, (\ref{eg10}), and \eqref{va1}  that
 \be \la{eg4} \ba
\si^4|K_1|&\le C\si^4\|\n^{1/2}\te_{tt}\|_{L^2}\|\na \te_{tt}\|_{L^2}\|u\|_{L^\infty}\\
&\le \de \si^4\|\na \te_{tt}\|_{L^2}^2+C(\de) \si^4\|\n^{1/2}\te_{tt}\|^2_{L^2} ,\ea\ee
\be \la{eg16}\ba
\si^4|K_2|&\le C \si^4\|\n_{tt}\|_{L^2}\|\te_{tt}\|_{L^6} \left( \|\te_t\|_{H^1}
+\|\na\te\|_{L^3}+\|\na u\|_{L^6}\|\te\|_{L^6}\right) \\
&\le C\si^2 (\|\na \te_{tt}\|_{L^2}+\|\n^{1/2} \te_{tt}\|_{L^2})\\
&\le \de\si^4\|\na \te_{tt}\|_{L^2}^2+C(\de)(\si^4\|\n^{1/2} \te_{tt}\|_{L^2}^2+1),\ea\ee
\be \la{eg7}\ba
\si^4|K_4|&\le C\si^4\|\te_{tt}\|_{L^6}
\left( \|\na\te\|_{L^3}\|\n u_{tt}\|_{L^2}
+\|\na\te_t\|_{L^2}\|u_t\|_{L^3}\right)\\
&\quad+ C\si^4\|\te_{tt}\|_{L^6} \left( \|\na u\|_{L^3}\|\n \te_{tt}\|_{L^2}
+ \|\na u_t\|_{L^2}\| \te_t\|_{L^3}\right)\\
&\quad+C\si^4\|\te\|_{L^\infty}\|\n\te_{tt}\|_{L^2} \|\na u_{tt}\|_{L^2} \\
&\le \de\si^4\|\na \te_{tt}\|_{L^2}^2
+C(\de)\left(\si^4 \|\n^{1/2} \te_{tt}\|_{L^2}^2+\si^3\|\na u_{tt}\|_{L^2}^2  \right)+C(\de),\ea\ee
\be \la{eg8}\ba
\si^4|K_5|&\le C\si^4\|\te_{tt}\|_{L^6}
\left( \|\na u_t\|_{L^2}^{3/2}\|\na u_{t}\|_{L^6}^{1/2}+\|\na u\|_{L^3}\|\na u_{tt}\|_{L^2}\right)  \\
&\le \de\si^4\|\na \te_{tt}\|_{L^2}^2
+C(\de)\si^4\left(\|\n^{1/2} \te_{tt}\|_{L^2}^2+\|\na u_{tt}\|_{L^2}^2  \right) +C(\de),\ea\ee
and
\be \la{eg6}\ba
\si^4|K_3|&\le C \si^4\|\n_t\|_{L^3} \|\te_{tt}\|_{L^6}
\left( \si^{-1/2}\|\na u_t\|_{L^2} +\|\rho^{1/2} \theta_t\|_{L^2} +\|\na\te_{t}\|_{L^2} \right) \\
&\le \de\si^4\|\na \te_{tt}\|_{L^2}^2+ C \si^4 \|\n^{1/2} \te_{tt}\|_{L^2}^2 +C(\de),\ea\ee
where in the last inequality we have used (\ref{eg12}).

Then, multiplying (\ref{eg3})  by $\si^4,$ substituting (\ref{eg4})--(\ref{eg6}) into the resulting equality and choosing $\de$ suitably small, one obtains
\bnn \ba
& \frac{d}{dt}\int\si^4\n|\te_{tt}|^2dx +\frac{\ka(\ga-1)}{R}\int\si^4|\na \te_{tt}|^2dx\\
& \le  C\si^2\left(\|\n^{1/2} \te_{tt}\|_{L^2}^2
+\|\na u_{tt}\|_{L^2}^2  \right)+C,\ea\enn
which together with (\ref{eg10})   gives
\be\la{eg13} \sup_{ 0\le t\le T}\si^4\int  \n |\te_{tt}|^2dx
+\int_{0}^T\si^4\int_{ } |\nabla \te_{tt}|^2 dxdt\le C.\ee

Finally, applying the standard $L^2$-estimate  to (\ref{eg1}), one obtains after using Lemmas \ref{le11}--\ref{pe1}, \ref{sq90}, (\ref{va1}), and (\ref{eg13}) that
\be\la{eg14}\ba
&\sup_{0\le t\le T}\si^2\|\na^2\te_t\|_{L^2}\\
&\le  C\sup_{0\le t\le T}\si^2\left(\|\n\te_{tt}\|_{L^2}
+  \|\n_t\|_{L^3}\|\te_t\|_{L^6}+\|\n_t\|_{L^6} \left(\|\na\te\|_{L^3}+\|\te\|_{L^6}\|\na u\|_{L^6}\right)\right)\\
& \quad +C\sup_{0\le t\le T}\si^2\left(\|\n^{1/2} \te_t\|_{L^2}+\|\na\te_t\|_{L^2}+ (1+\|\na^2\te\|_{L^2}) \|\na u_t\|_{L^2}+
 \|\na u_t\|_{L^6}\right)\\
& \le  C.\ea\ee
Moreover, it follows from  the standard $H^2$-estimate  to
$(\ref{3.29})$, (\ref{hs}), \eqref{nq1}, and Lemma \ref{le11}   that
\bnn\ba \|\na^2\te\|_{H^2}
&\le C\left(\|\n\te_t\|_{H^2}+\|\n u\cdot\na\te\|_{H^2}
+\|\n\te\div u\|_{H^2}+\||\na u|^2\|_{H^2}\right)\\
&\le C\left( \|\n\|_{H^2} \|\te_t\|_{H^2}
+\|\n\|_{H^2} \| u\|_{H^2}\|\na\te\|_{H^2}\right)\\
&\quad+C\|\n\|_{H^2} \|\te\|_{H^2} \| \div u\|_{H^2}+C\|\na u\|^2_{H^2}+C\\
&\le C\si^{-1}+ C\| \na^3\te \|_{L^2}+C\| \te_t\|_{H^2}.\ea\enn
Combining this with  (\ref{eg17}),  (\ref{eg14}), and  (\ref{eg13}) shows (\ref{egg17}).
The proof of Lemma \ref{sq91} is completed.
\end{proof}

\section{\la{se5}Proof of  Theorems  \ref{th1} and \ref{th2}}

With all the a priori estimates in Sections \ref{se3} and \ref{se4}
at hand, we are ready to prove the main results  of this paper in
this section.

\begin{pro} \la{pro2}

 For  given numbers $M>0$ (not necessarily small),
  $\on> 2,$ and $\bt>1,$   assume that  $(\rho_0,u_0,\te_0)$ satisfies (\ref{2.1}),  (\ref{3.1}),
and   (\ref{z01}). Then    there exists a unique classical solution  $(\rho,u,\te) $      of problem (\ref{a1})--(\ref{h1})
 in $\Omega\times (0,\infty)$ satisfying (\ref{mn5})--(\ref{mn2}) with $T_0$ replaced by any $T\in (0,\infty).$
  Moreover,  (\ref{zs2}), (\ref{a2.112}), (\ref{ae3.7}), and (\ref{vu15})  hold for any $T\in (0,\infty)$ and (\ref{h22}) holds for any $t\geq 1$.


 \end{pro}

\begin{proof}
First, by the standard local existence result (Lemma \ref{th0}), there exists a $T_0>0$ which may depend on
$\inf\limits_{x\in \Omega}\n_0(x), $  such that the  problem
 (\ref{a1})--(\ref{h1})  with   initial data $(\n_0 ,u_0,\te_0 )$
 has   a unique classical solution $(\n,u,\te)$ on $\O\times(0,T_0]$  satisfyinng (\ref{mn6})--(\ref{mn2}).
It follows from (\ref{As1})--(\ref{3.1}) and (\ref{z01}) that
\bnn A_1(0)\le M^2,\quad  A_2(0)\le  C_0^{1/4},\quad A_3(0)=0, \quad  \n_0<
 \hat{\rho},\quad \te_0\le \bt,\enn  which implies  there exists a
$T_1\in(0,T_0]$ such that (\ref{z1}) holds for $T=T_1.$
 We set \bnn \notag T^* =\sup\left\{T\,\left|\, \sup_{t\in [0,T]}\|(\n,u,\te)\|_{H^3}<\infty\right\},\right.\enn  and \be \la{s1}T_*=\sup\{T\le T^* \,|\,{\rm (\ref{z1}) \
holds}\}.\ee Then $ T^*\ge T_* \geq T_1>0.$
 Next, we claim that
 \be \la{s2}  T_*=\infty.\ee  Otherwise,    $T_*<\infty.$
Proposition \ref{pr1} shows   (\ref{zs2}) holds for all $0<T<T_*,$ which together with \eqref{z01} yields
  Lemmas \ref{le11}--\ref{sq91} still hold for all  $0< T< T_*.$
Note here that all  constants $C$  in  Lemmas \ref{le11}--\ref{sq91} 
depend  on $T_*  $ and $\inf\limits_{x\in \Omega}\n_0(x)$, and are in fact independent  of  $T$.
Then,  we claim that  there
exists a positive constant $\tilde{C}$ which may  depend  on $T_* $
and $\inf\limits_{x\in \Omega}\n_0(x)$   such that, for all  $0< T<
 T_*,$  \be\la{y12}\ba \sup_{0\le t\le T}
\| \n\|_{H^3}   \le \tilde{C},\ea \ee which together with Lemmas \ref{le11}, \ref{pe1},  \ref{sq90},  \eqref{mn2}, and (\ref{3.1}) gives
 \bnn
 \|(\n(x,T_*),u(x,T_*),\te(x,T_*))\|_{H^3}
 \le \tilde{C},\quad\inf_{x\in \Omega}\n(x,T_*)>0,\quad\inf_{x\in \Omega}\te(x,T_*)>0.\enn
  Thus, Lemma \ref{th0} implies that there exists some $T^{**}>T_*,$  such that
(\ref{z1}) holds for $T=T^{**},$   which contradicts (\ref{s1}).
Hence, (\ref{s2}) holds. This along with Lemmas \ref{th0}, \ref{a13.1}, \ref{le8},  and Proposition \ref{pr1},  thus
finishes the proof of   Proposition \ref{pro2}.

 Finally, it remains to prove (\ref{y12}). Using  $(\ref{a1})_3$ and (\ref{mn6}), we can define
 \bnn
 \theta_t(\cdot,0)\triangleq - u_0 \cdot\na \te_0 + \frac{\ga-1}{R} \rho_0^{-1}
\left(\ka\Delta\te_0-R\rho_0 \theta_0 \div u_0+\lambda (\div u_0)^2+2\mu |\mathfrak{D}(u_0)|^2\right),
 \enn
 which along with  (\ref{2.1}) gives
 \be \la{ssp91}\|\theta_t(\cdot,0)\|_{L^2}\le \tilde{C}.\ee
 Thus, one deduces from  (\ref{3.99}), \eqref{2.1}, \eqref{ssp91}, and Lemma \ref{le11} that
\be  \ba \la{a51}
\sup\limits_{0\le t\le T} \int \n|\dot\te|^2dx+\int_0^T \|\na\dot\te\|_{L^2}^2dt \le \tilde{C},
\ea\ee
which together with  (\ref{lop4}) and Lemma \ref{le11}  yields
\be\la{sp211}
\sup\limits_{0\le t\le T}\|\na^2 \theta\|_{L^2} \le \tilde{C}.\ee
Using  $(\ref{a1})_2$ and  (\ref{mn6}), we can define
\bnn
u_t(\cdot,0) \triangleq -u_0\cdot\na u_0+\n_0^{-1}\left( \mu \Delta u_0 + (\mu+\lambda) \na \div u_0 - R\na (\n_0\te_0)\right),
\enn
which along with  (\ref{2.1}) gives
\be \la{ssp9}\|\na u_t(\cdot,0)\|_{L^2}\le \tilde{C}.\ee
Thus, it follows from Lemmas \ref{le11}, \ref{le9-1},  (\ref{4.052}),  (\ref{s4}), (\ref{a51})--(\ref{ssp9}), and Gr\"{o}nwall's inequality that
\be \la{ssp1}
\sup_{0\le t\le T}\|\na u_t\|_{L^2}+\int_0^T\int \n |u_{tt}|^2dxdt\le \tilde{C},
\ee
which as well as (\ref{va2}), \eqref{sp211}, and (\ref{qq1}) yields
\be\la{sp221} \sup\limits_{0\le t\le T}\|u\|_{H^3} \le \tilde{C}.\ee
Combining this with Lemma \ref{le11}, \eqref{ex4}, \eqref{nt4}, (\ref{a51}), (\ref{sp211}), (\ref{ssp1}),   and  (\ref{sp221}) gives
\be\la{ssp24} \ia\left(\|\na^3\te\|_{L^2}^2+ \|\nabla u_t\|_{H^1}^2\right)dt\le \tilde{C} . \ee
Then, applying (\ref{rmk1}), \eqref{hs}, \eqref{sp211}, \eqref{ssp1}, \eqref{sp221}, and Lemma \ref{le11}, one has
\be \notag\ba \|\na^2 u\|_{H^2}
&\le\tilde{C}\left(\| \n \dot u \|_{H^2}+\|\na  P\|_{H^2} + \|\na u\|_{L^2}\right)\\
&\le \tilde{C}\left(\| \n \|_{H^2}\|  u_t \|_{H^2}+\| \n \|_{H^2}\|  u \|_{H^2}\|  \na u \|_{H^2}\right)\\
&\quad+\tilde C\left(\|\na\n \|_{H^2}\|\te\|_{H^2}+\|\n\|_{H^2}\|\na  \te\|_{H^2}+1\right)\\
& \le \tilde{C} (1+ \|\na^2  u_t\|_{L^2}+\|\na^3 \n \|_{L^2}+\|\na^3 \te \|_{L^2}),\ea\ee
which along with some standard calculations leads to
 \bnn\la{sp134}\ba  & \left(\|\na^3 \n\|_{L^2} \right)_t \\
&\le \tilde{C}\left(\| |\na^3u| |\na \n| \|_{L^2}+ \||\na^2u||\na^2
      \n|\|_{L^2}+ \||\na u||\na^3 \n|\|_{L^2} +\| \na^4u \|_{L^2} \right)\\
&\le \tilde{C}\left(\| \na^3 u\|_{L^2}\|\na \n \|_{H^2}+ \| \na^2u\|_{L^3}\|\na^2 \n \|_{L^6}
       +\|\na u\|_{L^\infty}\|\na^3 \n\|_{L^2} +\| \na^4u \|_{L^2}\right) \\
&\le \tilde{C}(1+ \| \na^3\n \|_{L^2}+ \| \na^2 u_t\|^2_{L^2}+ \|\na^3\te\|^2_{L^2}), \ea\enn where we have used (\ref{sp221}) and Lemma \ref{le11}.
 Combining this with (\ref{ssp24})  and
Gr\"{o}nwall's inequality  yields   \bnn\la{sp26} \sup\limits_{0\le t\le
T}\|\nabla^3  \n\|_{L^2} \le \tilde{C},\enn which together with
(\ref{qq1}) gives (\ref{y12}).
The proof of Proposition \ref{pro2} is completed.
\end{proof}

With  Proposition \ref{pro2} at hand, we are now in a position to prove  Theorem \ref{th1}.

\begin{proof}[Proof of  Theorem   \ref{th1}]
 Let $(\n_0,u_0,\te_0)$  satisfying (\ref{co3})--(\ref{co2}) be the initial data in Theorem \ref{th1}.  Assume that  $C_0$  satisfies (\ref{co14}) with
\be\la{xia}\ve\triangleq \ve_0/2,\ee
where  $\ve_0$  is given in Proposition \ref{pr1}.

First, we construct the approximate initial data $(\n_0^{m,\eta},u_0^{m,\eta}, \te_0^{m,\eta})$ as follows. For constants
\be\la{5d0}
m \in \mathbb{Z}^+,\ \  \eta \in \left(0, \eta_0 \right),\ \  \eta_0\triangleq \min\xl\{1,\frac{1}{2}(\on-\sup\limits_{x\in \O}\n_0(x)) \xr\},
\ee
we define
\begin{align*}
\n_0^{m,\eta} = \n_0^{m}+ \eta,\ \  u_0^{m,\eta}=\frac{u_0^m }{1+\eta},\ \  \te_0^{m,\eta}= \frac{\te_0^{m} + \eta}{1+2\eta},
\end{align*}
where $\n_0^{m}$ satisfies
\begin{align*}
0 \le \n_0^{m} \in C^{\infty},\ \  \lim_{m \to \infty} \|\n_0^{m} -\rho_0\|_{W^{2,q}}=0,
\end{align*}
 $u_0^m$ is the unique smooth solution to the following elliptic equation:
\be\notag\begin{cases}
	\Delta u_0^m=\Delta \tilde{u}_0^m,&\text{in}\,\, \O,\\
	u_0^m\cdot n=0 ,\,\,\curl u_0^m\times n=0,&\text{on}\,\,\p\O,
\end{cases}\ee
with $\tilde{u}_0^m \in C^{\infty}$ satisfying $\lim_{m \to \infty}\| \tilde{u}_0^m -{u}_0\|_{H^2}=0$, and $\te_0^m$ satisfying $\int_\O\te_0^m dx=\int_\O\te_0 dx$ is the unique smooth solution to the following Poisson equation:
\be\notag\begin{cases}
	\Delta \te_0^m=\Delta \tilde{\te}_0^m- \overline{\Delta \tilde{\te}_0^m},&\text{in}\,\,\O,\\
	\na \te_0^m\cdot n=0 ,&\text{on}\,\,\p\O,
\end{cases}\ee
with $0\le \tilde{\te}_0^m \in C^{\infty}$ satisfying $\lim\limits_{m \to \infty}\| \tilde{\te}_0^m -{\te}_0\|_{H^2}=0$.

Then for any $\eta\in (0, \eta_0)$, there exists $m_1(\eta)\ge 1$ such that for $m \ge m_1(\eta)$, the approximate initial data
$(\n_0^{m,\eta},u_0^{m,\eta}, \te_0^{m,\eta})$ satisfies
\be \la{de3}\begin{cases}(\n_0^{m,\eta},u_0^{m,\eta}, \te_0^{m,\eta})\in C^\infty ,\\
	\dis \eta\le  \n_0^{m,\eta}  <\hat\n,~~\, \frac{\eta}{4}\le \te_0^{m,\eta} \le \hat \te,~~\,\|\na u_0^{m,\eta}\|_{L^2} \le M, \\
	\dis u_0^{m,\eta}\cdot n=0,~~\,\curl u_0^{m,\eta}\times n=0,~~\,\na \te_0^{m,\eta}\cdot n=0\,\, \text{on}\,\p\O,\end{cases}
\ee
and
 \be \la{de03}
\lim\limits_{\eta\rightarrow 0} \lim\limits_{m\rightarrow \infty}
\left(\| \n_0^{m,\eta} - \n_0 \|_ {W^{2,q}}+\| u_0^{m,\eta}-u_0\|_{H^2}+\| \te_0^{m,\eta}- \te_0  \|_{H^2}\right)=0.
\ee
Moreover,  the initial norm $C_0^{m,\eta}$
for $(\n_0^{m,\eta},u_0^{m,\eta}, \te_0^{m,\eta}),$ which is defined by  the right-hand side of (\ref{e})
with $(\n_0,u_0,\te_0)$   replaced by
$(\n_0^{m,\eta},u_0^{m,\eta}, \te_0^{m,\eta}),$
satisfies \bnn \lim\limits_{\eta\rightarrow 0} \lim\limits_{m\rightarrow \infty} C_0^{m,\eta}=C_0.\enn
Therefore, there exists  an  $\eta_1\in(0, \eta_0) $
such that, for any $\eta\in(0,\eta_1),$ we can find some $m_2(\eta)\geq m_1(\eta)$  such that   \be \la{de1} C_0^{m,\eta}\le C_0+\ve_0/2\le  \ve_0 , \ee
provided that\be  \la{de7}0<\eta<\eta_1 ,\,\, m\geq m_2(\eta).\ee

  We assume that $m,\eta$ satisfy (\ref{de7}).
  Proposition \ref{pro2} together with (\ref{de1}) and (\ref{de3}) thus yields that
   there exists a smooth solution  $(\n^{m,\eta},u^{m,\eta}, \te^{m,\eta}) $
   of problem (\ref{a1})--(\ref{h1}) with  initial data $(\n_0^{m,\eta},u_0^{m,\eta}, \te_0^{m,\eta})$
   on $\Omega\times (0,T] $ for all $T>0. $
    Moreover, one has (\ref{h8}), \eqref{zs2}, \eqref{h22}, (\ref{a2.112}), \eqref{ae3.7}, and (\ref{vu15})   with $(\n,u,\te)$  being replaced by $(\n^{m,\eta},u^{m,\eta}, \te^{m,\eta}).$

 Next, for the initial data $(\n_0^{m,\eta},u_0^{m,\eta}, \te_0^{m,\eta})$, the function $\tilde g$ in (\ref{co12})  is
 \be \la{co5}\ba \tilde g & \triangleq(\n_0^{m,\eta})^{-1/2}\left(-\mu \Delta u_0^{m,\eta}-(\mu+\lambda)\na\div
 u_0^{m,\eta}+R\na (\n_0^{m,\eta}\te^{m,\eta}_0)\right)\\
& = (\n_0^{m,\eta})^{-1/2}\sqrt{\n_0}g+\mu(\n_0^{m,\eta})^{-1/2}\Delta(u_0-u_0^{m,\eta})\\
&\quad+(\mu+\lambda) (\n_0^{m,\eta})^{-1/2} \na \div(u_0-u_0^{m,\eta})+ R(\n_0^{m,\eta})^{-1/2} \na(\n_0^{m,\eta}\te_0^{m,\eta}-\n_0\te_0),\ea\ee
where in the second equality we have used (\ref{co2}).
Since $g \in L^2,$ one deduces from (\ref{co5}),  (\ref{de3}), (\ref{de03}), and  (\ref{co3})  that for any $\eta\in(0,\eta_1),$ there exist some $m_3(\eta)\geq m_2(\eta)$ and a positive constant $C$ independent of $m$ and $\eta$ such that
 \be\la{de4}
 	\|\tilde g\|_{L^2}\le \|g\|_{L^2}+C\eta^{-1/2}\de(m) + C\eta^{1/2},
 	\ee
with   $0\le\de(m) \rightarrow 0$ as
$m \rightarrow \infty.$ Hence,  for any  $\eta\in(0,\eta_1),$ there exists some $m_4(\eta)\geq m_3(\eta)$ such that for any $ m\geq m_4(\eta)$,
 \be \la{de9}\de(m) <\eta.\ee  We thus obtain from (\ref{de4}) and (\ref{de9}) that
there exists some positive constant $C$ independent of $m$ and $\eta$ such that  \be\la{de14}  \|\tilde g \|_{L^2}\le \|g \|_{L^2}+C,\ee provided that\be \la{de10} 0<\eta<\eta_1,\,\,  m\geq m_4(\eta).\ee

 Now, we   assume that $m,$  $\eta$ satisfy (\ref{de10}).
 It thus follows from (\ref{de3})--(\ref{de1}),  (\ref{de14}), Proposition \ref{pr1}, 
 and Lemmas \ref{le8}, \ref{le11}--\ref{sq91} that for any $T>0,$
 there exists some positive constant $C$ independent of $m$ and $\eta$ such that
 (\ref{h8}), (\ref{zs2}),   (\ref{a2.112}),  \eqref{ae3.7}, (\ref{vu15}), \eqref{lee2}, \eqref{qq1},  (\ref{va5}),  (\ref{vva5}), (\ref{y2}),  (\ref{eg17}),
  and  (\ref{egg17})  hold for  $(\n^{m,\eta},u^{m,\eta}, \te^{m,\eta}) .$
   Then passing  to the limit first $m\rightarrow \infty,$ then $\eta\rightarrow 0,$
   together with standard arguments yields that there exists a solution $(\n,u,\te)$ of the problem (\ref{a1})--(\ref{h1})
   on $\Omega\times (0,T]$ for all $T>0$, such that the solution  $(\n,u,\te)$
   satisfies  (\ref{h8}), (\ref{a2.112}),    \eqref{ae3.7},  (\ref{vu15}),  \eqref{lee2}, \eqref{qq1}, (\ref{va5}),  (\ref{vva5}), (\ref{y2}),  (\ref{eg17}), (\ref{egg17}),
   and  the estimates of $A_i(T)\,(i=1,2,3)$ in
   (\ref{zs2}). Hence,    $(\n,u,\te)$ satisfying  (\ref{h8}) and \eqref{h9} refers to \cite{H-L} for the detailed proof. Moreover, one deduces from Proposition \ref{pr1} that the desired exponential decay property \eqref{h11}.

Finally,  the proof of the uniqueness of $(\n,u,\te)$ is similar to that  of \cite[Theorem 1]{choe1} and will be omitted here for simplicity. The proof of Theorem \ref{th1} is completed.
\end{proof}

{\it Proof of  Theorem   \ref{th2}. } We will prove  Theorem   \ref{th2} in two steps.

 {\it Step 1. Construction  of approximate  solutions.} Assume $(\n_0,u_0,\te_0)$ satisfying (\ref{co4}) and \eqref{dt7} is the initial data in Theorem \ref{th2} and  $C_0$ satisfies (\ref{co14})   with  $\ve $  as in  (\ref{xia}).   For $j_{m^{-1}}(x)$ being the standard mollifying kernel of width $m^{-1}$, we construct
\begin{align*}
\hat\n_0^{m,\eta} = (\n_0 1_\O)\ast j_{m^{-1}}1_\O+ \eta,\ \  \hat u_0^{m,\eta}=\frac{u_0^m }{1+\eta},\ \  \hat\te_0^{m,\eta}= \frac{(\n_0\te_0 1_{\O_m})\ast j_{m^{-1}} + \eta}{(\n_01_{\O_m})\ast j_{m^{-1}}+ \eta},
\end{align*}
where  $\O_m=\{x\in\O| dist(x,\p\O)>2/m\}$ and $u_0^{m}$ satisfies
\be\notag u_0^m\in C^\infty\cap H^1_\omega\ \ \text{and}\ \ \lim_{m\rightarrow \infty}\|u_0^{m}-u_0\|_{H^1}=0.\ee
 Then for any $\eta\in (0, \eta_0)$ with $\eta_0$  as in (\ref{5d0}), there exists $m(\eta)>1$ such that for $m \ge m(\eta)$, the approximate initial data
$(\hat\n_0^{m,\eta},\hat u_0^{m,\eta}, \hat\te_0^{m,\eta})$ satisfies
\be \la{dee3}\begin{cases}(\hat\n_0^{m,\eta},\hat u_0^{m,\eta}, \hat\te_0^{m,\eta})\in C^\infty ,\\
	\dis \eta\le \hat \n_0^{m,\eta}  <\hat\n,~~\, \frac{\eta}{\hat\n+\eta}\le \hat\te_0^{m,\eta} \le \hat \te,~~\,\|\na \hat u_0^{m,\eta}\|_{L^2} \le M, \\
\dis \hat u_0^{m,\eta}\cdot n=0,~~\,\curl \hat u_0^{m,\eta}\times n=0,~~\,\na \hat \te_0^{m,\eta}\cdot n=0,\,\,\,\, \text{on}\,\p\O,\end{cases}
\ee
and for any $p\geq1$,
 \be \la{dee03}
\lim\limits_{\eta\rightarrow 0} \lim\limits_{m\rightarrow \infty}
\left(\| \hat\n_0^{m,\eta} - \n_0 \|_ {L^p}+\|\hat u_0^{m,\eta}-u_0\|_{H^1}+\|\hat\n_0^{m,\eta}\hat\te_0^{m,\eta}-\n_0\te_0  \|_{L^2}\right)=0
\ee
owing to (\ref{co4})  and (\ref{co14}).

Now, we claim that  the initial norm $\hat C_0^{m,\eta}$
for $(\hat\n_0^{m,\eta},\hat u_0^{m,\eta},\hat\te_0^{m,\eta}),$  i.e., the right hand side of
(\ref{e}) with $(\n_0,u_0,\te_0)$  replaced by
$(\hat\n_0^{m,\eta},\hat u_0^{m,\eta},\hat\te_0^{m,\eta}),$   satisfies
\be \la{uv9}  \lim\limits_{\eta\rightarrow 0}
\lim\limits_{m\rightarrow \infty}\hat C_0^{m,\eta}\le C_0,\ee
which leads to that there exists an $\hat\eta\in(0,\eta_0)$ such that, for any $\eta\in
(0,\hat\eta ),$ there exists some $\hat m (\eta)\geq m(\eta)$ such that
\be \la{uv8}\hat C_0^{m,\eta}\le C_0+\ve_0/2\le \ve_0 , \ee
provided \be\la{uv01}
0<\eta<\hat\eta  , \quad m\geq\hat m (\eta).\ee
Then if we assume (\ref{uv01}) holds, it directly follows from Proposition \ref{pro2}, (\ref{dee3}) and (\ref{uv8}) that there exists a classical solution  $(\hat\n^{m,\eta},\hat u^{m,\eta},\hat\te^{m,\eta})  $  of problem (\ref{a1})--(\ref{h1}) with  initial data $(\hat\n_0^{m,\eta},\hat u_0^{m,\eta},\hat\te_0^{m,\eta})$   on $\O\times(0,T]$ for all $T>0$.  Furthermore, $(\hat\n^{m,\eta},\hat u^{m,\eta},\hat\te^{m,\eta}) $  satisfies \eqref{h8}, (\ref{zs2}), (\ref{a2.112}), \eqref{key}, (\ref{ae3.7}),  (\ref{vu15}), and \eqref{h22} respectively  for any $T>0$ and $t\geq1$ with $(\n ,u ,\te )$   replaced by $(\hat\n^{m,\eta},\hat u^{m,\eta},\hat\te^{m,\eta})$.

It remains to prove (\ref{uv9}). Indeed, we just need to infer
\be\la{uv10}\lim_{\eta\rightarrow 0}\lim_{m\rightarrow \infty}\int\hat\n_0^{m,\eta}\left(\hat\te_0^{m,\eta}- \log
 \hat\te_0^{m,\eta} -1 \right)dx\le \int\n_0\left(\te_0- \log
 \te_0 -1 \right)dx ,\ee since   the other terms in  (\ref{uv9}) can be proved in a  similar and even simpler way.
Note that
\bnn\ba
&\hat\n_0^{m,\eta}\left(\hat\te_0^{m,\eta}- \log \hat\te_0^{m,\eta} -1\right)\\
&=\hat\n_0^{m,\eta}(\hat\te_0^{m,\eta}-1)^2 \int_0^1\frac{\al}{\al(\hat\te_0^{m,\eta}-1)+1}d\al\\
&= \frac{\hat\n_0^{m,\eta}(j_{m^{-1}}*(\n_0(\te_0-1)1_{\O_m}))^2}{j_{m^{-1}}* (\n_01_{\O_m} )+\eta}\\
&\quad\cdot\int_0^1\frac{\al}{\al j_{m^{-1}}*(\n_0(\te_0-1)1_{\O_m})+j_{m^{-1}}* (\n_01_{\O_m} )+\eta}d\al \\
&\in \left[0, \, \hat\n\eta^{-2}(j_{m^{-1}}*(\n_0(\te_0-1)1_{\O_m}))^2 \right],
\ea\enn
which combined with  Lebesgue's dominated convergence theorem yields that
\be \notag\ba
&\lim_{m\rightarrow \infty}\int\hat\n_0^{m,\eta}\left(\hat\te_0^{m,\eta}- \log
 \hat\te_0^{m,\eta} -1 \right)dx\\
&=\int(\n_0+\eta)\left(\frac{\n_0\te_0+\eta}{\n_0+\eta}
 - \log\frac{\n_0\te_0+\eta}{\n_0+\eta}-1 \right)dx\\
&=\int\left(\n_0\te_0 -\n_0+(\n_0+\eta)\log(\n_0+\eta)\right)dx\\
 &\quad-\int\left(\n_0\log(\n_0\te_0+\eta)+\eta\log(\n_0\te_0+\eta)\right)dx\\
 &\le \int\left(\n_0\te_0 -\n_0+(\n_0+\eta)\log(\n_0+\eta)\right)dx-\int\left(\n_0\log(\n_0\te_0)+\eta\log\eta\right)dx\\
 & \rightarrow \int \n_0\left(\te_0-  \log\te_0 -1 \right)dx,\quad \mbox{ as }\eta\rightarrow 0.
\ea\ee
It thus gives \eqref{uv10}.

{\it Step 2. Compactness results.}
 With the approximate solutions $(\hat\n^{m,\eta},\hat u^{m,\eta},\hat\te^{m,\eta}) $ obtained in the previous step at hand,  we can derive the global existence of weak solutions by passing to the limit  first   $m\rightarrow \infty,$ then $\eta\rightarrow 0 .$   Since the two steps are similar, we will only   sketch the  arguments for $m\rightarrow \infty.$  For any fixed  $\eta\in (0,\hat\eta)$, we simply denote  $(\hat\n^{m,\eta},\hat u^{m,\eta},\hat\te^{m,\eta}) $
   by $(\n^m ,u^m ,\te^m ).$ Then  the combination of Aubin-Lions Lemma with (\ref{zs2}), (\ref{a2.112}),  \eqref{key},  (\ref{vu15}),  and Lemma \ref{le4} yields that there exists some appropriate subsequence $ m_j \rightarrow \infty$ of $m\rightarrow \infty$ such that, for any $0<\tau<T<\infty $, $p\in[1,\infty)$, and $\tilde p\in[1,6)$,
\be \la{vu4}
u^{m_j}\rightharpoonup u  \,\,\mbox{ weakly star in }\,\, L^\infty(0,T; H^1),
\ee
\be\la{lk}
\te^{m_j}\rightharpoonup \te \,\,\mbox{ weakly in }\,\, L^2(0,T;H^1),
\ee
\be \la{vu1}
\n^{m_j}\rightarrow \n  \quad \mbox{ in }\,\, C([0,T];L^p\mbox{-weak})\cap C([0,T];H^{-1}),\ee
\be \la{vu5}
\n^{m_j}u^{m_j}\rightarrow \n u,\, \n^{m_j}\te^{m_j}
\rightarrow \n \te   \, \mbox{ in }\,  C([0,T];L^2 \mbox{-weak})\cap C([0,T];H^{-1}),
\ee
\be \la{vu26} \n^{m_j} |u^{m_j}|^2\rightarrow \n |u|^2 \,\, \mbox{ in }\,\, C([0,T];L^3 \mbox{-weak})\cap C([0,T];H^{-1}),\ee
\be\la{ghh}
G^{m_j}\rightarrow G,\,\curl u^{m_j}\rightarrow\curl u\,\,\mbox{ in }\,\,C([\tau,T];H^1\mbox{-weak})\cap C([\tau,T];L^{\tilde p}),
\ee
\be\la{lll}u^{m_j}\rightarrow u\,\mbox{ in }\,\,C([\tau,T];W^{1,6}\mbox{-weak})\cap C(\O\times[\tau,T]),\ee
and
\be \la{vu18}
\te^{m_j}\rightarrow \te\,\,\mbox{in}\,\,C([\tau,T];H^2\mbox{-weak})\cap C([\tau,T];W^{1,\tilde p}),
\ee
referring to \cite{H-L} for the detailed proof.
Now we consider the approximate solutions $(\n^{m_j},u^{m_j},\te^{m_j})$ in the weak forms, i.e. \eqref{def1}--\eqref{def3}, then  take appropriate limits. Standard arguments as well as (\ref{dee03}) and (\ref{vu4})--(\ref{vu18}) thus conclude that the limit $(\n,u,\te)$ is a weak solution  of   (\ref{a0})   (\ref{h2})  (\ref{h1})  in the sense of Definition \ref{def} and satisfies (\ref{hq1})--(\ref{hq4}) and the exponential decay property \eqref{h11}. 
Moreover, we obtain the estimates (\ref{hq5})--(\ref{hq8}) with the aid of (\ref{zs2}), (\ref{a2.112}), (\ref{vu15}),    and (\ref{vu4})--(\ref{vu18}). Finally, (\ref{vu019}) shall be obtained by adopting the same way as in \cite{H-L}.
The proof of Theorem \ref{th2} is finished.

\section*{Acknowledgements}   The research  is
partially supported by the National Center for Mathematics and Interdisciplinary Sciences, CAS,
NSFC Grant (Nos.  11688101,  12071200, 11971217),   Double-Thousand Plan of Jiangxi Province (No. jxsq2019101008),  Academic and Technical Leaders Training Plan of Jiangxi Province (No. 20212BCJ23027), and Natural Science
Foundation of Jiangxi Province (No.  20202ACBL211002).

\end{document}